\documentclass[final,onefignum,onetabnum]{siamonline190516}
\usepackage[scale=0.8]{geometry} 

\usepackage{lipsum}
\usepackage{amsfonts}
\usepackage{epstopdf}
\ifpdf
  \DeclareGraphicsExtensions{.eps,.pdf,.png,.jpg}
\else
  \DeclareGraphicsExtensions{.eps}
\fi

\usepackage{datetime}
\newdateformat{monthyeardate}{%
  \monthname[\THEMONTH] \THEDAY, \THEYEAR}

\usepackage{academicons}
\usepackage{xcolor}
\renewcommand{\orcid}[1]{\href{https://orcid.org/#1}{\textcolor[HTML]{A6CE39}{orcid.org/#1}}}

\usepackage{amsmath}
\allowdisplaybreaks
\usepackage{amssymb}
\usepackage{commath}
\usepackage{mathtools}
\usepackage{bbm}
\usepackage{bm}

\usepackage{color}
\usepackage{graphicx}
\usepackage[small]{caption}
\usepackage{subcaption}

\usepackage{algorithm}
\usepackage{algorithmic}

\usepackage{relsize}
\usepackage{adjustbox}
\usepackage{booktabs}
\usepackage{tikz}
\usepackage{pifont}

\usepackage{enumitem}
\setlist[enumerate]{leftmargin=.5in}
\setlist[itemize]{leftmargin=.5in}

\newsiamremark{remark}{Remark}
\newsiamremark{example}{Example}
\newsiamremark{hypothesis}{Hypothesis}
\crefname{hypothesis}{Hypothesis}{Hypotheses}
\newsiamthm{claim}{Claim}

\headers{Towards stability results for global RBF-QFs}{J.\ Glaubitz and J.\,A.\ Reeger}

\title{Towards stability results for global radial basis function based quadrature formulas\thanks{\monthyeardate\today \corresponding{Jan Glaubitz (\email{Jan.Glaubitz@Dartmouth.edu}, \orcid{0000-0002-3434-5563})}
\disclaimer{The views expressed in this academic research paper are those of the authors and do not reflect the official policy or position of the United States Government or Department of Defense.  In accordance with the Air Force Instruction 51-303, it is not copyrighted, but is the property of the United States government.}
}}

\author{Jan Glaubitz\thanks{Department of Mathematics, Dartmouth College, Hanover, NH 03755, USA}
\and
Jonah Reeger\thanks{Air Force Institute of Technology, Wright--Patterson Air Force Base, OH 45433, USA.}
}

\usepackage{amsopn}

\usepackage[normalem]{ulem}

\DeclareMathOperator{\diag}{diag}

\DeclareMathOperator*{\argmin}{arg\,min}

\DeclareMathOperator{\arcsinh}{arcsinh}
\newcommand{\scp}[2]{\left\langle{#1,\, #2}\right\rangle}

\newcommand{\intd}{\, \mathrm{d}}
\newcommand{\N}{\mathbb{N}}

\newcommand{\R}{\mathbb{R}}

\usepackage{lineno}

\ifpdf
\hypersetup{
  pdftitle={Glaubitz2021Stability},
  pdfauthor={Jan Glaubitz}
}
\fi

\begin{document}

\maketitle

\begin{abstract}
Quadrature formulas (QFs) based on radial basis functions (RBFs) have become an essential tool for multivariate numerical integration of scattered data. 
Although numerous works have been published on RBF-QFs, their stability theory can still be considered as underdeveloped. 
Here, we strive to pave the way towards a more mature stability theory for global and function-independent RBF-QFs. 
In particular, we prove stability of these for compactly supported RBFs under certain conditions on the shape parameter and the data points. 
As an alternative to changing the shape parameter, we demonstrate how the least-squares approach can be used to construct stable RBF-QFs by allowing the number of data points used for numerical integration to be larger than the number of centers used to generate the RBF approximation space. 
Moreover, it is shown that asymptotic stability of many global RBF-QFs is independent of polynomial terms, which are often included in RBF approximations. 
While our findings provide some novel conditions for stability of global RBF-QFs, the present work also demonstrates that there are still many gaps to fill in future investigations. 
\end{abstract}



\section{Introduction} 
\label{sec:introduction} 

Numerical integration is an omnipresent task in mathematics and myriad applications. 
While these are too numerous to list fully, prominent examples include numerical differential equations \cite{hesthaven2007nodal,quarteroni2008numerical,ames2014numerical}, machine learning \cite{murphy2012machine}, finance \cite{glasserman2013monte}, and biology \cite{manly2006randomization}. 
In many cases, the problem can be formulated as follows.  
Let $\Omega \subset \R^D$ be a bounded domain with positive volume, $|\Omega| > 0$.
Given $N$ distinct data pairs $\{ (\mathbf{x}_n,f_n) \}_{n=1}^N \subset \Omega \times \R$ with $f: \Omega \to \R$ and $f_n := f(\mathbf{x}_n)$, the aim is to approximate the weighted integral 
\begin{equation*}\label{eq:I}
  I[f] := \int_\Omega f(\boldsymbol{x}) \omega(\boldsymbol{x}) \intd \boldsymbol{x}
\end{equation*}
by an \emph{$N$-point QF}. 
That is, by a weighted finite sum over the given function of the form 
\begin{equation*}\label{eq:CR}
  C_N[f] = \sum_{n=1}^N w_n f(\mathbf{x}_n).
\end{equation*} 
In higher dimensions, $C_N$ is sometimes referred to as an $N$-point cubature formula. 
The distinct points $\{ \mathbf{x}_n \}_{n=1}^N$ are called \emph{data points} and the $\{ w_n \}_{n=1}^N$ are referred to as \emph{quadrature weights}. 
Many QFs are derived based on the idea to approximate the (unknown) function $f$ and then exactly integrate this approximation \cite{haber1970numerical,stroud1971approximate,engels1980numerical,cools1997constructing,krommer1998computational,cools2003encyclopaedia,krylov2006approximate,davis2007methods,brass2011quadrature,trefethen2017cubature}. 
Arguably, most of the existing QFs have been derived from being exact for polynomials up to a certain degree. 
See \cite{maxwell1877approximate,mysovskikh1980approximation,cools2001cubature,mysovskikh2001cubature,cools2003encyclopaedia,trefethen2021exactness}, in addition to the above references.  

That said, in recent years, QFs based on the exact integration of RBFs have received a growing amount of interest \cite{sommariva2005integration,sommariva2006numerical,sommariva2006meshless,punzi2008meshless,aziz2012numerical,fuselier2014kernel,reeger2016numericalA,reeger2016numericalB,watts2016radial,reeger2018numerical,sommariva2021rbf}. 
The increased use of RBFs for numerical integration and numerical differential equations \cite{kansa1990multiquadrics,iske1996structure,fasshauer1996solving,kansa2000circumventing,larsson2003numerical,iske2003radial,shu2007integrated,fornberg2015solving,flyer2016enhancing,glaubitz2021stabilizing,glaubitz2021towards} seems to be only logical, considering their story of success in the last few decades. 
In fact, since their introduction in Hardy’s work on cartography from 1971 (see \cite{hardy1971multiquadric}), RBFs have become a powerful tool in numerical analysis, including multivariate interpolation and approximation theory  \cite{buhmann2000radial,buhmann2003radial,wendland2004scattered,fasshauer2007meshfree,iske2011scattered,fornberg2015primer}. 
It should also be mentioned that RBF-QF can be connected to (statistical) Bayesian quadrature \cite{o1991bayes,minka2000deriving,briol2019probabilistic,karvonen2019positivity}.
Finally, recent literature in the quadrature area \cite{reeger2016numericalA,reeger2016numericalB,reeger2018numerical,reeger2020approximate} has focused on 'local' RBF-FD-type implementations to reduce computational costs for large node numbers. 
While an extension to such local approaches would be of interest, we restrict ourselves to global RBF methods in this work. 
That said, to reduce the cost of constructing and integrating a global interpolant, a piecewise RBF interpolant could be considered and integrated in a manner similar to the construction of Newton--Cotes formulas. 
Some of our results would easily carry over to this setting, which might be seen as an extreme version (no overlap of nonzero measure) of RBF partition of unity methods \cite{babuvska1997partition,wendland2002fast,fasshauer2007meshfree} or the overlapped RBF-FD methods \cite{shankar2017overlapped}. 

Even though RBF-QFs have been proposed and applied in numerous works, their stability theory can still be considered as under-developed,
especially compared to more traditional---e.\,g\ polynomial based---methods. 
Stability of RBF-QFs was broached, for instance, in \cite{sommariva2005integration,sommariva2006numerical,punzi2008meshless}. 
Further, stability of RBF-QF was discussed in \cite{fuselier2014kernel} for integration on certain manifolds. 
However, to the best of our knowledge, an exhaustive stability theory for RBF-QFs is still missing in the literature. 
In particular, theoretical results providing clear conditions under which stability of RBF-QFs is ensured are rarely encountered, even for global RBF methods. 
%

The present work strives to fill this gap in the RBF literature partially. 
This is done by providing a detailed theoretical and numerical investigation on stability of global RBF-QFs\footnote{Henceforth, we will refer to these as ``RBF-QFs".} for different families of kernels,
including compactly supported and Gaussian RBFs as well as polyharmonic splines (PHS). 
Our analysis resembles classic stability theory for quadratures exact for polynomial spaces. 
In contrast to some existing works (see \cite{cavoretto2022rbfcub} and references therein), we consider RBF approximations with function-independent shape parameters to obtain quadrature formulas that do not have to be recomputed when another function is considered. 

In particular, we report on the following findings. 
(1) We provide a sufficient condition for compactly supported RBFs to yield a provable stable RBF-QF (see Theorem \ref{thm:main} in \S\ref{sec:compact}). 
The result is independent of the degree of the polynomial term that is included in the global RBF interpolant and assumes the data points to come from an equidistributed (space-filling) sequence. 
(2) We demonstrate how the idea of least squares can be employed to construct provable stable RBF-QFs. 
(3) Asymptotic stability of pure RBF-QFs is connected to asymptotic stability of the same RBF-QF but augmented with polynomials of a fixed arbitrary degree. 
Essentially, we can show that for a sufficiently large number of data points, stability of RBF-QFs is independent of the presence of polynomials in the RBF interpolant.  


The rest of this work is organized as follows. 
We collect some preliminaries on RBF interpolants and QFs in \S\ref{sec:prelim}. 
In \S\ref{sec:stability}, a few initial comments on stability of (RBF-)QFs are offered. 
Next, \S\ref{sec:compact} contains our main theoretical result regarding stability of RBF-QFs based on compactly supported kernels. 
\S\ref{sub:LS} demonstrates how the concept of least squares can be used to construct provable stable RBF-QFs. 
Furthermore, it is proven in \S\ref{sec:connection} that, under certain assumptions, asymptotic stability of RBF-QFs is independent of the polynomial terms included in the RBF interpolant. 
Numerical tests in \S\ref{sec:numerical} accompany the previous theoretical findings. 
Finally, concluding thoughts are offered in \S\ref{sec:summary}.
\section{Preliminaries} 
\label{sec:prelim} 

We collect some preliminaries on RBF interpolants (\S\ref{sub:prelim_RBFs}) and RBF-QFs (\S\ref{sub:prelim_CFs}).

\subsection{Radial basis function interpolation}
\label{sub:prelim_RBFs} 

RBFs are often considered a powerful tool in numerical analysis, including multivariate interpolation and approximation theory \cite{buhmann2000radial,buhmann2003radial,wendland2004scattered,fasshauer2007meshfree,iske2011scattered,fornberg2015primer}.
We are especially interested in RBF interpolants. 
Let $f: \R^D \supset \Omega \to \R$ be a scalar valued function. 
Given a set of distinct \emph{data points} (sometimes also referred to as \emph{centers}), the \emph{RBF interpolant} of $f$ is of the form 
\begin{equation}\label{eq:RBF-interpol}
  (s_{N,d}f)(\boldsymbol{x}) 
    = \sum_{n=1}^N \alpha_n \varphi( \varepsilon_{n} \| \boldsymbol{x} - \mathbf{x}_n \|_2 ) + \sum_{k=1}^K \beta_k p_k(\boldsymbol{x}).
\end{equation} 
Here, $\varphi: \R_0^+ \to \R$ is the \emph{RBF} (also called \emph{kernel}), $\{p_k\}_{k=1}^K$ is a basis of the space of algebraic polynomials up to degree $d$, $\mathbb{P}_d(\Omega)$, and the $\varepsilon_{n}$'s are nonnegative shape parameters.\footnote{For polyharmonic splines, it is common practice to not include a shape parameter in \eqref{eq:RBF-interpol}. For simplicity, we still use \eqref{eq:RBF-interpol} and set $\varepsilon_n = 1$, $n=1,\dots,n$, in this case.}
The RBF interpolant \eqref{eq:RBF-interpol} is uniquely determined by the conditions 
\begin{alignat}{2}
	(s_{N,d}f)(\mathbf{x}_n) 
    		& = f(\mathbf{x}_n), \quad 
		&& n=1,\dots,N, \label{eq:interpol_cond} \\ 
  	\sum_{n=1}^N \alpha_n p_{k}(\mathbf{x}_n) 
    		& = 0 , \quad 
		&& k=1,\dots,K. \label{eq:cond2}
\end{alignat} 
Note that \eqref{eq:interpol_cond} and \eqref{eq:cond2} can be reformulated as a linear system for the coefficient vectors $\boldsymbol{\alpha} = [\alpha_1,\dots,\alpha_N]^T$ and $\boldsymbol{\beta} = [\beta_1,\dots,\beta_K]^T$. 
This linear system is given by 
\begin{equation}\label{eq:system} 
	\begin{bmatrix} \Phi & P \\ P^T & 0 \end{bmatrix}
	\begin{bmatrix} \boldsymbol{\alpha} \\ \boldsymbol{\beta} \end{bmatrix} 
	= 
	\begin{bmatrix} \mathbf{f} \\ \mathbf{0} \end{bmatrix}, 
\end{equation} 
where $\mathbf{f} = [f(\mathbf{x}_1),\dots,f(\mathbf{x}_N)]^T$ as well as 
\begin{equation}\label{eq:Phi_P}
	\Phi = 
	\begin{bmatrix} 
		\varphi( \varepsilon_{1} \| \mathbf{x}_1 - \mathbf{x}_1 \|_2 ) & \dots & \varphi( \varepsilon_{N} \| \mathbf{x}_1 - \mathbf{x}_N \|_2 ) \\ 
		\vdots & & \vdots \\ 
		\varphi( \varepsilon_{1} \| \mathbf{x}_N - \mathbf{x}_1 \|_2 ) & \dots & \varphi( \varepsilon_{N} \| \mathbf{x}_N - \mathbf{x}_N \|_2 )
	\end{bmatrix}, 
	\ 
	P = 
	\begin{bmatrix} 
		p_1(\mathbf{x}_1) & \dots & p_K(\mathbf{x}_1) \\ 
		\vdots & & \vdots \\  
		p_1(\mathbf{x}_N) & \dots & p_K(\mathbf{x}_N) 
	\end{bmatrix}.
\end{equation} 
For a constant shape parameter $\varepsilon_{1} = \dots = \varepsilon_{N}$, \eqref{eq:system} is ensured to have a unique solution---corresponding to existence and uniqueness of the RBF interpolant---if the kernel $\varphi$ is conditionally positive definite of order $d$ and the set of data points is $\mathbb{P}_{d}(\Omega)$-unisolvent. 
See, for instance, \cite[Chapter 7]{fasshauer2007meshfree} and \cite[Chapter 3.1]{glaubitz2020shock} or references therein. 
In this work, we shall focus on the popular choices of RBFs listed in Table \ref{tab:RBFs}. 
A more complete list of RBFs and their properties can be found in the monographs \cite{buhmann2003radial,wendland2004scattered,fasshauer2007meshfree,fornberg2015primer} and references therein.

\begin{table}[tb]
  \centering 
  \renewcommand{\arraystretch}{1.3}
  \begin{tabular}{c|c|c|c}
    RBF & $\varphi(r)$ & parameter & order \\ \hline 
    Gaussian & $\exp( -r^2)$ & & 0 \\ 
    Wendland's & $\varphi_{D,k}(r)$, see \cite{wendland1995piecewise} & $D, k \in \N_0$ & 0 \\ 
    Polyharmonic splines & $r^{2k-1}$ & $k \in \N$ & $k$ \\ 
    & $r^{2k} \log r$ & $k \in \N$ & $k+1$ 
  \end{tabular} 
  \caption{Some popular RBFs. 
  The ``order" $k$ of an RBF refers to the RBF being conditionally positive of order $k$.}
  \label{tab:RBFs}
\end{table}

The set of all RBF interpolants \eqref{eq:RBF-interpol} forms an $N$-dimensional linear space, denote by $\mathcal{S}_{N,d}$. 
This space is spanned by the \emph{cardinal functions}  
\begin{equation}\label{eq:cardinal}
  	c_m(\boldsymbol{x}) 
    = \sum_{n=1}^N \alpha_n^{(m)} \varphi( \varepsilon_{n} \| \boldsymbol{x} - \mathbf{x}_n \|_2 ) + \sum_{k=1}^K \beta^{(m)}_k p_k(\boldsymbol{x}), 
    \quad m=1,\dots,N,
\end{equation}
which are uniquely determined by the \emph{cardinal property} 
\begin{equation}\label{eq:cond_cardinal}
  c_m(\mathbf{x}_n) = \delta_{mn} := 
  \begin{cases} 
    1 & \text{if } m=n, \\ 
    0 & \text{otherwise}, 
  \end{cases} 
  \quad m,n=1,\dots,N,
\end{equation}
and condition \eqref{eq:cond2}. 
They provide us with the following representation of the RBF interpolant \eqref{eq:RBF-interpol}: 
\begin{equation*}
	(s_{N,d}f)(\boldsymbol{x}) = \sum_{n=1}^N f(\mathbf{x}_n) c_n(\boldsymbol{x})
\end{equation*} 
This representation is convenient to subsequently derive quadrature weights based on RBFs that are independent of the function $f$.

\subsection{Quadrature formulas based on radial basis functions} 
\label{sub:prelim_CFs} 

A fundamental idea behind many QFs is to first approximate the (unknown) function $f: \Omega \to \R$ based on the given data pairs $\{\mathbf{x}_n,f_n\}_{n=1}^N \subset \Omega \times \R$ and to exactly integrate this approximation. 
In the case of RBF-QFs this approximation is chosen as the RBF interpolant \eqref{eq:RBF-interpol}. 
Hence, the corresponding RBF-QF is defined as 
\begin{equation}\label{eq:RBF-CRs_def}
	C_N[f] := I[s_{N,d}f] = \int_{\Omega} (s_{N,d}f)(\boldsymbol{x}) \omega(\boldsymbol{x}) \intd \boldsymbol{x}. 
\end{equation}
When formulated w.\,r.\,t.\ the cardinal functions $c_n$ we get 
\begin{equation}\label{eq:RBF-CRs}
	C_N[f] = \sum_{n=1}^N w_n f(x_n) 
	\quad \text{with} \quad w_n = I[c_n]. 
\end{equation} 
That is, the RBF quadrature weights $\mathbf{w} = [w_1,\dots,w_N]^T$ are given by the moments corresponding to the cardinal functions. 
This formulation is often preferred over \eqref{eq:RBF-CRs_def} since the weights $\mathbf{w}$ do not have to be recomputed when another function is considered. 
In our implementation, we compute the RBF quadrature weights by solving the linear system 
\begin{equation}\label{eq:LS_weights}
	\underbrace{\begin{bmatrix} \Phi & P \\ P^T & 0 \end{bmatrix}}_{= A}  
	\begin{bmatrix} \mathbf{w} \\ \mathbf{v} \end{bmatrix} 
	= 
	\begin{bmatrix} \mathbf{m}^{\text{RBF}} \\ \mathbf{m}^{\text{poly}} \end{bmatrix},
\end{equation} 
where $\mathbf{v} \in \R^K$ is a Lagrange multiplier\footnote{The solution of \eqref{eq:LS_weights} can be interpreted as the solution of an equality constrained linear optimization problem \cite{bayona2019insight}, where $\mathbf{v}$ plays the role of a Lagrange multiplier.}. 
Furthermore, the vectors ${\mathbf{m}^{\text{RBF}} \in \R^N}$ and ${\mathbf{m}^{\text{poly}} \in \R^K}$ respectively contain the moments of the translated kernels and polynomial basis functions: 
\begin{equation*} 
\begin{aligned}
	\mathbf{m}^{\text{RBF}} & = \left[ I[\varphi_1], \dots, I[\varphi_N] \right]^T, \\ 
	\mathbf{m}^{\text{poly}} & = \left[ I[p_1], \dots, I[p_K] \right]^T,
\end{aligned}
\end{equation*} 
with $\varphi_n(\boldsymbol{x}) = \varphi( \varepsilon_{n} \| \boldsymbol{x} - \mathbf{x}_n \|_2 )$. 
The moments of different RBFs can be found in the appendix \ref{sec:app_moments}. 
The polynomial moments can be found in the literature, e.\,g., \cite[Appendix A]{glaubitz2020stableCFs} and \cite{folland2001integrate,lasserre2021simple}.
\section{Stability and the Lebesgue constant} 
\label{sec:stability}

This section addresses the stability of RBF interpolants and the corresponding RBF-QFs. 
In particular, we show that both can be estimated in terms of the Lebesgue constant. 
This was also observed in \cite{fuselier2014kernel} for RBF-QFs on certain (compact) manifolds. 
That said, we also demonstrate that RBF-QFs often come with improved stability compared to RBF interpolation.

\subsection{Stability of quadrature formulas} 
\label{sub:stability_CFs}

We shall start by addressing stability of RBF-QFs. 
To this end, let us denote the best approximation of $f$ from $\mathcal{S}_{N,d}$ in the $L^\infty$-norm by $\hat{s}$. 
That is, 
\begin{equation*}\label{eq:Lebesgue}
	\hat{s} = \argmin_{s \in \mathcal{S}_{N,d}} \norm{ f - s }_{L^{\infty}(\Omega)} 
	\quad \text{with} \quad 
	\norm{ f - s }_{L^{\infty}(\Omega)} = \sup_{\mathbf{x} \in \Omega} | f(\mathbf{x}) - s(\mathbf{x}) |. 
\end{equation*} 
Note that this best approximation w.\,r.\,t.\ the $L^\infty$-norm is not necessarily equal to the RBF interpolant. 
Still, the following error bound holds for the RBF-QF \eqref{eq:RBF-CRs}, that corresponds to exactly integrating the RBF interpolant from $\mathcal{S}_{N,d}$: 
\begin{equation}\label{eq:L-inequality} 
\begin{aligned}
	| C_N[f] - I[f] | 
		\leq \left( \| I \|_{\infty} + \| C_N \|_{\infty} \right) \inf_{ s \in \mathcal{S}_{N,d} } \norm{ f - s }_{L^{\infty}(\Omega)}
\end{aligned}
\end{equation} 
Inequality \eqref{eq:L-inequality} is commonly known as the Lebesgue inequality; see, e.\,g., \cite{van2020adaptive} or \cite[Theorem 3.1.1]{brass2011quadrature}. 
It is often encountered in polynomial interpolation \cite{brutman1996lebesgue,ibrahimoglu2016lebesgue} but straightforwardly carries over to numerical integration.
In this context, the operator norms $\| I \|_{\infty}$ and $\|C_N\|_{\infty}$ are respectively given by $\| I \|_{\infty} = I[1]$ and 
\begin{equation*}\label{eq:stab_measure}
	\| C_N \|_{\infty} 
		= \sum_{n=1}^N |w_n| 
		= \sum_{n=1}^N | I[c_n] |.
\end{equation*}
Recall that the $c_n$'s are the cardinal functions (see \S\ref{sub:prelim_RBFs}).
%
In fact, $\| C_N \|_{\infty}$ is a common stability measure for QFs. 
This is because the propagation of input errors, e.\,g., due to noise or rounding errors, can be bounded by $\| C_N \|_{\infty}$: 
Let $\tilde{f}: \Omega \to \R$ be a perturbed version of $f$, e.\,g. including noise or measurement errors, then 
\begin{equation*}
	| C_N[f] - C_N[\tilde{f}] | 
		\leq \| C_N \|_{\infty} \| f - \tilde{f} \|_{L^\infty}.
\end{equation*} 
In other words, input errors are amplified at most by a factor that is equal to the operator norm $\| C_N \|_{\infty}$. 
At the same time, we have 
\begin{equation*}
	\| C_N \|_{\infty} 
		\geq C_N[1],
\end{equation*} 
where equality holds if and only if all quadrature weights are nonnegative. 
Also, for this reason, the construction of QFs is mainly devoted to nonnegative QFs. 

\begin{definition}[Stability]\label{def:stability}
	We call the RBF-QF $C_N$ \emph{stable} if $\| C_N \|_{\infty} = C_N[1]$. 
	This is the case if and only if $I[c_n] \geq 0$ for all cardinal functions $c_n$, $n=1,\dots,N$.
\end{definition}

It is also worth noting that $C_{N}[1] = \| I \|_{\infty}$ if the QF is exact for constants. 
For RBF-QFs, this is the case if at least constants are included in the underlying RBF interpolant ($d \geq 0$).

\subsection{Stability of RBF approximations} 
\label{sub:stability_RBFs}

We now demonstrate how stability of the RBF-QF $C_N$ can be connected to stability of the corresponding RBF interpolant. 
Indeed, the stability measure $\| C_N \|_{\infty}$ can be bounded from above by 
\begin{equation*}
	\| C_N \|_{\infty} 
		\leq \| I \|_{\infty} \Lambda_N, 
		\quad \text{with} \quad 
		\Lambda_N := \sup_{\mathbf{x} \in \Omega} \sum_{n=1}^N | c_n(\mathbf{x}) |.
\end{equation*}
Here, $\Lambda_N$ is the Lebesgue constant corresponding to the recovery process $f \mapsto s_{N,d}f$ (RBF interpolation). 
Obviously, $\Lambda_N \geq 1$. 
Also note that if $1 \in \mathcal{S}_{N,d}$ (the RBF-QF is exact for constants), we observe  
\begin{equation}\label{eq:stab_eq1}
	\| I \|_{\infty} 
		\leq \| C_N \|_{\infty} 
		\leq \| I \|_{\infty} \Lambda_N.
\end{equation} 
Hence, the RBF-QF is stable ($\| C_N \|_{\infty} = \| I \|_{\infty}$) if $\Lambda_N$ is minimal ($\Lambda_N=1$).
We briefly note that the inequality $\| C_N \|_{\infty} \leq \| I \|_{\infty} \Lambda_N$ is sharp by considering the following example. 

\begin{example}[$\|C_N\|_{\infty} = \Lambda_N$]\label{ex:sharp} 
	Let us consider the domain $\Omega = [0,1]$ with $\omega \equiv 1$, which immediately implies $\| I \|_{\infty} = 1$.
	In \cite{bos2008univariate} it was shown that for the linear PHS $\varphi(r) = r$ and data points $0 = x_1 < \dots < x_N = 1$ the corresponding cardinal functions $c_m$ are simple hat functions. 
	In particular, $c_m$ is the ordinary ``connect the dots'' piecewise linear interpolant of the data pairs $(x_n,\delta_{nm})$, $n=1,\dots,N$. 
	Thus, $\Lambda_N = 1$. 
	At the same time, this yields $\|C_N\|_{\infty} = 1$ and therefore  
	$\|C_N\|_{\infty} = \Lambda_N$. 
\end{example} 

Looking for minimal Lebesgue constants is a classical problem in approximation and recovery theory \cite{micchelli1977survey,trefethen2019approximation}. 
For instance, it is well known that for polynomial interpolation, even near-optimal sets of data points yield a Lebesgue constant that grows as $\mathcal{O}(\log N)$ in one dimension and as $\mathcal{O}(\log^2 N)$ in two dimensions; see \cite{brutman1996lebesgue,bos2006bivariate,bos2007bivariate,ibrahimoglu2016lebesgue}. 
In the case of RBF interpolation, the Lebesgue constant and appropriate data point distributions were studied, for instance, in \cite{iske2003approximation,de2003optimal,mehri2007lebesgue,de2010stability}. 
That said, the second inequality in \eqref{eq:stab_eq1} also tells us that in some cases, we can expect the RBF-QF to have superior stability properties compared to the underlying RBF interpolant. 
Finally, it should be stressed that \eqref{eq:stab_eq1} only holds if $1 \in \mathcal{S}_{N,d}$. 
In general, we have 
\begin{equation*}\label{eq:stab_eq2}
	C_N[1]  
		\leq \| C_N \|_{\infty} 
		\leq \| I \|_{\infty} \Lambda_N.
\end{equation*} 
Still, this indicates that a recovery space $\mathcal{S}_{N,d}$ is desired that yields a small Lebesgue constant as well as the RBF-QF potentially having superior stability compared to RBF interpolation.
\section{Compactly supported radial basis functions}
\label{sec:compact}
 
Despite the increased use of RBF-QFs in applications, provable stability results are rarely encountered in the literature. 
As a first step towards a more mature stability theory, we next prove stability of RBF-QFs for compactly supported kernels with nonoverlapping supports. 
To be more precise, we subsequently consider RBFs $\varphi: \R_0^+ \to \R$ satisfying the following restrictions: 
\begin{enumerate}[label=(R\arabic*)] 
	\item \label{item:R1}
	$\varphi$ is nonnegative, i.\,e., $\varphi \geq 0$.
	
	\item \label{item:R2} 
	$\varphi$ is uniformly bounded. 
	W.\,l.\,o.\,g.\ we assume $\max_{r \in \R_0^+} |\varphi(r)| = 1$.
	
	\item \label{item:R3} 
	$\varphi$ is compactly supported. 
	W.\,l.\,o.\,g.\ we assume $\operatorname{supp} \varphi = [0,1]$.
	
\end{enumerate}
Already note that \ref{item:R3} implies $\operatorname{supp} \varphi_n = B_{\varepsilon_{n}^{-1}}(\mathbf{x}_n)$, where 
\begin{equation*}
	B_{\varepsilon_{n}^{-1}}(\mathbf{x}_n) := \{ \, \mathbf{x} \in \Omega \mid \| \mathbf{x}_n - \mathbf{x} \|_2 \leq \varepsilon_{n}^{-1} \, \}, 
	\quad  
	\varphi_n(\boldsymbol{x}) := \varphi( \varepsilon_{n} \| \mathbf{x}_n - \boldsymbol{x} \|_2 ).
\end{equation*}
The $\varphi_n$'s will have nonoverlapping support if the shape parameters $\varepsilon_{n}$ are sufficiently large. 
This can be ensured by the following condition: 
\begin{equation}\label{eq:R4}
	\varepsilon_{n}^{-1} \leq h_{n} 
		:= \min\left\{ \, \| \mathbf{x}_n - \mathbf{x}_m \|_2 \mid \mathbf{x}_m \in X \setminus \{\mathbf{x}_n\} \, \right\}, 
		\quad n=1,\dots,N
\end{equation} 
Here, $X$ denotes the set of data points. 
Finally, it should be pointed out that throughout this section, we assume $\omega \equiv 1$. 
This assumption is made for the main result, Theorem \ref{thm:main}, to hold. 
Its role will become clearer after consulting the proof of Theorem \ref{thm:main} and is revisited in Remark \ref{rem:omega}.

\subsection{Main result}
\label{sub:compact_main}

Our main result is the following Theorem \ref{thm:main}. 
After collecting a few preliminary results, its proof is given in \S\ref{sub:compact_proof}. 

\begin{theorem}\label{thm:main}
	Let $(\mathbf{x}_n)_{n \in \N}$ be an equidistributed sequence in $\Omega$ and $X_N = \{ \mathbf{x}_n \}_{n=1}^N$. 
	Furthermore, let $\omega \equiv 1$, let $\varphi: \R_0^+ \to \R$ be a RBF satisfying \ref{item:R1} to \ref{item:R3}, and choose the shape parameters $\varepsilon_n$ such that the corresponding functions $\varphi_n$ have nonoverlapping support and equal moments ($I[\varphi_n] = I[\varphi_m]$ for all $n,m=1,\dots,N$). 
	For every polynomial degree $d \in \N$ there exists an $N_0 \in \N$ such that for all $N \geq N_0$ the corresponding RBF-QF \eqref{eq:RBF-CRs} is stable. 
	That is, $I[c_m] \geq 0$ for all $m=1,\dots,N$. 
\end{theorem}

Note that a sequence $(\mathbf{x}_n)_{n \in \N}$ is \emph{equidistributed in $\Omega$} if and only if 
\begin{equation*}
	\lim_{N \to \infty} \frac{|\Omega|}{N} \sum_{n=1}^N g(\mathbf{x}_n) 
		= \int_{\Omega} g(\boldsymbol{x}) \intd \boldsymbol{x}
\end{equation*}
holds for all measurable bounded functions $g: \Omega \to \R$ that are continuous almost everywhere (in the sense of Lebesgue), see \cite{weyl1916gleichverteilung}. 
For details on equidistributed sequences, we refer to the monograph \cite{kuipers2012uniform}.\footnote{
Examples for equidistributed sequences include low-discrepancy points \cite{hlawka1961funktionen,niederreiter1992random,caflisch1998monte,dick2013high} used in quasi-Monte Carlo methods, such as the Halton points \cite{halton1960efficiency}.
}
Still, it should be noted that equidistributed sequences are dense sequences with a special ordering. 
In particular, if $(\mathbf{x}_n)_{n \in \N} \subset \Omega$ is equidistributed, then for every $d \in \N$ there exists an $N_0 \in \N$ such that $X_N$ is $\mathbb{P}_d(\Omega)$-unisolvent for all $N \geq N_0$; see \cite{glaubitz2020constructing}. 
This ensures that the corresponding RBF interpolant is well-defined. 
It should also be noted that if $\Omega \subset \R^D$ is bounded and has a boundary of measure zero (again in the sense of Lebesgue), then an equidistributed sequence in $\Omega$ is induced by an equidistributed sequence in the $D$-dimensional hypercube. 
More details on how an equidistributed sequence in $\Omega$ can be constructed are provided in \cite{glaubitz2020constructing}. 

\begin{remark} 
	It is always possible to ensure the equal moment condition, $I[\varphi_n] = I[\varphi_m]$ for all $n,m=1,\dots,N$, in Theorem \ref{thm:main} by allowing the  points closer to the boundary to come with a smaller shape parameter. 
	In this way, one can compensate for the part of the support cut off by the boundary of the domain $\Omega$.  
	For instance, if equally spaced points are used on $[a,b]$ with $a = x_1 < \dots < x_N = b$, then the shape parameter is $\varepsilon_n = \varepsilon$ for the interior points ($n=2,\dots,N-1$) and $\varepsilon_1 = \varepsilon_N = \varepsilon/2$ for the boundary points, where $\varepsilon$ is a suitable chosen reference parameter. 
	That said, in our numerical tests, we observed Theorem \ref{thm:main} also to hold when the equal moment condition was not satisfied. 
\end{remark}

\begin{remark}
	It is not necessary to include polynomials in the RBF-QF \eqref{eq:RBF-CRs} for Theorem \ref{thm:main} to imply stability. 
	Indeed, it is subsequently proved by Lemma \ref{lem:stab_noPoly} that the RBF-QF \eqref{eq:RBF-CRs} can also be stable when no polynomials are included. 
	Sometimes, the RBF-QF \eqref{eq:RBF-CRs} is also referred to as the ``RBF+poly-QF" when polynomials are included. 
	In this regard, Theorem \ref{thm:main} shows that stability of RBF-QFs carries over to RBF+poly-QFs under the assumptions listed in Theorem \ref{thm:main}. 
	The influence of including polynomials into the RBF-QFs on their stability is also discussed for other kernels in \S\ref{sec:connection}.  
\end{remark}

\subsection{Explicit representation of the cardinal functions}
\label{sub:compact_explicit}

In preparation of proving Theorem \ref{thm:main} we derive an explicit representation for the cardinal functions $c_n$ under the restrictions \ref{item:R1}--\ref{item:R3} and \eqref{eq:R4}. 
In particular, we use the concept of discrete orthogonal polynomials (DOPs). 
Let us define the following discrete inner product corresponding to the data points $X_N = \{\mathbf{x}_n\}_{n=1}^N$: 
\begin{equation}\label{eq:discrete_scp}
	[u,v]_{X_N} = \frac{|\Omega|}{N} \sum_{n=1}^N u(\mathbf{x}_n) v(\mathbf{x}_n)
\end{equation} 
Recall that the data points $X_N$ are coming from an equidistributed sequence and are ensured to be $\mathbb{P}_d(\Omega)$-unisolvent for any degree $d \in \N$ if $N$ is sufficiently large. 
In this case, \eqref{eq:discrete_scp} is positive definite on $\mathbb{P}_d(\Omega)$, i.\,e., $[u,u]_{X_N} > 0$ if $u \in \mathbb{P}_d(\Omega)$ and $u \neq 0$. 
We say that the basis $\{p_k\}_{k=1}^K$ of $\mathbb{P}_d(\Omega)$, where $K = \dim \mathbb{P}_d(\Omega)$, consists of \emph{DOPs} if 
\begin{equation*}
	[p_k,p_l]_{X_N} = \delta_{kl} := 
	\begin{cases} 
		1 & \text{ if } k=l, \\ 
		0 & \text{ otherwise}, 
	\end{cases} 
	\quad k,l=1,\dots,K.
\end{equation*} 
We now come to the desired explicit representation for the cardinal functions $c_m$. 

\begin{lemma}[Explicit representation for $c_m$]\label{lem:rep_cm}
	Let the RBF $\varphi: \R_0^+ \to \R$ satisfy \ref{item:R2} and \ref{item:R3}. 
	Furthermore, choose the shape parameters $\varepsilon_n$ such that the corresponding functions $\varphi_n$ have nonoverlapping support and let the basis $\{p_k\}_{k=1}^K$ consists of DOPs. 
	Then, the cardinal function $c_m$, $m=1,\dots,N$, is given by 
	\begin{equation}\label{eq:rep_cm}
	\begin{aligned}
		c_m(\boldsymbol{x}) 
			= \varphi_m(\boldsymbol{x}) 
			- \frac{|\Omega|}{N} \sum_{n=1}^N \left( \sum_{k=1}^K p_k(\mathbf{x}_m) p_k(\mathbf{x}_n) \right) \varphi_n(\boldsymbol{x}) 
			+ \frac{|\Omega|}{N} \sum_{k=1}^K p_k(\mathbf{x}_m) p_k(\boldsymbol{x}). 
	\end{aligned}
	\end{equation} 
\end{lemma}

\begin{proof} 
	Let $m,n \in \{1,\dots,N\}$. 
	The restrictions \ref{item:R2}, \ref{item:R3} together with the assumption of the $\varphi_n$'s having nonoverlapping support yields $\varphi_n(\mathbf{x}_m) = \delta_{mn}$.
	Hence, \eqref{eq:cardinal} and \eqref{eq:cond_cardinal} imply 
	\begin{equation}\label{eq:alpha}
		\alpha_n^{(m)} = \delta_{mn} - \sum_{k=1}^K \beta^{(m)}_k p_k(\mathbf{x}_n).
	\end{equation} 
	If we substitute \eqref{eq:alpha} into \eqref{eq:cond2}, we get 
	\begin{equation*} 
		p_l(\mathbf{x}_m) - \frac{N}{|\Omega|} \sum_{k=1}^K \beta^{(m)}_k [p_k,p_l]_{X_N} = 0, 
		\quad l=1,\dots,K. 
	\end{equation*} 
	Thus, if $\{p_k\}_{k=1}^K$ consists of DOPs, this gives us 
	\begin{equation}\label{eq:beta}
		\beta^{(m)}_l = \frac{N}{|\Omega|} p_l(\mathbf{x}_m), \quad l=1,\dots,K.
	\end{equation} 
	Finally, substituting \eqref{eq:beta} into \eqref{eq:alpha} yields 
	\begin{equation*} 
		\alpha_n^{(m)} = \delta_{mn} - \frac{N}{|\Omega|} \sum_{k=1}^K p_k(\mathbf{x}_m) p_k(\mathbf{x}_n), 
	\end{equation*} 
	and therefore the assertion. 
\end{proof}

It should be stressed that using a basis of DOPs is not necessary for implementing RBF-QFs. 
In fact, the quadrature weights are---ignoring computational considerations---independent of the polynomial basis w.\,r.\,t.\ which the matrix $P$ and the corresponding moments $\mathbf{m}^{\text{poly}}$ are formulated. 
We only use DOPs as a theoretical tool to show stability of RBF-QFs.

\subsection{Some low hanging fruits}
\label{sub:compact_low}

Using the explicit representation \eqref{eq:rep_cm} it is trivial to prove stability of RBF-QFs ($I[c_m] \geq 0$ for all $m=1,\dots,N$) when no polynomial term or only a constant is included in the RBF interpolant. 

\begin{lemma}[No polynomials]\label{lem:stab_noPoly}
	Let the RBF $\varphi: \R_0^+ \to \R$ satisfy \ref{item:R1} to \ref{item:R3} and choose the shape parameters $\varepsilon_n$ such that the corresponding functions $\varphi_n$ have nonoverlapping support.
	Assume that no polynomials are included in the corresponding RBF interpolant ($K=0$). 
	Then, the associated RBF-QF is stable.
\end{lemma}

\begin{proof} 
	It is obvious that $c_m(\boldsymbol{x}) = \varphi_m(\boldsymbol{x})$. 
	Thus, by restriction \ref{item:R1}, $c_m$ is nonnegative and therefore $I[c_m] \geq 0$.
\end{proof}

\begin{lemma}[Only a constant]
	Let the RBF $\varphi: \R_0^+ \to \R$ satisfy \ref{item:R1} to \ref{item:R3} and choose the shape parameters $\varepsilon_n$ such that the corresponding functions $\varphi_n$ have nonoverlapping support.
	Assume that only a constant is included in the corresponding RBF interpolant ($K=1$). 
	Then, the associated RBF-QF is stable.
\end{lemma}

\begin{proof} 
	Let $m \in \{1,\dots,N\}$. 
	If we choose $p_1 \equiv |\Omega|^{-1/2}$, Lemma \ref{lem:rep_cm} yields 
	\begin{equation*} 
		c_m(\boldsymbol{x}) 
			= \varphi_m(\boldsymbol{x}) 
			+ \frac{1}{N} \left( 1 - \sum_{n=1}^N \varphi_n(\boldsymbol{x}) \right). 
	\end{equation*} 
	Note that by \ref{item:R2}, \ref{item:R3}, and \eqref{eq:R4}, we therefore have $c_m(\boldsymbol{x}) \geq \varphi_m(\boldsymbol{x})$. 
	Hence, \ref{item:R1} implies the assertion. 
\end{proof}

\subsection{Proof of the main results}
\label{sub:compact_proof} 

The following technical Lemma will be convenient to the proof of Theorem \ref{thm:main}. 

\begin{lemma}\label{lem:technical}
	Let $(\mathbf{x}_n)_{n \in \N}$ be equidistributed in $\Omega$, $X_N = \{ \mathbf{x}_n \}_{n=1}^N$, and let $[\cdot,\cdot]_{X_N}$ be the discrete inner product \eqref{eq:discrete_scp}. 
	Furthermore, let $\{ p_k^{(N)} \}_{k=1}^K$ be a basis of $\mathbb{P}_d(\Omega)$ consisting of DOPs w.\,r.\,t.\ $[\cdot,\cdot]_{X_N}$. 
	Then, for all $k=1,\dots,K$, 
	\begin{equation*}
		p_k^{(N)} \to p_k \quad \text{in } L^{\infty}(\Omega), \quad N \to \infty,
	\end{equation*} 
	where $\{ p_k \}_{k=1}^K$ is a basis of $\mathbb{P}_d(\Omega)$ consisting of continuous orthogonal polynomials satisfying 
	\begin{equation*} 
		\int_{\Omega} p_k(\boldsymbol{x}) p_l(\boldsymbol{x}) \intd \boldsymbol{x} 
			= \delta_{kl}, \quad k,l=1,\dots,K.
	\end{equation*}
	Moreover, it holds that 
	\begin{equation*} 
		\lim_{N \to \infty} \int_{\Omega} p_k^{(N)}(\boldsymbol{x}) p_l^{(N)}(\boldsymbol{x}) \intd \boldsymbol{x} = \delta_{kl}, \quad k,l=1,\dots,K.
	\end{equation*}
\end{lemma}

\begin{proof} 
	The assertion is a combination of Lemma 11 and 12 from \cite{glaubitz2020stableCFs}, where a general positive weight function $\omega$ was considered. 
	Here, we only consider the case $\omega \equiv 1$. 
\end{proof}

Essentially, Lemma \ref{lem:technical} states that if a sequence of discrete inner products converges to a continuous one, then also the corresponding DOPs---assuming that the ordering of the elements does not change---converges to a basis of continuous orthogonal polynomials. 
Furthermore, this convergence holds in a uniform sense. 
We are now able to provide a proof for Theorem \ref{thm:main}. 

\begin{proof}[Proof of Theorem \ref{thm:main}]
	Let $d \in \N$ and $m \in \{1,\dots,N\}$. 
	Under the assumptions of Theorem \ref{thm:main}, we have $I[\varphi_n] = I[\varphi_m]$ for all $n=1,\dots,N$. 
	Thus, Lemma \ref{lem:rep_cm}implies  
	\begin{equation*}
		I[c_m] 
			= I[\varphi_m] \left( 1 - \frac{|\Omega|}{N} \sum_{n=1}^N \sum_{k=1}^K p^{(N)}_k(\mathbf{x}_m) p^{(N)}_k(\mathbf{x}_{n}) \right)
			+ \frac{|\Omega|}{N} \sum_{k=1}^K p^{(N)}_k(\mathbf{x}_m) I[p_k]. 
	\end{equation*} 
	Let $\{ p_k^{(N)} \}_{k=1}^K$ be a basis of $\mathbb{P}_d(\Omega)$ consisting of DOPs. 
	That is, $[p_k^{(N)},p_l^{(N)}]_{X_N} = \delta_{kl}$. 
	In particular, $p_1^{(N)} \equiv |\Omega|^{-1/2}$. 
	With this in mind, it is easy to verify that 
	\begin{equation}\label{eq:omega_proof1} 
	\begin{aligned} 
		\frac{|\Omega|}{N} \sum_{n=1}^N \sum_{k=1}^K p^{(N)}_k(\mathbf{x}_m) p^{(N)}_k(\mathbf{x}_n) 
		= \sum_{k=1}^K p^{(N)}_k(\mathbf{x}_m) |\Omega|^{1/2} [p^{(N)}_k,p^{(N)}_1]_{X_N}
		= 1. 
	\end{aligned}
	\end{equation} 
	Thus, we have 
	\begin{equation*} 
		I[c_m] \geq 0 \iff 
			\sum_{k=1}^K p_k^{(N)}(\mathbf{x}_m) I[p_k^{(N)}] \geq 0.
	\end{equation*} 
	Finally, observe that 
	\begin{equation*}
		\sum_{k=1}^K p_k^{(N)}(\mathbf{x}_m) I[p_k^{(N)}] 
			= |\Omega|^{1/2} \sum_{k=1}^K p_k^{(N)}(\mathbf{x}_m) \int_{\Omega} p_k^{(N)}(\boldsymbol{x}) p_1^{(N)}(\boldsymbol{x}) \intd \boldsymbol{x}, 
	\end{equation*} 
	under the assumption that $\omega \equiv 1$. 
	Lemma \ref{lem:technical} therefore implies 
	\begin{equation}\label{eq:omega_proof2} 
		\lim_{N \to \infty} \sum_{k=1}^K p_k^{(N)}(\mathbf{x}_m) I[p_k^{(N)}] = 1, 
	\end{equation}
	which completes the proof. 
\end{proof} 

\begin{remark}[On the assumption that $\omega \equiv 1$]\label{rem:omega}
	The assumption that $\omega \equiv 1$ in Theorem \ref{thm:main} is necessary for \eqref{eq:omega_proof1} and \eqref{eq:omega_proof2} to both hold true. 
	On the one hand, \eqref{eq:omega_proof1} is ensured by the the DOPs being orthogonal w.\,r.\,t.\ the discrete inner product \eqref{eq:discrete_scp}. 
	This discrete inner product can be considered as an approximation to the continuous inner product ${\scp{u}{v} = \int_{\Omega} u(\boldsymbol{x}) v(\boldsymbol{x}) \intd \boldsymbol{x}}$. 
	This also results in Lemma\ref{lem:technical}. 
	On the other hand, in general, \eqref{eq:omega_proof2} only holds if the DOPs converge to a basis of polynomials that is orthogonal w.\,r.\,t.\ the weighted continuous inner product ${\scp{u}{v}_{\omega} = \int_{\Omega} u(\boldsymbol{x}) v(\boldsymbol{x}) \omega(\boldsymbol{x}) \intd \boldsymbol{x}}$. 
	Hence, for \eqref{eq:omega_proof1} and \eqref{eq:omega_proof2} to both hold true at the same time, we have to assume that $\omega \equiv 1$. 
	In this case, the two continuous inner products are the same.
\end{remark}

\section{Provable stable least squares RBF-QFs} 
\label{sub:LS} 

Theorem \ref{thm:main} shows that compactly supported RBFs (e.\,g.\ Wendland's kernels) can lead to stable interpolatory QFs if the shape parameter is so that none of the shifted kernels have a region of overlap. 
In our numerical tests, we observed this condition not just to be sufficient but also often being ``close to" necessary. 
We often found the RBF-QF even to have negative weights when the support regions only slightly overlapped. 
At the same time, it is known that scaling Wendland's kernels so that the support decreases with the number of data points results in the interpolation error to decrease only slowly or even to stagnate \cite{fasshauer2007meshfree}. 

To provide a more practical procedure for ensuring stability of RBF-QFs, we now demonstrate how a least-squares approach \cite{huybrechs2009stable,glaubitz2020stable,glaubitz2020stableQFs,glaubitz2020stableCFs} can be used to construct stable RBF-QFs by allowing the number of data points used for numerical integration to be larger than the number of centers that are used to generate the RBF approximation space. 
The subsequent least-squares approach is \emph{not} limited to compactly supported kernels and can be used to construct stable QFs that are exact for fairly general RBF approximation spaces. 
The only restrictions are that the RBF approximation space consists of continuous and bounded functions and contains constants. 
Further, the number of data points used by quadrature has to be sufficiently larger than the dimension of the RBF approximation space. 
Although the least-squares approach has recently been extended to general multi-dimensional function spaces that include constants in \cite{glaubitz2020constructing}, the implications for RBF-QF have not yet been explored. 
To this end, we consider a given center point set $Y_M = \{\mathbf{y}_m\}_{m=1}^M$, generating the $M$-dimensional RBF space $\mathcal{S}_{M,d}$, and a larger data point set $X_N = \{\mathbf{x}_n\}_{n=1}^N$ with $N > M$. 
Then, any QF $C_N[f] = \sum_{n=1}^N w_n f(x_n)$ that is exact for all $f \in \mathcal{S}_{M,d}$ has to satisfy 
\begin{equation}\label{eq:LS_weights2} 
	\underbrace{
	\begin{bmatrix} 
		b_1(\mathbf{x}_1) & \dots & b_1(\mathbf{x}_N) \\ 
		\vdots & & \vdots \\ 
		b_M(\mathbf{x}_1) & \dots & b_M(\mathbf{x}_N) 
	\end{bmatrix} 
	}_{= B}
	\underbrace{
	\begin{bmatrix} w_1 \\ \vdots \\ w_N \end{bmatrix} 
	}_{= \mathbf{w}}
	= 
	\underbrace{
	\begin{bmatrix} I[b_1] \\ \vdots \\ I[b_M] \end{bmatrix}
	}_{= \mathbf{m}}, 
\end{equation}
where $\{b_m\}_{m=1}^M$ is a basis of $\mathcal{S}_{M,d}$. 
The matrix $B$ in \eqref{eq:LS_weights2} depends on $X_N$ and $Y_M$ (as well as on the kernel $\varphi$ and the polynomial degree $d$), which we denote by $B = B(X_N,Y_M)$. 
Assume that the data point set $X_N$ is $\mathcal{S}_{M,d}$-unisolvent, i.\,e., 
\begin{equation*}
	f(\mathbf{x}_n) = 0, \ \forall \mathbf{x}_n \in X_N \implies f \equiv 0 
\end{equation*} 
holds for all $f \in \mathcal{S}_{M,d}$.\footnote{ 
$X_N$ is $\mathcal{S}_{M,d}$-unisolvent, for instance, when the kernel $\varphi$ is conditionally positive definite of order $d$ and $X_N$ is $\mathbb{P}_{d}(\Omega)$-unisolvent, which is a common assumption to ensure uniqueness of RBF interpolants.
}
Then \eqref{eq:LS_weights2} has infinitely many solutions, which form a $(N-M)$-dimensional affine linear subspace $W$. 
Every $\mathbf{w} \in W$ yields a QF that is exact for all functions from the RBF space $\mathcal{S}_{M,d}$. 
We want to find a positive solution $\mathbf{w} \in W$ so that the corresponding QF with weights $\mathbf{w}$ is stable (see Definition \ref{def:stability}). 
To this end, we use the following result from \cite{glaubitz2020constructing}. 
\begin{lemma}[Corollary 3.6 in \cite{glaubitz2020constructing}]\label{lem:LS}
	Let $\Omega \subset \R^D$, $\omega: \Omega \to \R_0^+$ be a Riemann integrable weight function that is positive almost everywhere, and let $\mathcal{F} = {\rm span}\{ \, b_m \mid m=1,\dots,M \, \}$ be a finite-dimensional linear space of continuous and bounded functions that contains constants. 
	Further, let $(\mathbf{x}_n)_{n \in \N}$ be an equidistributed sequence in $\Omega$ with $\omega(\mathbf{x}_n) > 0$ for all $n \in \N$ and denote the affine linear subspace of solutions of \eqref{eq:LS_weights2} by $W_{\mathcal{F}}$. 
	Then there exists an $N_0 \in \N$ such that for all $N \geq N_0$ and discrete weights 
	\begin{equation*}
		r_{n,N} = |\Omega| \omega(\mathbf{x}_n)/N, \quad n=1,\dots,N,  
	\end{equation*}
	the corresponding least-squares QF 
	\begin{equation*} 
		C_N^{\rm LS}[f] = \sum_{n=1}^N w_n^{\rm LS} f(\mathbf{x}_n) 
		\quad \text{with} \quad 
		\mathbf{w}^{\rm LS} = \argmin_{\mathbf{w} \in W_{\mathcal{F}}} \| R^{-1/2} \mathbf{w} \|_2,
	\end{equation*} 
	where $R^{-1/2} = \diag( 1/\sqrt{r_1}, \dots, 1/\sqrt{r_N} )$, is positive and exact for all $f \in \mathcal{F}$. 
\end{lemma}
If we apply Lemma \ref{lem:LS} to the RBF function space $\mathcal{S}_{M,d}$, we get Corollary \ref{cor:LSRBFQF}. 
\begin{corollary}\label{cor:LSRBFQF}
	Let $\Omega \subset \R^D$ be compact and let $\omega: \Omega \to \R_0^+$ be a Riemann integrable weight function that is positive almost everywhere. 
	Further, let $d\geq0$ be an integer, let $\varphi:\R^+_0: \to \R$ be a continuous and conditionally positive kernel of order $d$, and let $\{ \mathbf{y}_m \}_{m=1}^M$ be a given set of centers. 
	If $(\mathbf{x}_n)_{n \in \N}$ is an equidistributed sequence in $\Omega$ with $\omega(\mathbf{x}_n) > 0$ for all $n \in \N$, then there exists an $N_0 \in \N$ such that for all $N \geq N_0$ and discrete weights 
	\begin{equation*}
		r_{n,N} = |\Omega| \omega(\mathbf{x}_n)/N, \quad n=1,\dots,N,  
	\end{equation*}
	the corresponding least-squares RBF-QF
	\begin{equation}\label{eq:weighted_LS_sol} 
		C_N^{\rm LS}[f] = \sum_{n=1}^N w_n^{\rm LS} f(\mathbf{x}_n) 
		\quad \text{with} \quad 
		\mathbf{w}^{\rm LS} = \argmin_{\mathbf{w} \in W} \| R^{-1/2} \mathbf{w} \|_2,
	\end{equation} 
	where $R^{-1/2} = \diag( 1/\sqrt{r_1}, \dots, 1/\sqrt{r_N} )$, is positive and exact for all $f \in \mathcal{S}_{M,d}$.
\end{corollary} 
\begin{proof}
	We first note that the RBF function space $\mathcal{S}_{M,d}$, which we defined in \S \ref{sec:prelim}, is $M$-dimensional with $M < \infty$, i.\,e., finite-dimensional. 
	Because the kernel $\varphi$ and all polynomials up to degree $d$ are continuous, all functions from $\mathcal{S}_{M,d}$ are continuous. 
	Further, since $\Omega \subset \R^D$ is compact and all functions from $\mathcal{S}_{M,d}$ are continuous, they are also bounded. 
	Finally, $d \geq 0$ implies that $\mathcal{S}_{M,d}$ contains constants. 
	Corollary \ref{cor:LSRBFQF} now follows from Lemma \ref{lem:LS} with $\mathcal{F} = \mathcal{S}_{M,d}$. 
\end{proof}

The weighted least-squares solution \eqref{eq:weighted_LS_sol} has the advantage of being easy and efficient to compute using standard tools from linear algebra.  
The above discussion motivates us to formulate the following procedure to construct stable least-squares RBF-QFs (LSRBF-QFs). 

\begin{algorithm}[h]
\caption{Constructing stable LSRBF-QFs}\label{algo:LSRBF} 
\begin{algorithmic}[1]
    \STATE{\textbf{Input:} Center points $\{\mathbf{y}_m\}_{m=1}^M$, kernel $\varphi$, polynomial degree $d \geq 0$, weight function $\omega$, and equidistributed data points $(\mathbf{x}_n)_{n \in \N}$} 
    \STATE{\textbf{Output:} An integer $N \geq M$ and a stable LSRBF-QF with points $\{\mathbf{x}_n\}_{n=1}^N$ and weights $\mathbf{w}^{\rm LS} \in \R^N$} 
    \STATE{Set $\mathbf{w}^{\rm LS}$ equal to the weights of the interpolatory RBF-QF given by \eqref{eq:LS_weights}}
    \REPEAT
    		\STATE{Increase the number of data points by one: $N = N+1$} 
    		\STATE{Set $X_N = \{\mathbf{x}_n\}_{n=1}^N$} 
		\STATE{Compute the matrix $B = B(X_N,Y_M)$ as in \eqref{eq:LS_weights2}}
		\STATE{Compute the weighted least-squares solution $\mathbf{w}^{\rm LS}$ as in \eqref{eq:weighted_LS_sol}}
		\STATE{Determine the smallest weight: $w_{\rm min} = \min( \mathbf{w}^{\rm LS} )$}
    \UNTIL{$\mathbf{w}^{\rm LS} \geq 0$}
\end{algorithmic}
\end{algorithm} 

Algorithm \ref{algo:LSRBF} assumes that $X_M$ is $\mathcal{S}_{M,d}$-unisolvent, since the interpolatory RBF-QF given by \eqref{eq:LS_weights} would not be defined otherwise. 
The possible advantage of stable LSRBF-QFs compared to (potentially unstable) interpolatory RBF-QFs is demonstrated in \S\ref{sub:num_LSRBF}. 
Finally, we point out the potential application of stable LSRBF-QFs to the construction of stable RBF methods for time-dependent hyperbolic partial differential equations \cite{tominec2021stability,glaubitz2022energy}. 
A crucial part of these methods is replacing exact integrals involving the approximate solution---in this case, an (local) RBF function--- with a quadrature that should be as accurate as possible for functions from the approximation space. 
\section{Polynomial terms do not influence asymptotic stability}
\label{sec:connection} 

Recall that Theorem \ref{thm:main} in \S\ref{sec:compact} holds regardless of the degree $d$ of the polynomial term included in the RBF interpolant. 
Indeed, one might generally ask, ``how are polynomial terms influencing stability of the RBF-QF?".
In what follows, we address this question by showing that---under certain assumptions that are to be specified yet---at least asymptotic stability of RBF-QFs is independent of polynomial terms. 

Recently, the following explicit formula for the cardinal functions was derived in \cite{bayona2019insight,bayona2019comparison}. 
Let us denote ${\mathbf{c}(\boldsymbol{x}) = [c_1(\boldsymbol{x}),\dots,c_N(\boldsymbol{x})]^T}$, where $c_1,\dots,c_N$ are the cardinal functions spanning $\mathcal{S}_{N,d}$; see \eqref{eq:cardinal} and \eqref{eq:cond_cardinal}. 
Provided that $\Phi$ and $P$ in \eqref{eq:Phi_P} have full rank\footnote{
$P$ having full rank means that $P$ has full column rank, i.\,e., the columns of $P$ are linearly independent. 
This is equivalent to the set of data points being $\mathbb{P}_d(\Omega)$-unisolvent. 
}, 
\begin{equation}\label{eq:formula_Bayona}
	\mathbf{c}(\boldsymbol{x}) 
		= \hat{\mathbf{c}}(\boldsymbol{x}) 
		- B \boldsymbol{\tau}(\boldsymbol{x})
\end{equation} 
holds. 
Here, $\hat{\mathbf{c}}(\boldsymbol{x}) = [\hat{c}_1(\boldsymbol{x}),\dots,\hat{c}_N(\boldsymbol{x})]^T$ are the cardinal functions corresponding to the pure RBF interpolation without polynomials. 
That is, they span $\mathcal{S}_{N,-1}$. 
At the same time, $B$ and $\boldsymbol{\tau}$ are defined as 
\begin{equation*} 
	B := \Phi^{-1} P \left( P^T \Phi^{-1} P \right)^{-1}, 
	\quad 
	\boldsymbol{\tau}(\boldsymbol{x}) := P^T \hat{\mathbf{c}}(\boldsymbol{x}) - \mathbf{p}(\boldsymbol{x})
\end{equation*} 
with ${\mathbf{p}(\boldsymbol{x}) = [p_1(\boldsymbol{x}),\dots,p_K(\boldsymbol{x})]^T}$. 
Note that $\boldsymbol{\tau}$ can be interpreted as a residual measuring how well pure RBFs can approximate polynomials up to degree $d$. 
Recalling \eqref{eq:RBF-CRs}, we see that \eqref{eq:formula_Bayona} implies 
\begin{equation}\label{eq:relation_weights}
	\mathbf{w} 
		= \hat{\mathbf{w}} - B I[\boldsymbol{\tau}],
\end{equation} 
where $\mathbf{w}$ is the vector of quadrature weights of the RBF-QF with polynomials ($d \geq 0$). 
At the same time, $\hat{\mathbf{w}}$ is the vector of weights corresponding to the pure RBF-QF without polynomial augmentation ($d=-1$). 
Moreover, $I[\boldsymbol{\tau}]$ denotes the componentwise application of the integral operator $I$.
It was numerically demonstrated in \cite{bayona2019insight} that for fixed $d \in \N$ one has 
\begin{equation}\label{eq:observation_Bayona}
	\max_{\boldsymbol{x} \in \Omega} \| B \boldsymbol{\tau}(\boldsymbol{x}) \|_{\ell^\infty} \to 0 
	\quad \text{as} \quad N \to \infty
\end{equation}
if PHS are used. 
Note that, for fixed $\boldsymbol{x} \in \Omega$, $B \boldsymbol{\tau}(\boldsymbol{x})$ is an $N$-dimensional vector and $\| B \boldsymbol{\tau}(\boldsymbol{x}) \|_{\ell^\infty}$ denotes its $\ell^{\infty}$-norm. 
That is, the maximum absolute value of the $N$ components. 
It should be pointed out that \eqref{eq:observation_Bayona} was numerically demonstrated only for PHS in \cite{bayona2019insight}. 
However, the relations \eqref{eq:formula_Bayona} and \eqref{eq:relation_weights} hold for general RBFs as well as varying shape parameters, assuming that $\Phi$ and $P$ have full rank. 
Please see \cite[Section 4]{bayona2019insight} for more details. 
We also remark that \eqref{eq:observation_Bayona} implies the weaker statement 
\begin{equation}\label{eq:cond}
	\| B \boldsymbol{\tau}(\cdot) \|_{\ell^1} \to 0 
	\ \ \text{in } L^1(\Omega) \quad \text{as} \quad N \to \infty.
\end{equation} 
Here, $B \boldsymbol{\tau}(\cdot)$ denotes a vector-valued function, $B \boldsymbol{\tau}: \Omega \to \R^N$. 
That is, for a fixed argument $\boldsymbol{x} \in \Omega$, $B \boldsymbol{\tau}(\boldsymbol{x})$ is an $N$-dimensional vector in $\R^N$ and $\| B \boldsymbol{\tau}(\boldsymbol{x}) \|_{\ell^1}$ denotes the usual $\ell^1$-norm of this vector.  
Thus, \eqref{eq:cond} means that the integral of the $\ell^1$-norm of the vector-valued function $B \boldsymbol{\tau}(\cdot)$ converges to zero as $N \to \infty$.
The above condition is not just weaker than \eqref{eq:observation_Bayona} (see Remark \ref{rem:assumption2}), but also more convenient to investigate stability of QFs. 
Indeed, we have the following results.

\begin{theorem}\label{thm:connection} 
	Let $\omega \in L^{\infty}(\Omega)$.
	Assume $\Phi$ and $P$ in \eqref{eq:Phi_P} have full rank, and assume \eqref{eq:cond} holds. 
	Then the two following statements are equivalent: 
	\begin{enumerate} 
		\item[(a)] 
		$\| \hat{\mathbf{w}} \|_{\ell^1} \to \|I\|_{\infty}$ for $N \to \infty$ 
		
		\item[(b)]
		$\| \mathbf{w} \|_{\ell^1} \to \|I\|_{\infty}$ for $N \to \infty$
	\end{enumerate} 
	That is, either both the pure and polynomial augmented RBF-QF are asymptotically stable or none is. 
\end{theorem}

A short discussion on the term ``asymptotically stable" is subsequently provided in Remark \ref{rem:asymptotic_stable}. 

\begin{proof} 
	Assume $\Phi$ and $P$ in \eqref{eq:Phi_P} have full rank, and assume \eqref{eq:cond} holds. 
	Then \eqref{eq:relation_weights} follows and therefore 
	\begin{equation}\label{eq:connection_proof} 
	\begin{aligned}
		\| \mathbf{w} \|_{\ell^1}
			& \leq \| \hat{\mathbf{w}} \|_{\ell^1} + \| BI[\boldsymbol{\tau}] \|_{\ell^1}, \\ 
		\| \hat{\mathbf{w}} \|_{\ell^1}
			& \leq \| \mathbf{w} \|_{\ell^1} + \| BI[\boldsymbol{\tau}] \|_{\ell^1}.
	\end{aligned}
	\end{equation} 
	Next, note that $BI[\boldsymbol{\tau}] = I[ B \boldsymbol{\tau}]$, and thus 
	\begin{equation*} 
	\begin{aligned} 
		\| BI[\boldsymbol{\tau}] \|_{\ell^1} 
			= \sum_{n=1}^N \left| I[ (B \boldsymbol{\tau})_n ] \right|  
			\leq I \left[ \sum_{n=1}^N | (B \boldsymbol{\tau})_n | \right] 
			= I \left[ \| B \boldsymbol{\tau} \|_{\ell^1} \right]. 
	\end{aligned}
	\end{equation*} 
	Since $\omega \in L^{\infty}(\Omega)$, it follows that 
	\begin{equation*} 
		\| BI[\boldsymbol{\tau}] \|_{\ell^1} 
			\leq \| \omega \|_{L^{\infty}(\Omega)} \int_{\Omega} \| B \boldsymbol{\tau}(\boldsymbol{x}) \|_{\ell^1} \intd \boldsymbol{x}.
	\end{equation*} 
	Hence, by assuming that \eqref{eq:cond} holds, we get $\| BI[\boldsymbol{\tau}] \|_{\ell^1} \to 0$ for fixed $d \in \N$ and $N \to \infty$. 
	Finally, substituting this into \eqref{eq:connection_proof} yields the assertion. 
\end{proof}

Theorem \ref{thm:connection} states that--under the listed assumptions---it is sufficient to consider asymptotic stability of the pure RBF-QF. 
Once asymptotic (in)stability is established for the pure RBF-QF, by Theorem \ref{thm:connection}, it also carries over to all corresponding augmented RBF-QFs. 
Interestingly, this follows our findings for compactly supported RBFs reported in Theorem \ref{thm:main}. 
There, conditional stability was ensured independently of the degree of the augmented polynomials.

\begin{remark}[Asymptotic stability]\label{rem:asymptotic_stable}
	We call a sequence of QFs with weights $\mathbf{w}_N \in \R^N$ for $N \in \N$ asymptotically stable if $\| \mathbf{w}_N \|_{\ell^1} \to \| I \|_{\infty}$ for $N \to \infty$. 
	Recall that $\| \mathbf{w}_N \|_{\ell^1} = \|C_N\|_{\infty}$ if the weights $\mathbf{w}_N$ correspond to the $N$-point QF $C_N$. 
	It is easy to note that this is a weaker property than every single QF being stable, i.\,e., $\| \mathbf{w}_N \|_{\ell^1} = \| I \|_{\infty}$ for all $N \in \N$. 
	That said, consulting \eqref{eq:L-inequality}, asymptotic stability is sufficient for the QF to converge for all functions that can be approximated arbitrarily accurate by RBFs w.\,r.\,t.\ the $L^{\infty}(\Omega)$-norm. 
	Of course, the propagation of input errors might be suboptimal for every single QF. 
\end{remark}

Theorem \ref{thm:connection} makes two assumptions. 
(1) $\Phi$ and $P$ are full rank matrices; and 
(2) the condition \eqref{eq:observation_Bayona} holds. 
In the two following remarks, we comment on these assumptions. 

\begin{remark}[On the first assumption of Theorem \ref{thm:connection}]\label{rem:assumption1}
	Although requiring $A$ and $P$ to have full rank might seem restrictive, there are often even more restrictive constraints in practical problems. 
	For instance, when solving partial differential equations, the data points are usually required to be smoothly scattered so that the distance between data points is kept roughly constant. 
	It seems unlikely to find $A$ and $P$ to be singular for such data points. 
	See \cite{bayona2019insight} for more details. 
\end{remark}

\begin{remark}[On the second assumption of Theorem \ref{thm:connection}]\label{rem:assumption2} 
	The second assumption for Theorem \ref{thm:connection} to hold is that \eqref{eq:cond} is satisfied.
	That is, the integral of $\| B \boldsymbol{\tau}(\cdot) \|_{\ell^1}: \Omega \to \R_0^+$ converges to zero as $N \to \infty$. 
	This is a weaker condition than the maximum value of $\| B \boldsymbol{\tau}(\cdot) \|_{\ell^1}$ converging to zero, which was numerically observed to hold for PHS in \cite{bayona2019insight}. 
	The relation between these conditions can be observed by applying H\"older's inequality (see, for instance, \cite[Chapter 3]{rudin1987real}). 
	Let $1 \leq p,q \leq \infty$ with $1/p + 1/q = 1$ and assume that $\omega \in L^q(\Omega)$. 
	Then we have 
	\begin{equation*} 
		\int_{\Omega} \| B \boldsymbol{\tau}(\boldsymbol{x}) \|_{\ell^1} \omega(\boldsymbol{x}) \intd \boldsymbol{x} 
			\leq \left( \int_{\Omega} \| B \boldsymbol{\tau}(\boldsymbol{x}) \|_{\ell^1}^p \intd \boldsymbol{x} \right)^{1/p} 
			\left( \int_{\Omega} \omega(\boldsymbol{x})^q \intd \boldsymbol{x} \right)^{1/q}.
	\end{equation*}  
	Hence, $\| B \boldsymbol{\tau}\|_{\ell^1}$ converging to zero in $L^p(\Omega)$ as $N \to \infty$ for some $p \geq 1$ immediately implies \eqref{eq:relation_weights}. 
	The special case of $p = \infty$ corresponds to \eqref{eq:observation_Bayona}. 
\end{remark}
\section{Numerical results} 
\label{sec:numerical} 

We present a variety of numerical tests in one and two dimensions to demonstrate our theoretical findings.  
A constant weight function $\omega \equiv 1$ is used for simplicity. 
All numerical tests presented here were generated in MATLAB\footnote{See \url{https://github.com/jglaubitz/stability_RBF_CFs}}.

\subsection{Compactly supported RBFs} 
\label{sub:num_compact}

Let us start with demonstrating Theorem \ref{thm:main} in one dimension. 
To this end, we consider Wendland's compactly supported RBFs in $\Omega = [0,1]$.

\begin{figure}[tb]
	\centering 
	\begin{subfigure}[b]{0.32\textwidth}
		\includegraphics[width=\textwidth]{%
      	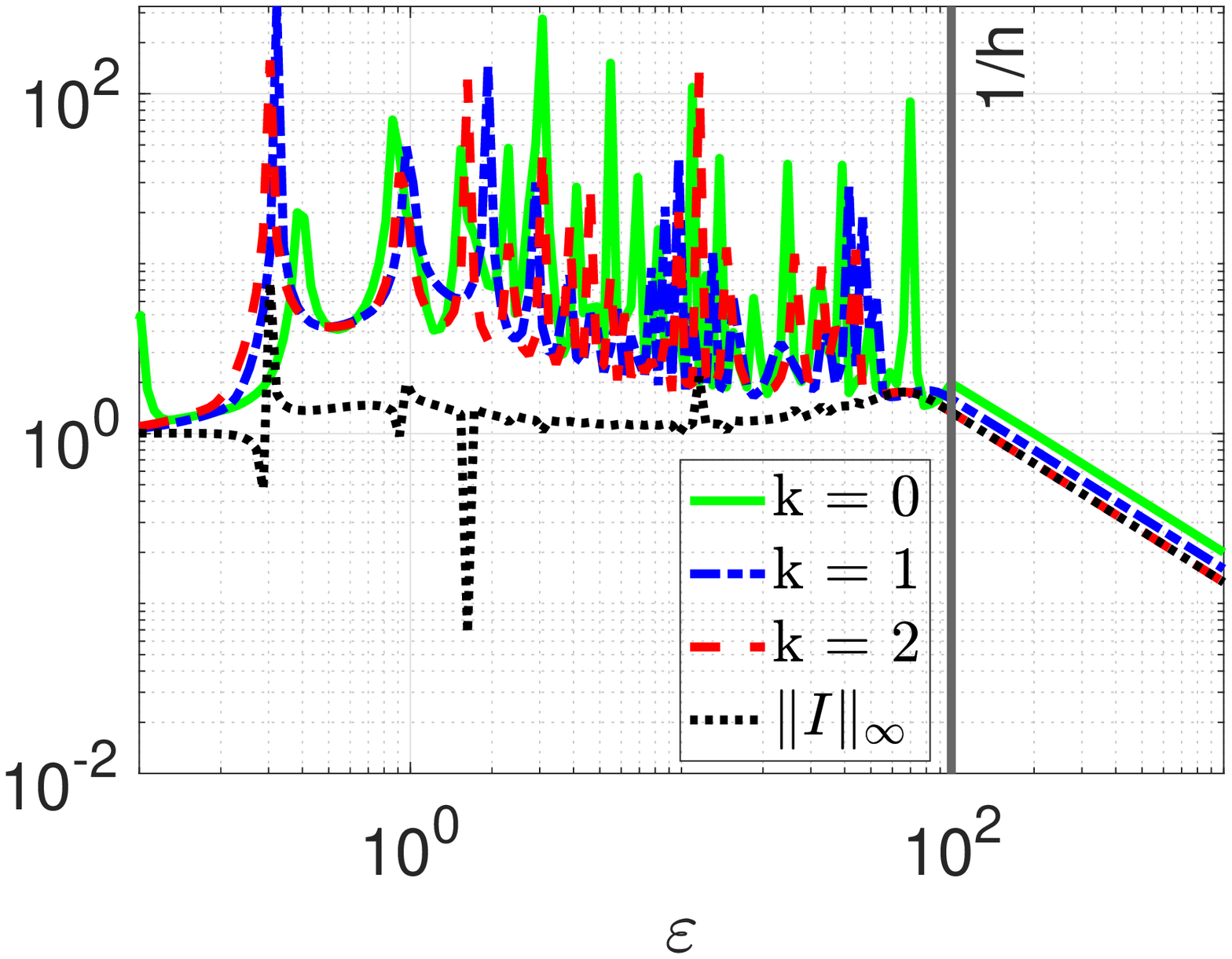} 
    \caption{$k=1$, $d=-1$ (pure RBF)}
    \label{fig:demo_noPol}
  	\end{subfigure}%
	~
  	\begin{subfigure}[b]{0.32\textwidth}
		\includegraphics[width=\textwidth]{%
      	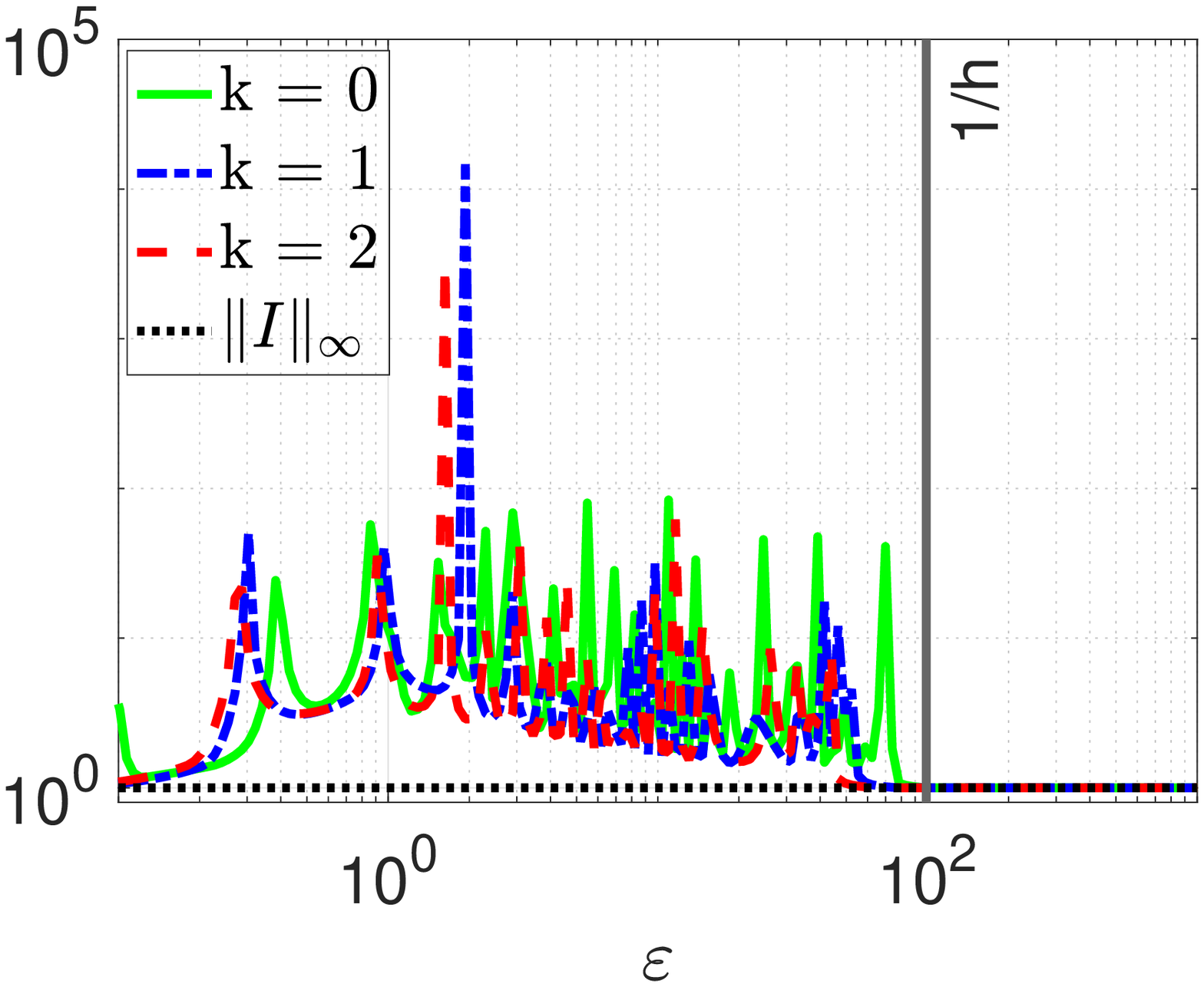} 
    \caption{$d=0$ (constant term)}
    \label{fig:demo_d0}
  	\end{subfigure}%
  	~ 
  	\begin{subfigure}[b]{0.32\textwidth}
		\includegraphics[width=\textwidth]{%
      	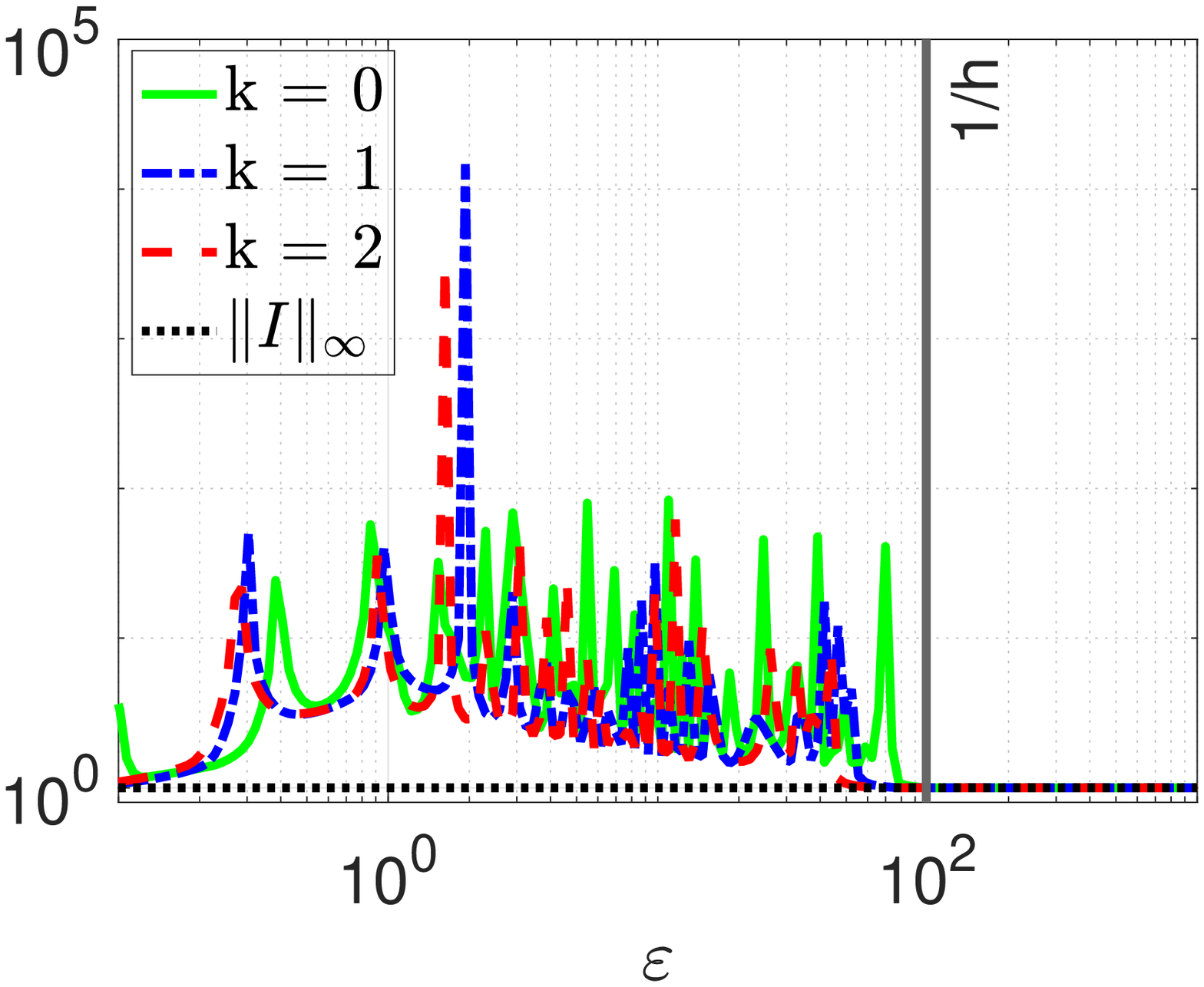} 
    \caption{$d=1$ (linear term)}
    \label{fig:demo_d1}
  	\end{subfigure}%
  	\caption{
  	The stability measure $\|C_N\|_{\infty}$ for Wendland's compactly supported RBF $\varphi_{1,k}$ with smoothness parameters $k=0,1,2$ on $N=100$ equidistant data points. 
	$1/h$ denotes the threshold above which the basis functions have nonoverlapping support. 
  	}
  	\label{fig:demo}
\end{figure} 

Figure \ref{fig:demo} illustrates the stability measure $\|C_N\|_{\infty}$ of Wendland's compactly supported RBF $\varphi_{1,k}$ with smoothness parameters $k=0,1,2$ as well as the optimal stability measure. 
The latter is given by $C_N[1]$ if no constants are included and by $\|I\|_\infty = 1$ if constants are included in the RBF approximation space. 
Furthermore, $N=100$ equidistant data points were used, including the end points, $x_1 = 0$ and $x_N = 1$, and the (reference) shape parameter $\varepsilon$ was allowed to vary. 
Finally, $1/h$ denotes the threshold above which the compactly supported RBFs have nonoverlapping support. 
We note that the RBF-QFs are stable for sufficiently small shape parameters. 
At the same time, we can also observe the RBF-QF be stable for $\varepsilon \geq 1/h$. 
It can be argued that this is in accordance with Theorem \ref{thm:main}. 
Recall that Theorem \ref{thm:main} essentially states that for $\varepsilon \geq 1/h$, and assuming that all basis functions have equal moments ($I[\varphi_n] = I[\varphi_m]$ for all $n,m$), the corresponding RBF-QF (including polynomials of any degree) is stable if a sufficiently large number of equidistribiuted data points is used. 
Here, the equal moments condition was ensured by choosing the shape parameter as $\varepsilon_n = \varepsilon$ for the interior data points ($n=2,\dots,N-1$) and as $\varepsilon_1 = \varepsilon_N = \varepsilon/2$ for the boundary data points. 

\begin{figure}[tb]
	\centering 
  	\begin{subfigure}[b]{0.45\textwidth}
		\includegraphics[width=\textwidth]{%
      	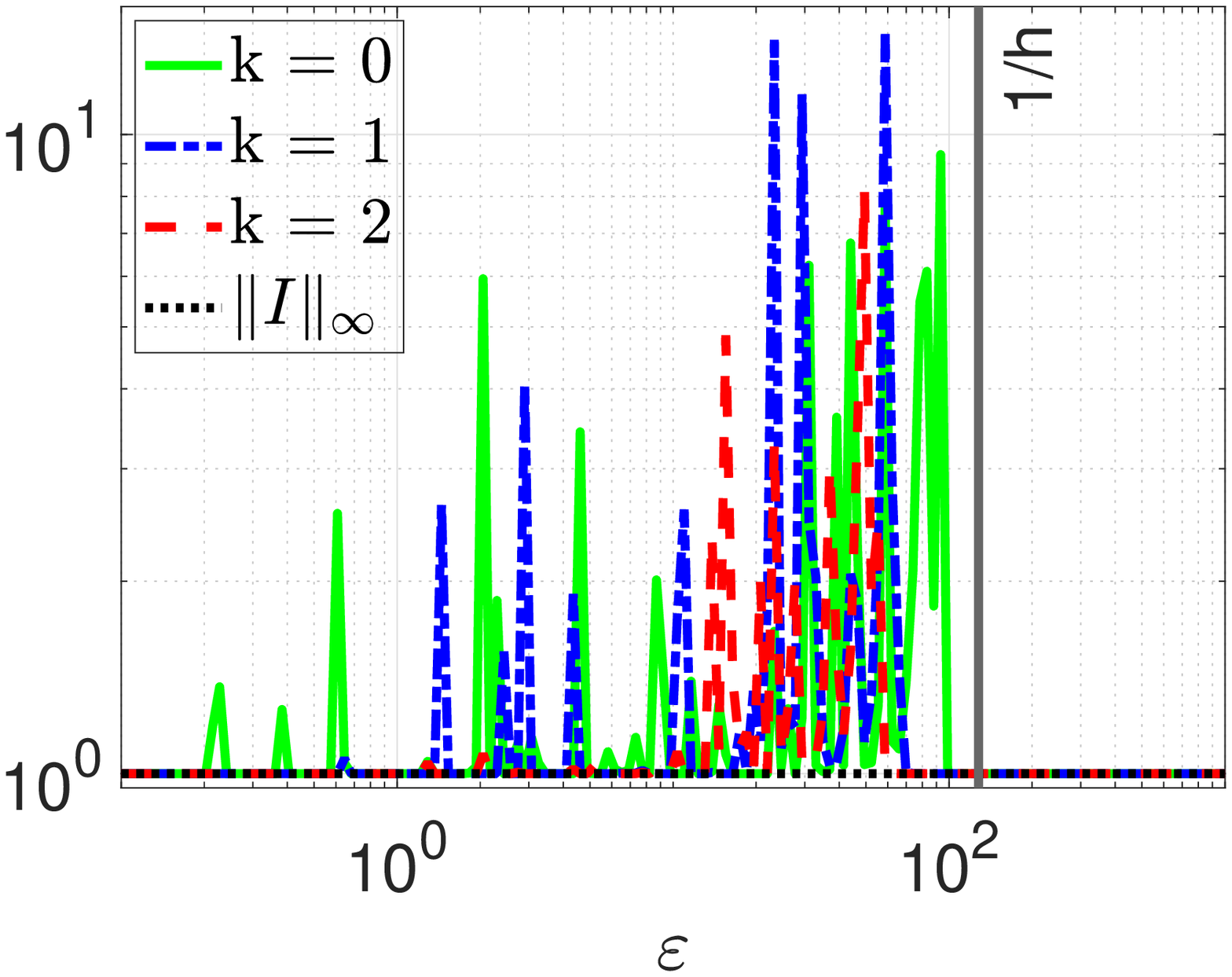} 
    \caption{$d=0$ (constant term)}
    \label{fig:nonequid_d0}
  	\end{subfigure}%
  	~ 
  	\begin{subfigure}[b]{0.45\textwidth}
		\includegraphics[width=\textwidth]{%
      	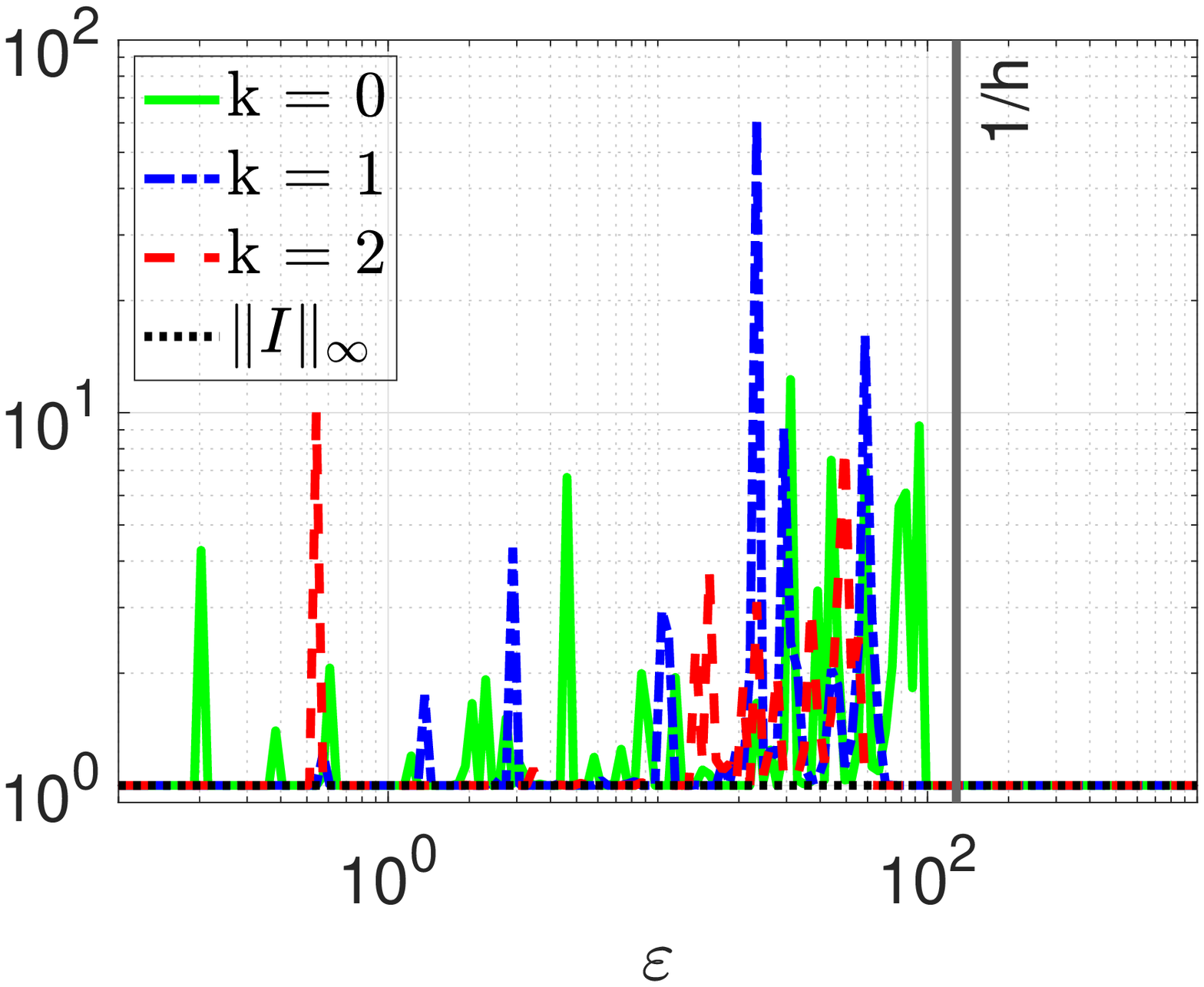} 
    \caption{$d=1$ (linear term)}
    \label{fig:nonequid_d1}
  	\end{subfigure}%
  	\caption{
  	The stability measure $\|C_N\|_{\infty}$ for Wendland's compactly supported RBF $\varphi_{1,k}$ with smoothness parameters $k=0,1,2$ on $N=100$ Halton points. 
	$1/h$ denotes the threshold above which the basis functions have nonoverlapping support.
  	}
  	\label{fig:nonequid}
\end{figure}

That said, at least numerically, we observe that it is possible to drop this equal moment condition. 
This is demonstrated by Figure \ref{fig:nonequid}. 
There, we perform the same test as in Figure \ref{fig:demo}, except choosing all the shape parameters to be equal ($\varepsilon_n = \varepsilon$, $n=1,\dots,N$) and going over to nonequidistant Halton points.  
Nevertheless, we can see in Figure \ref{fig:nonequid} that for $\varepsilon \geq 1/h$ the RBF-QFs are still stable. 

\begin{figure}[tb]
	\centering 
	\begin{subfigure}[b]{0.32\textwidth}
		\includegraphics[width=\textwidth]{%
      	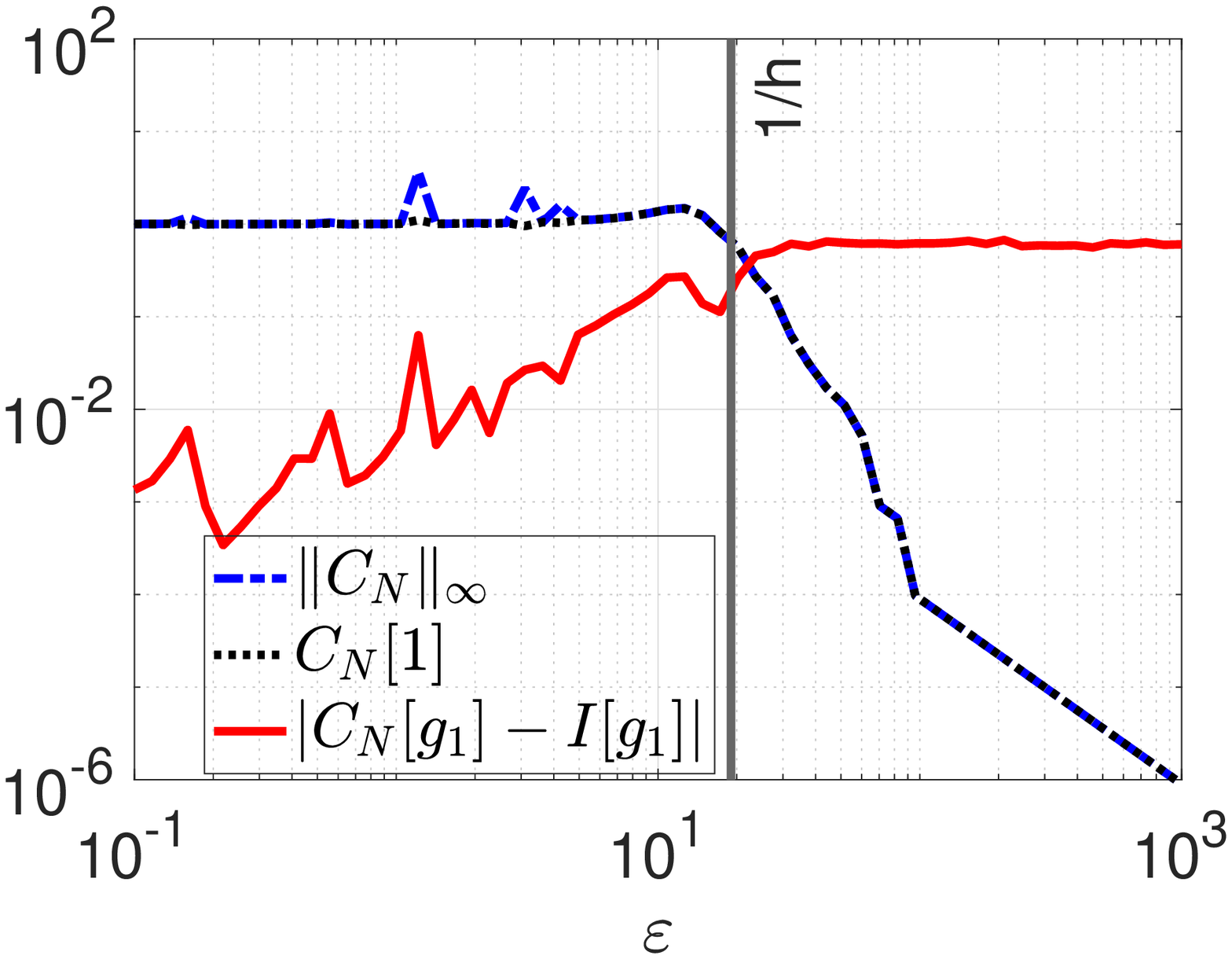} 
    \caption{Equidistant, $d=-1$}
    \label{fig:error_2d_Genz1_N400_equid_k1_noPol}
  	\end{subfigure}%
	~
	\begin{subfigure}[b]{0.32\textwidth}
		\includegraphics[width=\textwidth]{%
      	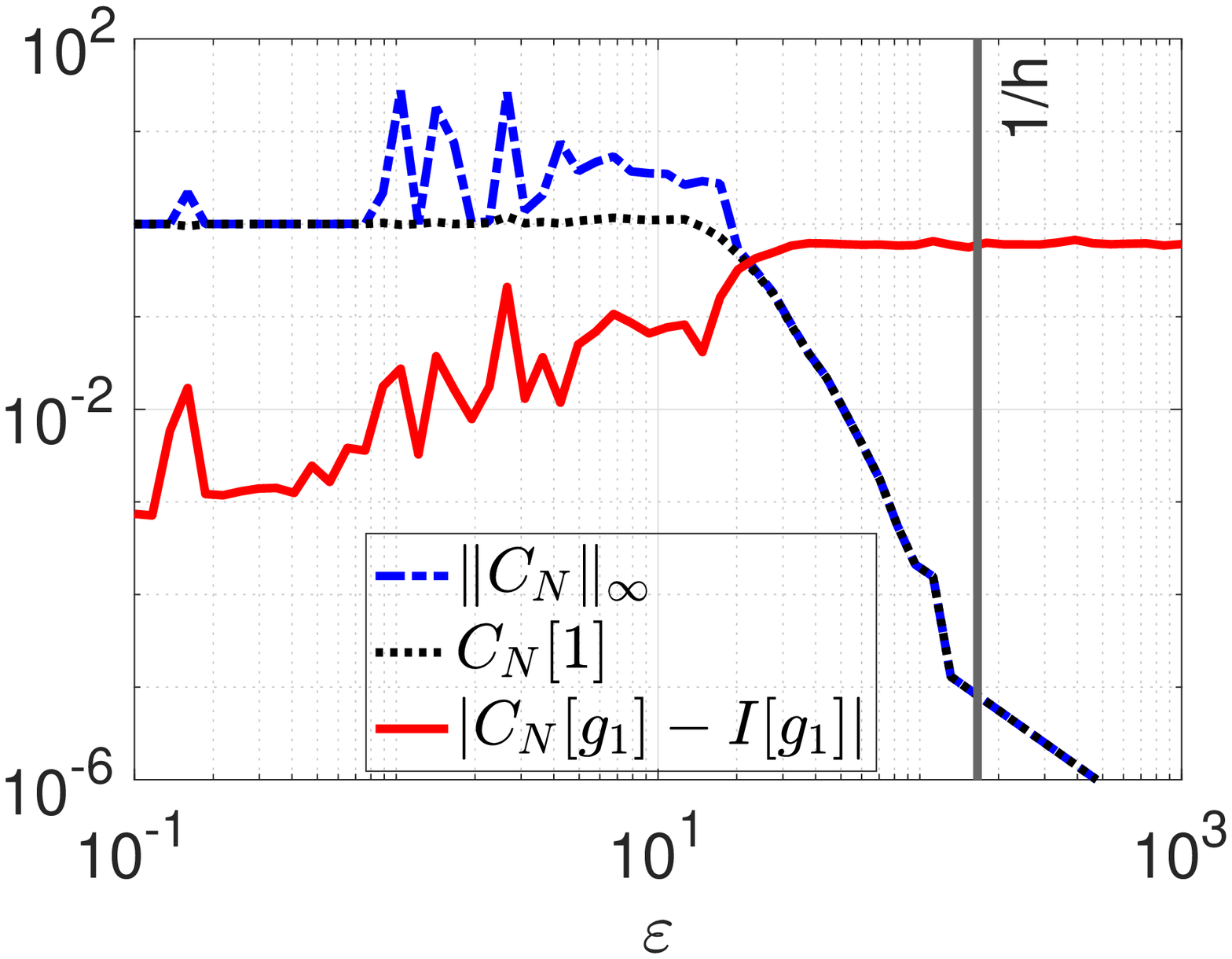} 
    \caption{Halton, $d=-1$}
    \label{fig:error_2d_Genz1_N400_Halton_k1_noPol}
  	\end{subfigure}%
	~ 
	\begin{subfigure}[b]{0.32\textwidth}
		\includegraphics[width=\textwidth]{%
      	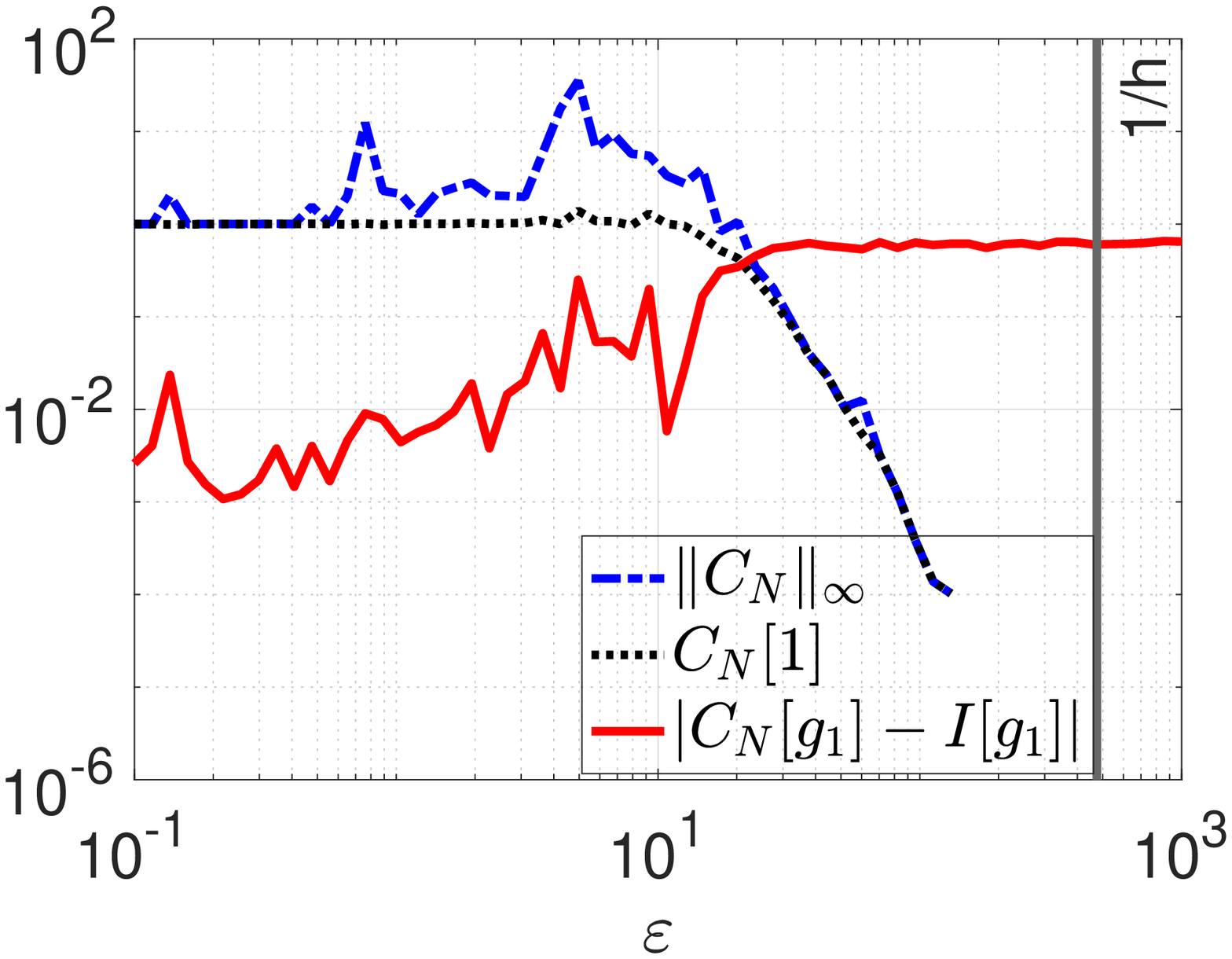} 
    \caption{Random, $d=-1$}
    \label{fig:error_2d_Genz1_N400_random_k1_noPol}
  	\end{subfigure}%
	\\
  	\begin{subfigure}[b]{0.32\textwidth}
		\includegraphics[width=\textwidth]{%
      	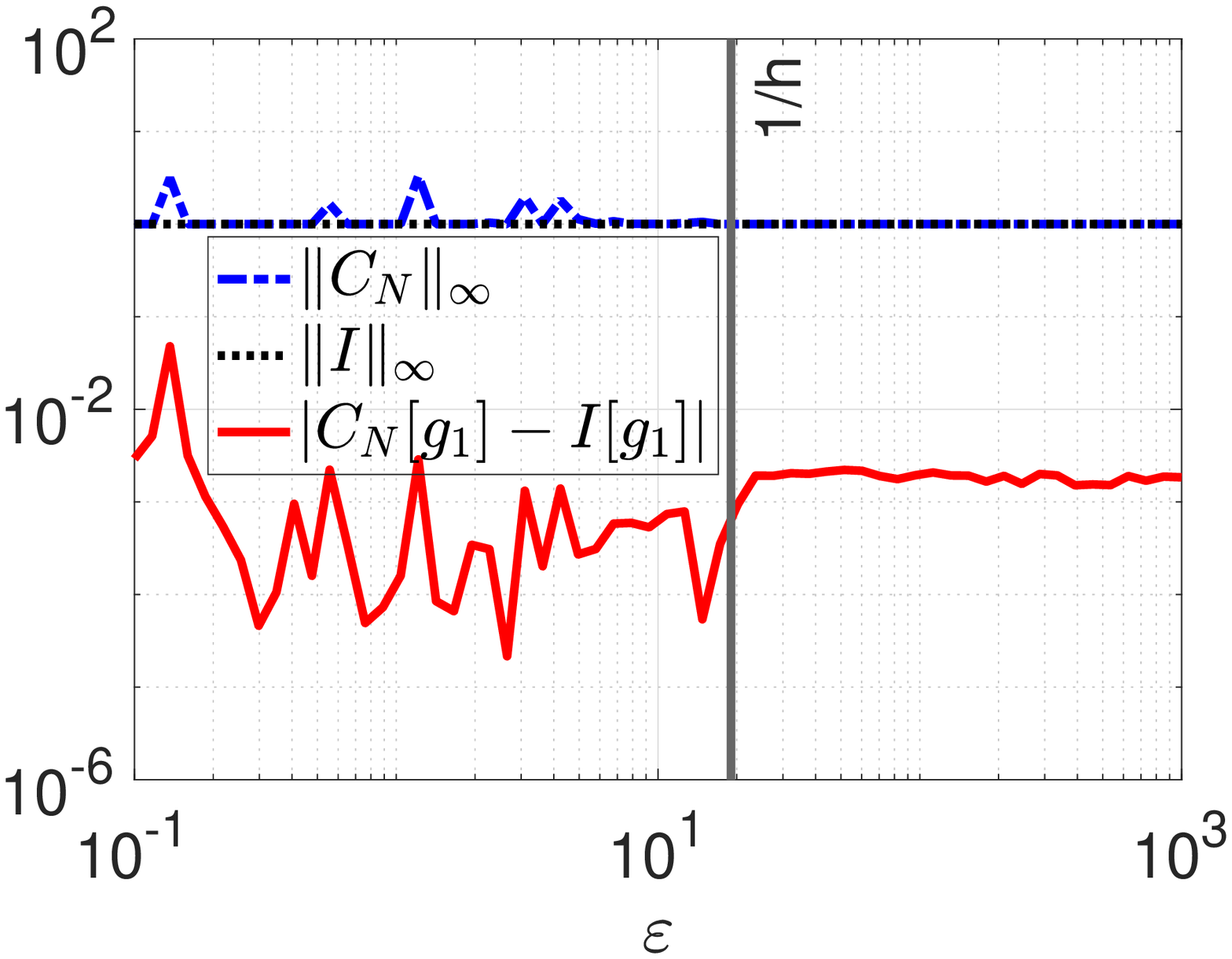} 
    \caption{Equidistant, $d=0$}
    \label{fig:error_2d_Genz1_N400_equid_k1_d0}
  	\end{subfigure}%
	~
	\begin{subfigure}[b]{0.32\textwidth}
		\includegraphics[width=\textwidth]{%
      	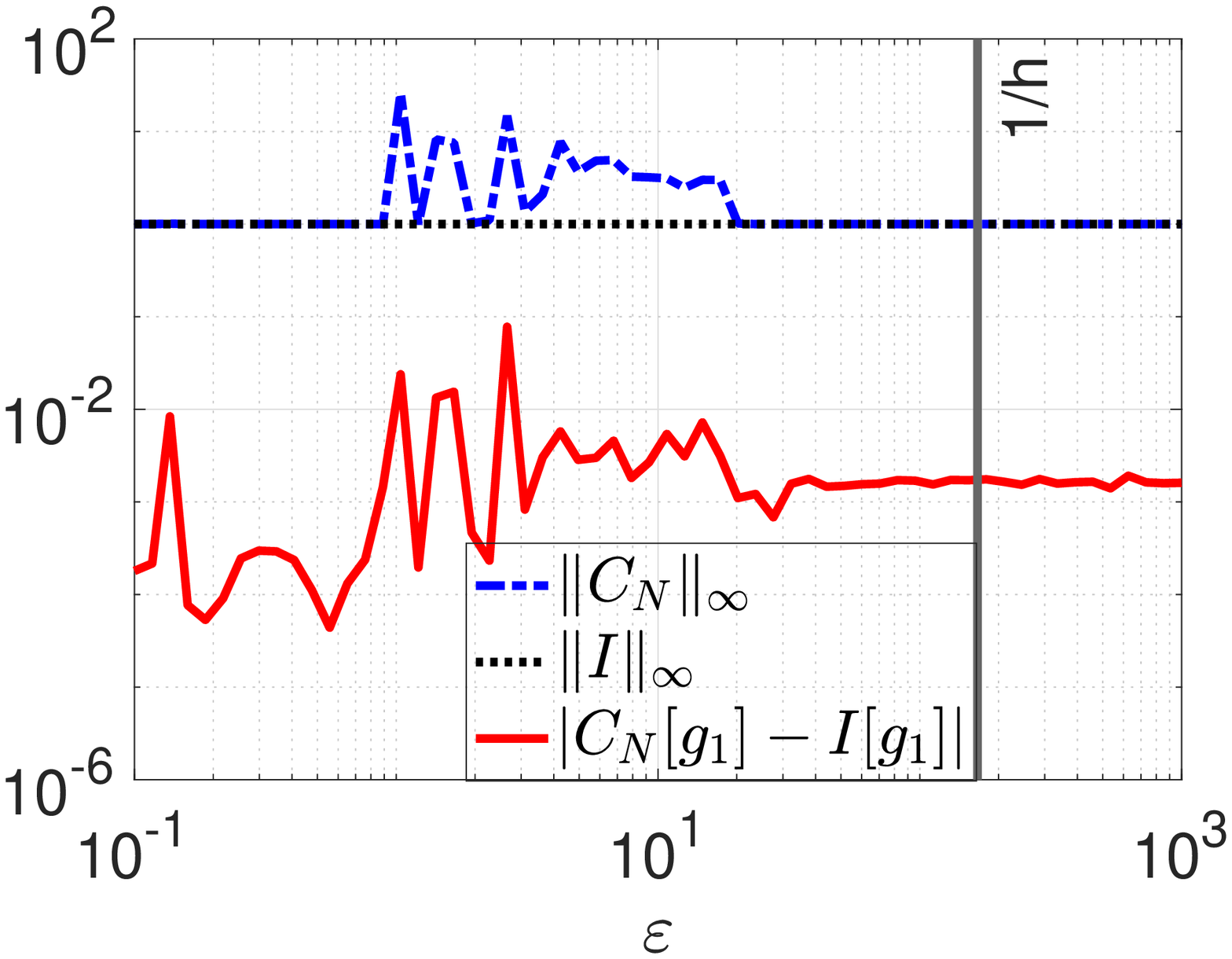} 
    \caption{Halton, $d=0$}
    \label{fig:error_2d_Genz1_N400_Halton_k1_d0}
  	\end{subfigure}%
	~ 
	\begin{subfigure}[b]{0.32\textwidth}
		\includegraphics[width=\textwidth]{%
      	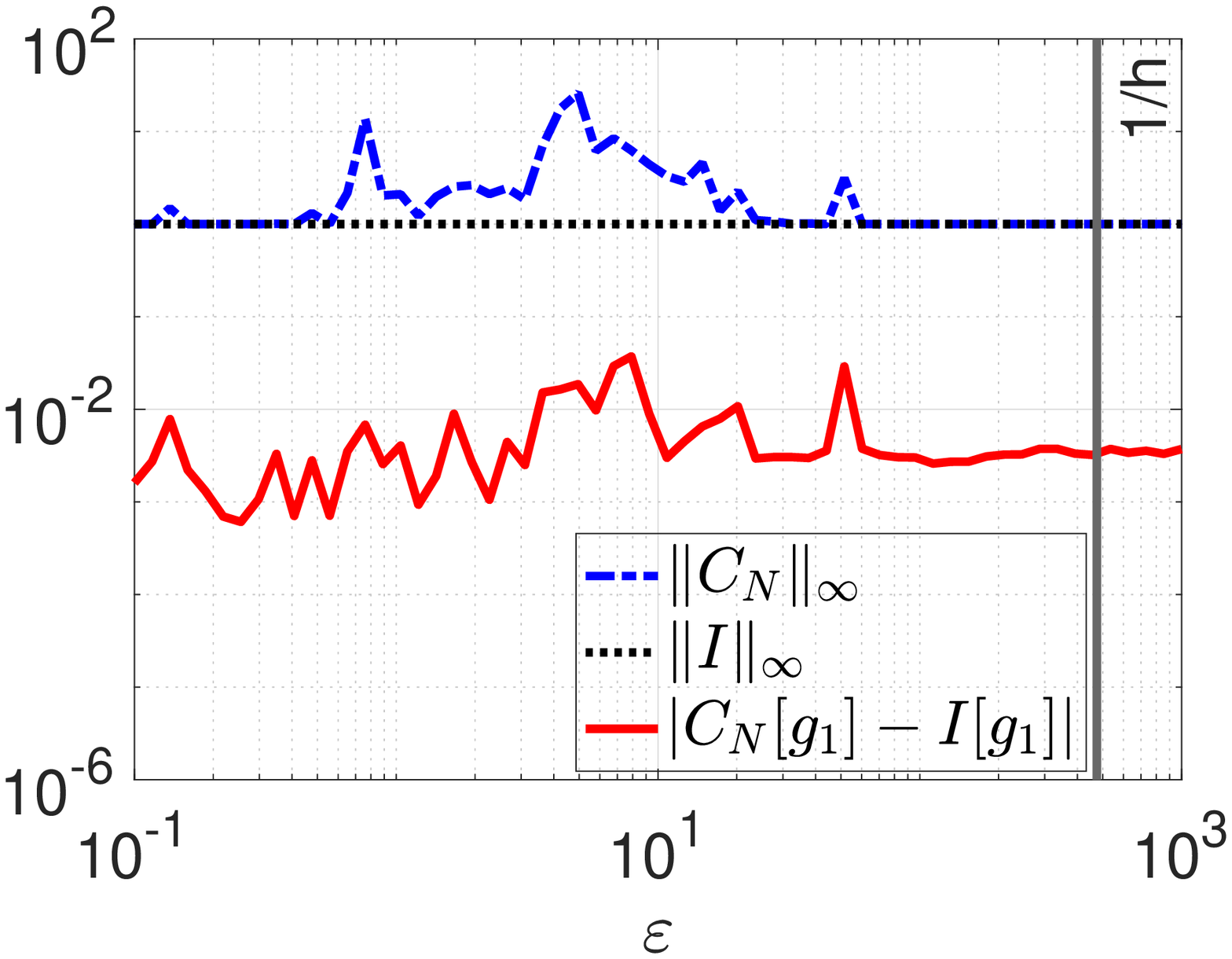} 
    \caption{Random, $d=0$}
    \label{fig:error_2d_Genz1_N400_random_k1_d0}
  	\end{subfigure}%
	\\
	\begin{subfigure}[b]{0.32\textwidth}
		\includegraphics[width=\textwidth]{%
      	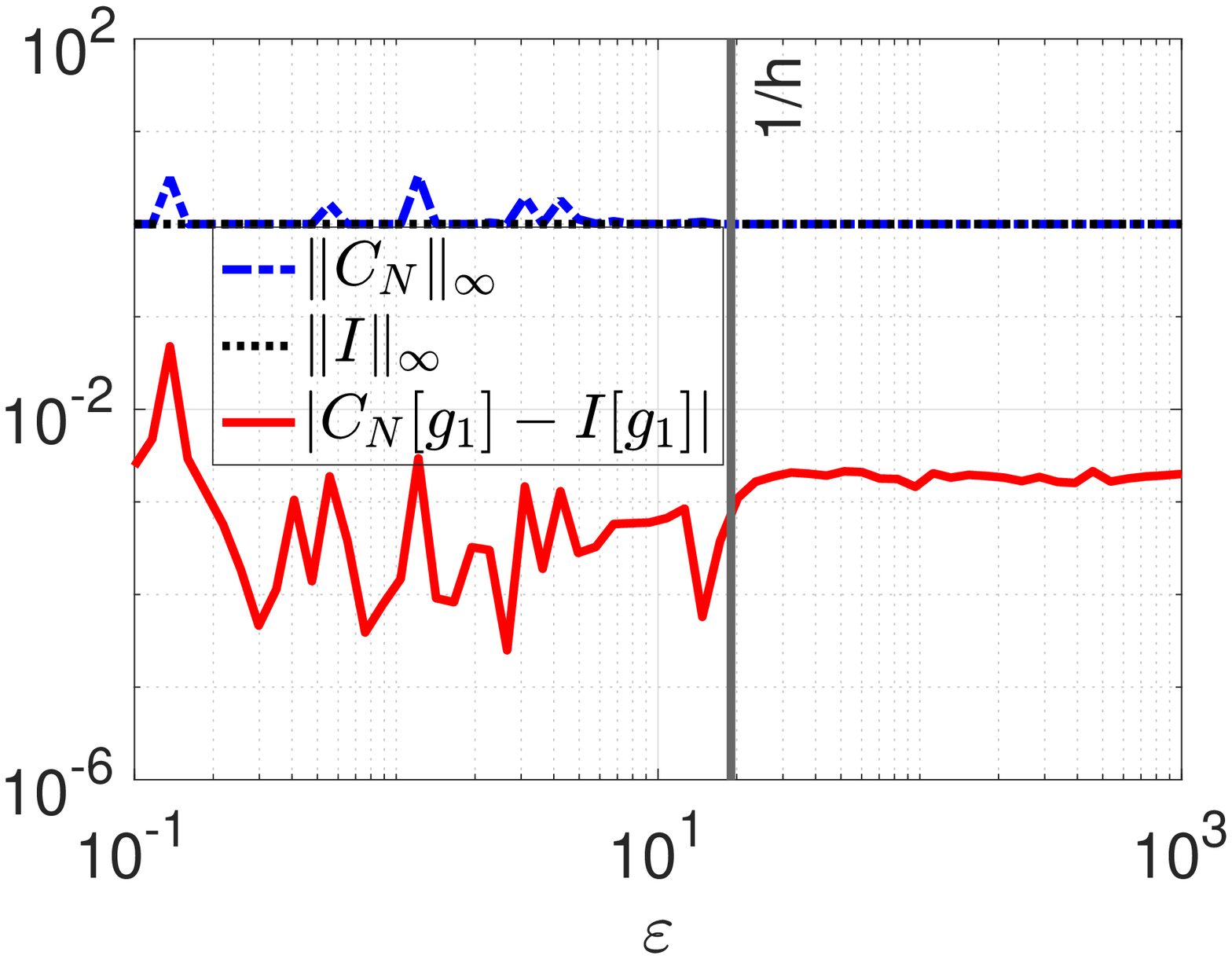} 
    \caption{Equidistant, $d=1$}
    \label{fig:error_2d_Genz1_N400_equid_k1_d1}
  	\end{subfigure}%
	~
	\begin{subfigure}[b]{0.32\textwidth}
		\includegraphics[width=\textwidth]{%
      	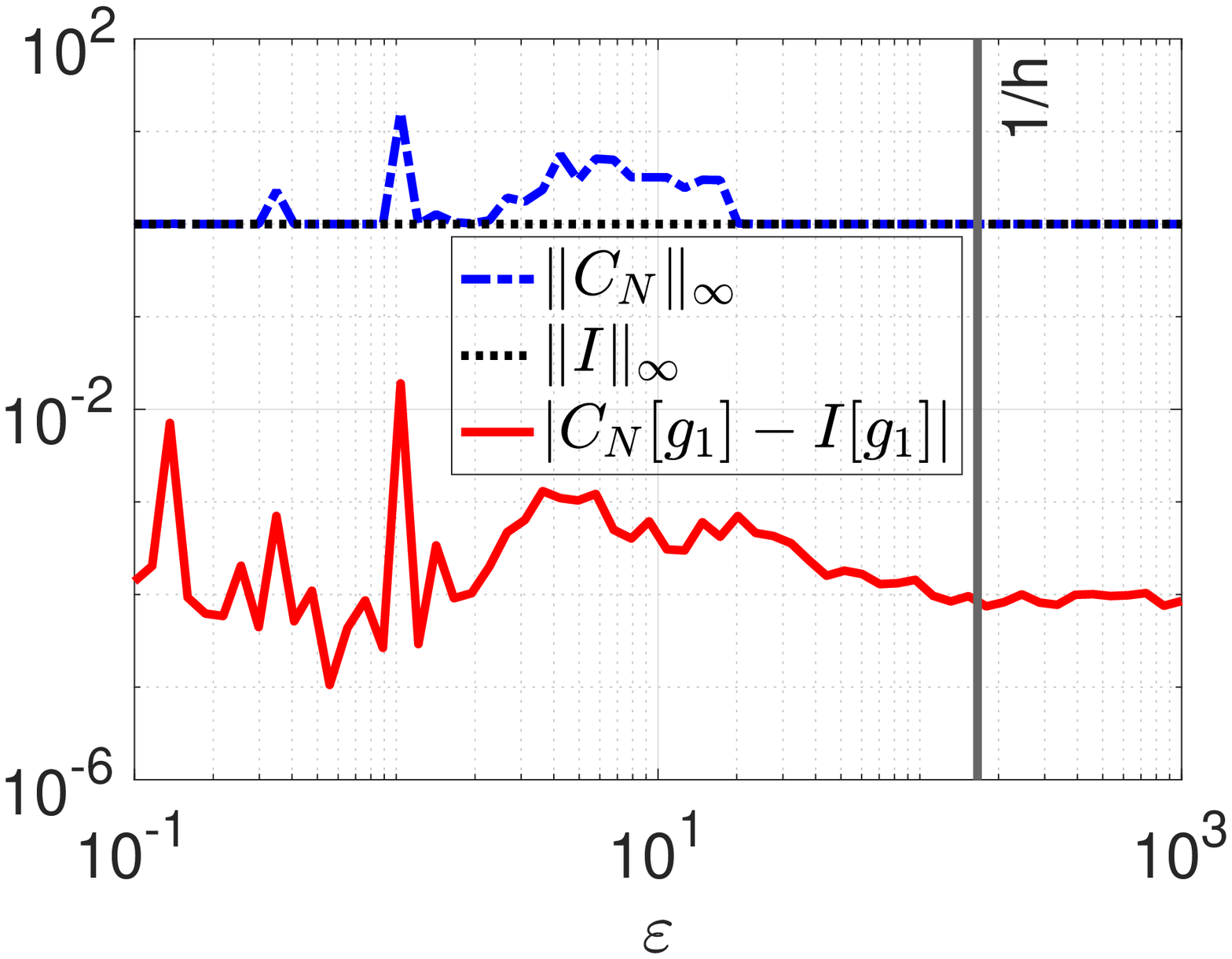} 
    \caption{Halton, $d=1$}
    \label{fig:error_2d_Genz1_N400_Halton_k1_d1}
  	\end{subfigure}%
	~ 
	\begin{subfigure}[b]{0.32\textwidth}
		\includegraphics[width=\textwidth]{%
      	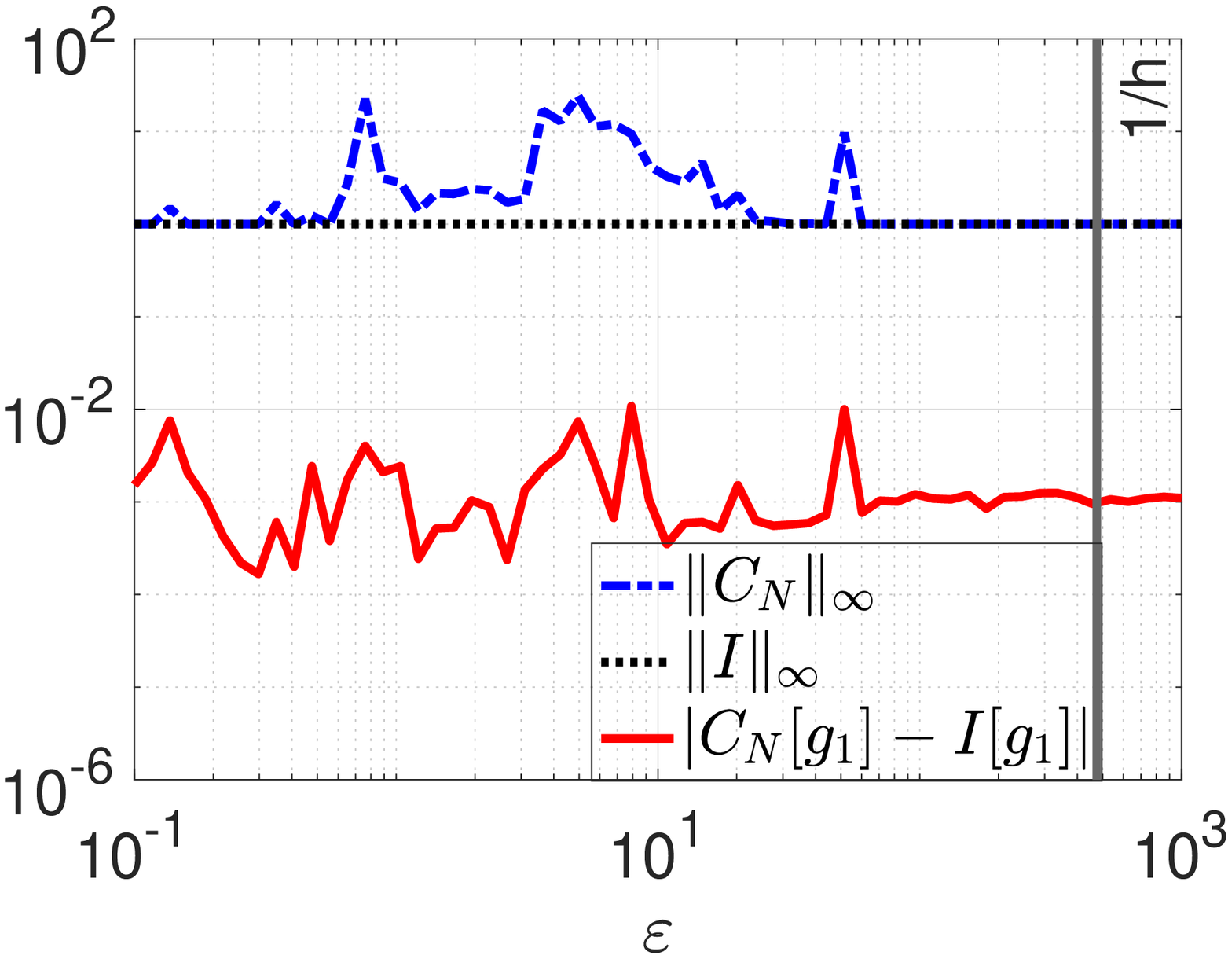} 
    \caption{Random, $d=1$}
    \label{fig:error_2d_Genz1_N400_random_k1_d1}
  	\end{subfigure}%
  	\caption{
	Error analysis for Wendland's compactly supported RBF $\varphi_{2,k}$ and the first Genz test function $g_1$ on $\Omega = [0,1]^2$; see \eqref{eq:Genz}.  
	In all cases, $N=400$ data points (equidistant, Halton, or random) were considered. 
	$1/h$ denotes the threshold above which the basis functions have nonoverlapping support.
  	}
  	\label{fig:error_2d_Genz1_k1}
\end{figure}

Next, we extend our numerical tests to the following Genz test functions \cite{genz1984testing} (also see \cite{van2020adaptive}) on $\Omega = [0,1]^q$: 
\begin{equation}\label{eq:Genz}
\begin{aligned}
	& g_1(\boldsymbol{x}) 
			= \cos\left( 2 \pi b_1 + \sum_{i=1}^q a_i x_i \right), \quad 
	&& g_2(\boldsymbol{x}) 
			= \prod_{i=1}^q \left( a_i^{-2} + (x_i - b_i)^2 \right)^{-1} ,\\
	&g_3(\boldsymbol{x}) 
			= \left( 1 + \sum_{i=1}^q a_i x_i \right)^{-(q+1)}, \quad 
	&& g_4(\boldsymbol{x}) 
			= \exp \left( - \sum_{i=1}^q a_i^2 ( x_i - b_i )^2 \right)
\end{aligned}
\end{equation} 
Here, $q$ denotes the dimension under consideration and is henceforth chosen as $q=2$. 
These functions are designed to have different complex characteristics for numerical integration routines.
The vectors $\mathbf{a} = (a_1,\dots,a_q)^T$ and $\mathbf{b} = (b_1,\dots,b_q)^T$ respectively contain (randomly chosen) shape and translation parameters. 
For each case, the experiment was repeated $100$ times. 
At the same time, for each experiment, the vectors $\mathbf{a}$ and $\mathbf{b}$ were drawn randomly from $[0,1]^2$. 
For reasons of space, we only report the results for $g_1$ and $k=1$ in Figure \ref{fig:error_2d_Genz1_k1}. 
As before, the smallest errors are found for shape parameters corresponding to the stable RBF-QF.  
The results for $g_2, g_3, g_4$ and $k=0,2$ were similar and are therefore not reported here.  
Since it might be hard to identify the smallest errors as well as the corresponding shape parameter and stability measure from Figure \ref{fig:error_2d_Genz1_k1}, these are listed in Table \ref{tab:min_error_Wendland} for $d=0,1$ together with the corresponding values for the fourth Genz test function $g_4$. 

\begin{table}[htb]
\renewcommand{\arraystretch}{1.4}
\centering 
  	\begin{tabular}{c c c c c c c c c} 
	\toprule 
	& & \multicolumn{3}{c}{$g_1$} & & \multicolumn{3}{c}{$g_4$} \\ \hline 
	& & $e_{\text{min}}$ & $\varepsilon$ & $\|C_N\|_{\infty}$ 
	& & $e_{\text{min}}$ & $\varepsilon$ & $\|C_N\|_{\infty}$ \\ \hline 
	& & \multicolumn{7}{c}{Equidistant Points} \\ \hline 
	$d=0$ 	& & 2.2e-05 & 2.6e+00 & 1.0e+00 
				& & 6.1e-05 & 2.6e+00 & 1.0e+00 \\ 
	$d=1$ 	& & 2.2e-05 & 2.6e+00 & 1.0e+00 			
				& & 6.2e-05 & 2.6e+00 & 1.0e+00 \\ \hline 
	& & \multicolumn{7}{c}{Halton Points} \\ \hline 
	$d=0$ 	& & 4.7e-05 & 5.5e-01 & 1.0e+00 
				& & 1.9e-05 & 5.5e-01 & 1.0e+00 \\ 
	$d=1$ 	& & 1.1e-05 & 5.5e-01 & 1.0e+00 
				& & 1.6e-05 & 5.5e-01 & 1.0e+00 \\ \hline 
	& & \multicolumn{7}{c}{Random Points} \\ \hline 
	$d=0$ 	& & 6.0e-04 & 2.5e-01 & 1.0e+00 
				& & 1.6e-04 & 2.9e-01 & 1.0e+00 \\ \hline 
	$d=1$ 	& & 2.2e-04 & 4.0e-01 & 1.0e+00 
				& & 1.7e-04 & 4.0e-01 & 1.0e+00 \\
	\bottomrule
\end{tabular} 
\caption{Minimal errors, $e_{\text{min}}$, for the first and fourth Genz test function, $g_1$ and $g_4$, together with the corresponding shape parameter, $\varepsilon$, and stability measure, $\|C_N\|_{\infty}$. 
	Wendland's compactly supported RBF with smoothness parameter $k=1$ was used in all cases.}
\label{tab:min_error_Wendland}
\end{table}

We see in Figure \ref{fig:error_2d_Genz1_k1} that for $d=-1$ and increasing $\varepsilon$, the error increases. 
This is because the supports of the translated kernels (disks in 2d) become smaller, resulting in ``holes" in the pure RBF interpolant, i.\,e., regions where it is zero. 
In Figure \ref{fig:error_2d_Genz1_N400_random_k1_noPol}, for random points and $d=-1$, the supports become so small that all the quadrature weights become zero.
The holes vanish if at least a constant is included in the RBF interpolant, which explains the reduced errors for the same value of $\varepsilon$ when $d=0$ or $d=1$. 
Finally, even for nonoverlapping supports ($\varepsilon = 1/h$), the area of the holes in the pure RBF part of the interpolant can converge to zero\footnote{Assuming the sequence of points is dense in $\Omega$} as $N \to \infty$. 

\begin{remark} 
	 If the RBF interpolant $s_{N,d}f$ convergences to $f$ in $L^1(\Omega)$ as $N \to \infty$, we get 
	 \begin{equation*}
	 	\left| C_N[f] - I[f] \right| 
			= \left| I[s_{N,d}f] - I[f] \right| 
			\leq \int_{\Omega} | (s_{N,d}f)(\boldsymbol{x}) - f(\boldsymbol{x}) | \intd \boldsymbol{x} 
			\to 0, \quad N \to \infty,
	 \end{equation*} 
	 and therefore $C_N[f] \to I[f]$ as $N \to \infty$. 
	 For convergence results of RBF interpolants, we refer to the monographs \cite{buhmann2000radial,wendland2004scattered,fasshauer2007meshfree}. 
	 That said, we point out that the convergence of $s_{N,d}f$ to $f$ depends on the area not covered by the supports. 
	 Let us denote the area that is covered by the supports by $\Omega_{\rm supp(\varphi,X_N,\varepsilon)}$, then the area that is not covered by the supports is $\Omega \setminus \Omega_{\rm supp(\varphi,X_N,\varepsilon)}$. 
	 A rough but simple lower bound for the $L^1(\Omega)$-error of $s_{N,d}f$ and $f$ is as follows. 
	 Note that the RBF interpolant is zero on $\Omega \setminus \Omega_{\rm supp(\varphi,X_N,\varepsilon)}$ and thus 
	 \begin{equation}\label{eq:lower_bound_holes}
	 	\int_{\Omega} | (s_{N,d}f)(\boldsymbol{x}) - f(\boldsymbol{x}) | \intd \boldsymbol{x} 
			\geq \int_{\Omega \setminus \Omega_{\rm supp(\varphi,X_N,\varepsilon)}} | f(\boldsymbol{x}) | \intd \boldsymbol{x},
	 \end{equation}
	 where the right-hand side is the average absolute value of $f$ in the area not covered by the supports.  
	 \eqref{eq:lower_bound_holes} indicates that convergence includes the rate with which ``holes" go to zero. 
\end{remark}
	 
In Figures \ref{fig:area_holes_noPol} and \ref{fig:area_holes_constants}, we relate the error in computing the integral of $g_1$ on $\Omega = [0,1]^2$ to the portion of the domain that is not covered by the possibly overlapping supports of the Wendland functions. 
For equidistant points, the horizontal/vertical distance between adjacent points is (i.e. $\mathbf{x}_{ij}$ and $\mathbf{x}_{(i\pm1)j}$ or $\mathbf{x}_{ij}$ and $\mathbf{x}_{i(j\pm1)}$) is $1/(\sqrt{N}-1)$, so as long as $\varepsilon>2(\sqrt{N}-1)$ the circles do not overlap and the total area that is not covered by the supports is given by
\begin{align}
    1-\frac{\pi}{\varepsilon^{2}}(\sqrt{N}-1)^{2}.\nonumber
\end{align}
Once $\varepsilon\leq2(\sqrt{N}-1)$, the circles overlap, and the total area that is not covered by the supports becomes
\begin{align}
    1-\frac{\pi}{\varepsilon^{2}}(\sqrt{N}-1)^{2}+\frac{2(\theta-\mbox{sin}(\theta))}{\varepsilon^{2}}(\sqrt{N}-1)^{2},\nonumber
\end{align}
where
\begin{align*}
    \theta=2\mbox{sin}^{-1}\left(\frac{\sqrt{4(\sqrt{N}-1)^{2}-\varepsilon^{2}}}{2(\sqrt{N}-1)}\right).\nonumber
\end{align*}
Finally, when $\varepsilon\leq\sqrt{2}(\sqrt{N}-1)$, the area that is not covered by the supports is 0.  
Figure \ref{fig:area_holes_noPol} illustrates the error (left frame) and the area not covered by the supports (right frame) in this situation for various $N$ and $\varepsilon$. 
The dashed lines represent the cases $\varepsilon=2(\sqrt{N}-1)$ and $\varepsilon=\sqrt{2}(\sqrt{N}-1)$. 
On the other hand, Figure \ref{fig:area_holes_constants} illustrates the same test for an RBF interpolant that includes a constant. 
It demonstrates the improvement when the constant basis element covers the holes.

\begin{figure}[tb]
	\centering 
	\begin{subfigure}[b]{0.95\textwidth}
		\includegraphics[width=\textwidth]{%
      	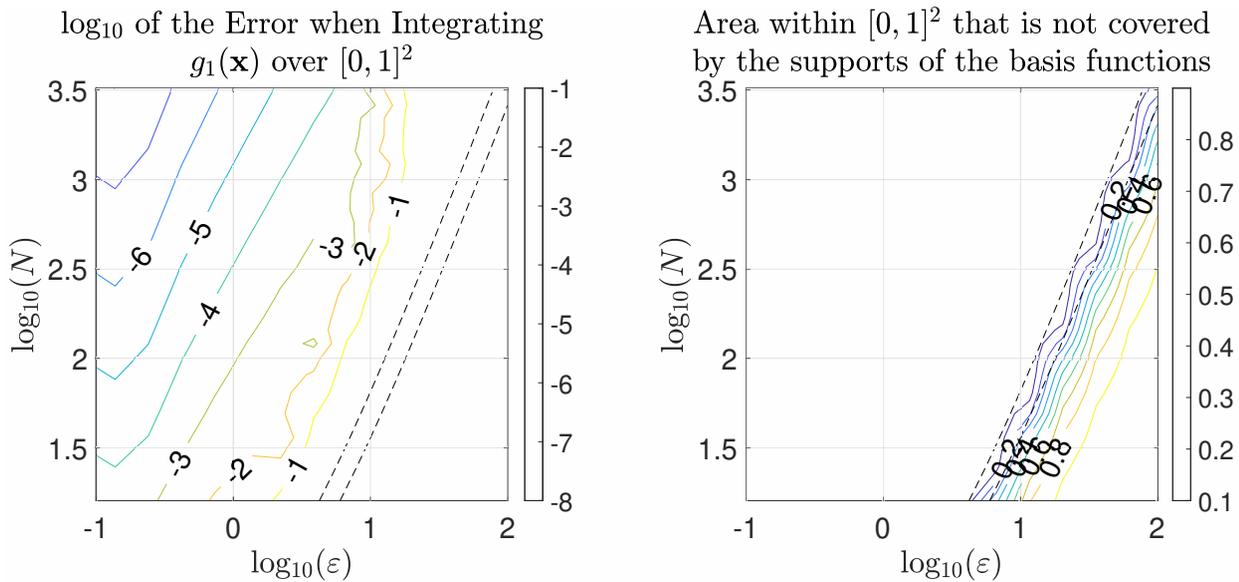} 
  	\end{subfigure}%
  	\caption{Error of $C_N[g_1]$ on $\Omega = [0,1]^2$ and the area not covered by the supports of the Wendland functions for various $N$ and $\varepsilon$ when no constant is included in the RBF interpolant}
    \label{fig:area_holes_noPol}
\end{figure} 

\begin{figure}[tb]
	\centering 
	\begin{subfigure}[b]{0.95\textwidth}
		\includegraphics[width=\textwidth]{%
      	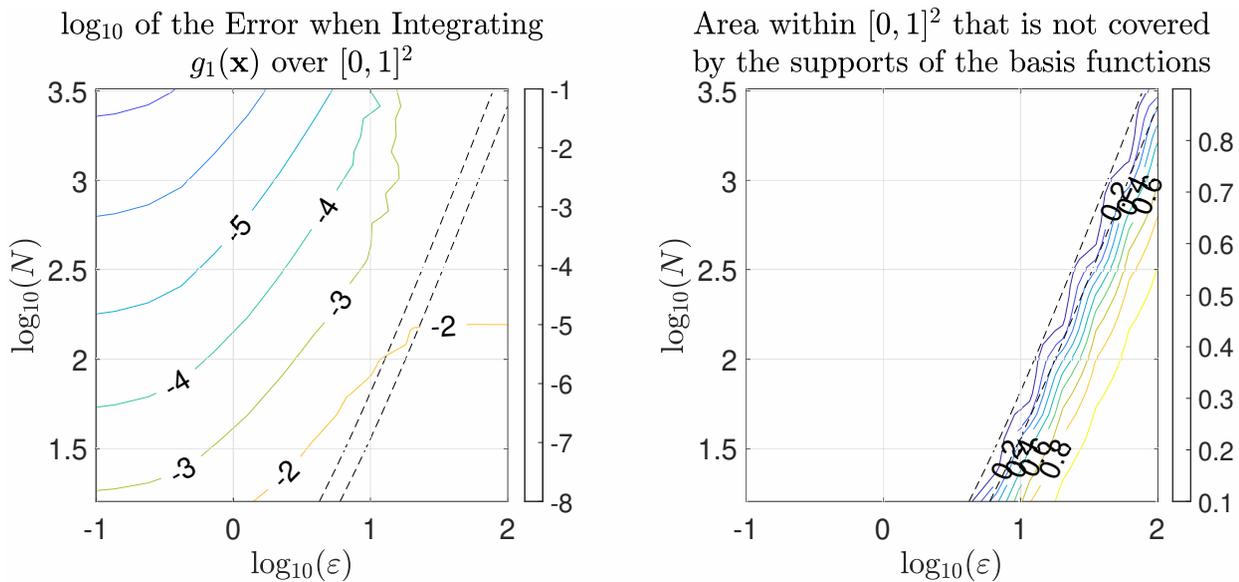} 
  	\end{subfigure}%
  	\caption{Error of $C_N[g_1]$ on $\Omega = [0,1]^2$ and the area not covered by the supports of the Wendland functions for various $N$ and $\varepsilon$ when a constant is included in the RBF interpolant}
    \label{fig:area_holes_constants}
\end{figure}


\subsection{Stable LSRBF-QFs} 
\label{sub:num_LSRBF}

We demonstrate that the least-squares approach discussed in \S5 can stabilize RBF-QFs. 
We repeat that stable LSRBF-QF can be constructed for any RBF function space as long as we are willing to oversample, i.\,e., the number of data points used by the quadrature is larger than the dimension of the RBF function space. 
In other words, there are more data points than center points. 
Notably, oversampling was used in some recent works \cite{tominec2021least,tominec2021stability,glaubitz2022energy} to stabilize RBF methods for partial differential equations, and it would be of interest to combine this with LSRBF-QF in future works.
Here, we demonstrate the possible advantage of LSRBF-QFs compared to interpolatory RBF-QFs (data and center points are the same) for 
the RBF function space spanned by a constant and the functions $\varphi(\varepsilon \|\boldsymbol{x} - \mathbf{y}_m\|)$ on $\Omega = [0,1]^2$ using a Gaussian kernel $\varphi(r) = \exp(-r^2)$ and a constant (independent of the center and data points) shape parameter $\varepsilon = 0.8$. 
The center and data points, $\{ \mathbf{y}_m \}_{m=1}^M$ and $\{\mathbf{x}_n\}_{n=1}^N$, are chosen as the first $M$ and $N$ elements of the same sequence of random points, respectively. 

\begin{figure}[tb]
	\centering
  	\begin{subfigure}[b]{0.45\textwidth}
		\includegraphics[width=\textwidth]{%
      		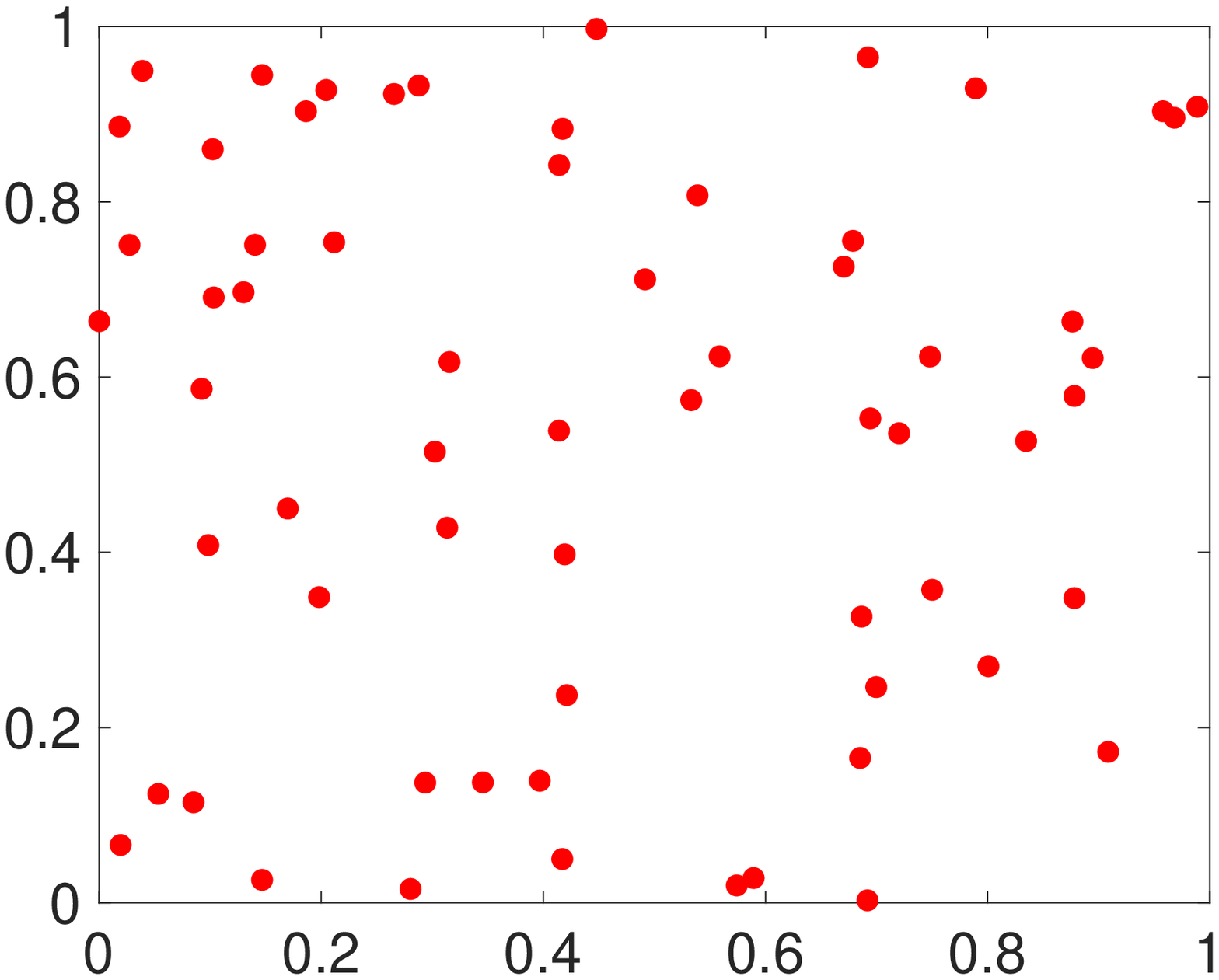} 
    		\caption{Random points}
    		\label{fig:points_random}
  	\end{subfigure}%
	~
	\begin{subfigure}[b]{0.45\textwidth}
		\includegraphics[width=\textwidth]{%
      		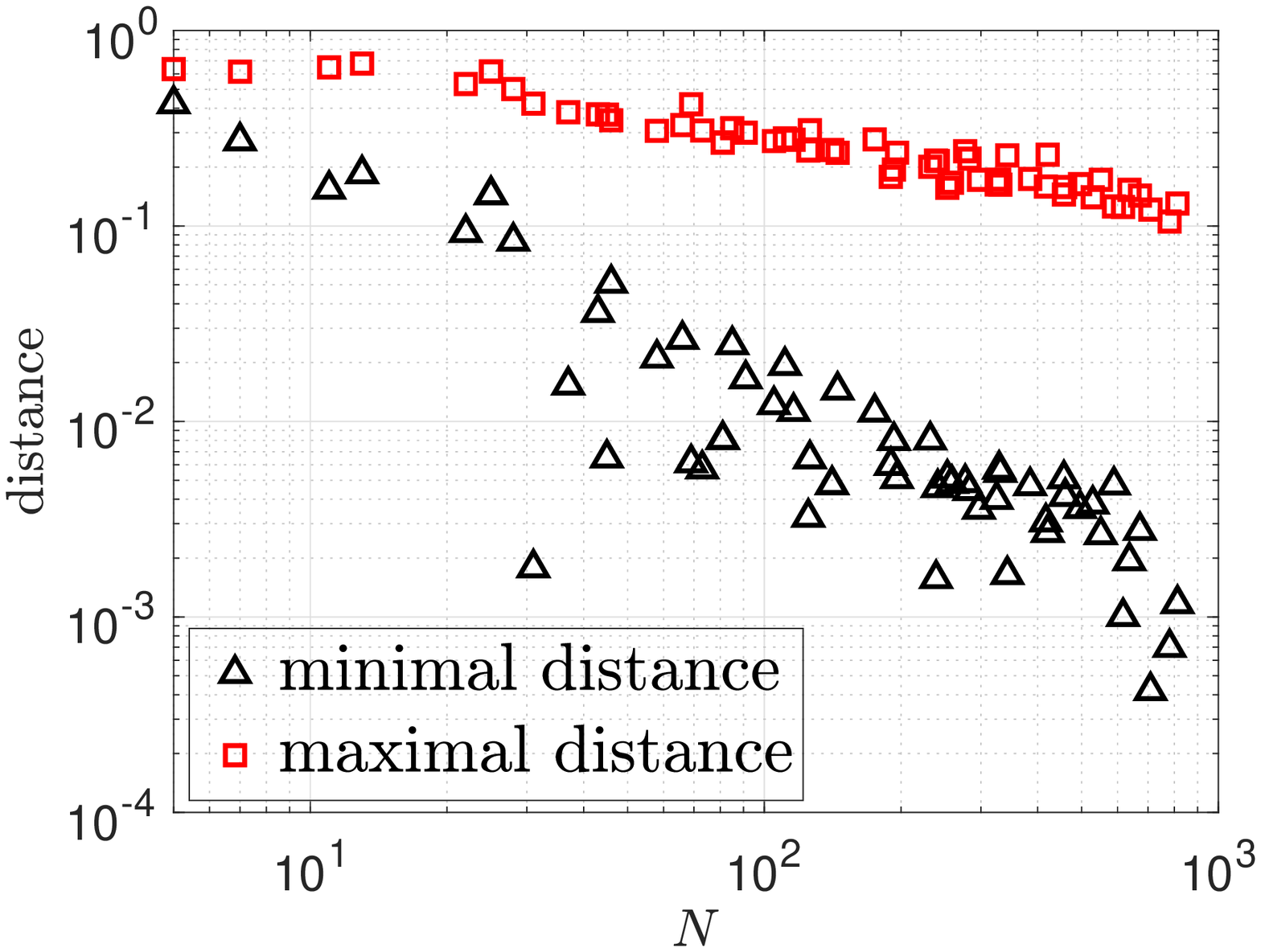} 
    		\caption{Minimal and maximal distance}
    		\label{fig:distance_random}
  	\end{subfigure}%
  	\caption{
		Random points and the corresponding minimal and maximal distance, \eqref{eq:min_distance} and \eqref{eq:max_distance}. 
		The minimal distance is the smallest distance between any two distinct points. 
		The maximal distance is the largest distance between any point and the closest distinct point.  
	}
  	\label{fig:LSRBF_points}
\end{figure} 

Figure \ref{fig:points_random} illustrates the first $64$ random points from the sequence used in our tests. 
Figure \ref{fig:distance_random} visualizes the minimal and maximal distance of the random data point set $X_N = \{ \mathbf{x}_n \}_{n=1}^N$ for different values for $N$. 
We define the minimal distance of $X_N$, denoted by $h_{\rm min}(X_N)$, as the smallest distance between any two distinct points, 
\begin{equation}\label{eq:min_distance}
	h_{\rm min}(X_N) = \min_{\mathbf{x}_n \in X_N} \ \min_{\mathbf{x}_m \in X_N \setminus \mathbf{x}_n} \| \mathbf{x}_m - \mathbf{x}_n \|_2.
\end{equation} 
At the same time, we define the maximal (filling) distance of $X_N$, denoted by $h_{\rm max}(X_N)$, as the largest distance between any point and the closest distinct point, 
\begin{equation}\label{eq:max_distance}
	h_{\rm max}(X_N) = \max_{\mathbf{x}_n \in X_N} \ \min_{\mathbf{x}_m \in X_N \setminus \mathbf{x}_n} \| \mathbf{x}_m - \mathbf{x}_n \|_2.
\end{equation} 

\begin{figure}[tb]
	\centering
  	\begin{subfigure}[b]{0.45\textwidth}
		\includegraphics[width=\textwidth]{%
      		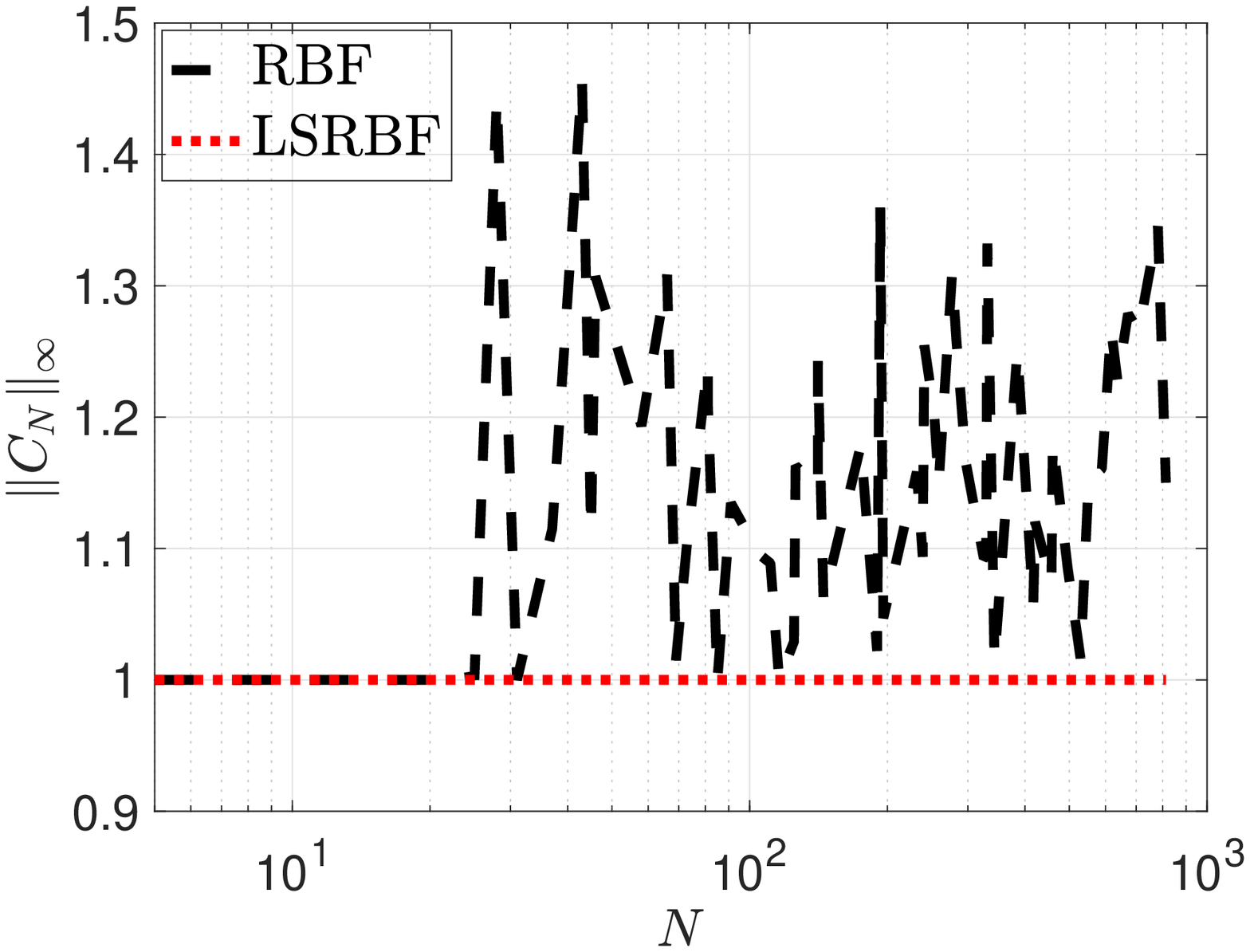} 
    		\caption{Stability measure}
    		\label{fig:LSRBF_stabMeasure}
  	\end{subfigure}%
	~
	\begin{subfigure}[b]{0.45\textwidth}
		\includegraphics[width=\textwidth]{%
      		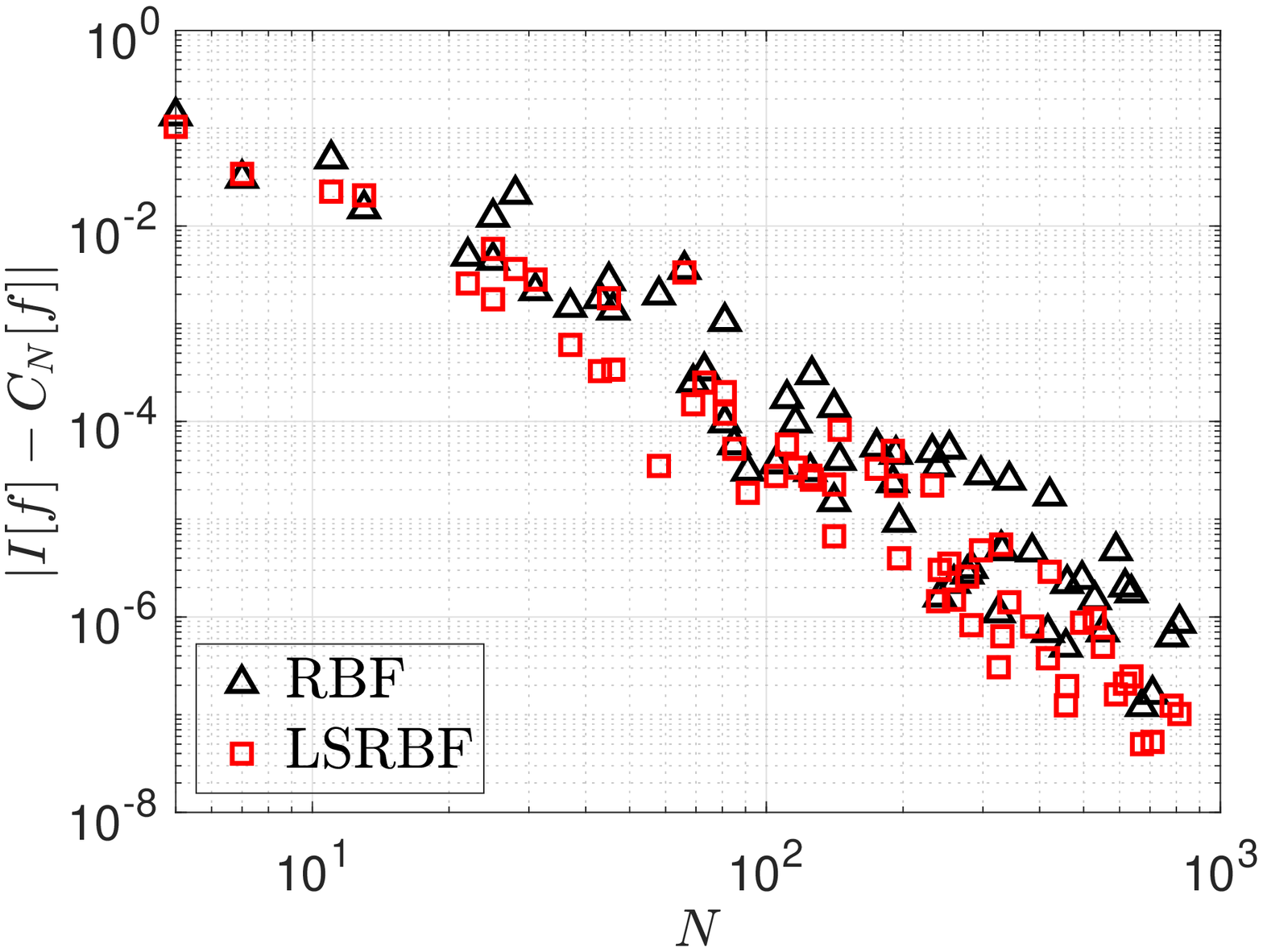} 
    		\caption{Errors, no noise}
    		\label{fig:LSRBF_Genz1_noNoise}
  	\end{subfigure}%
	\\ 
	\begin{subfigure}[b]{0.45\textwidth}
		\includegraphics[width=\textwidth]{%
      		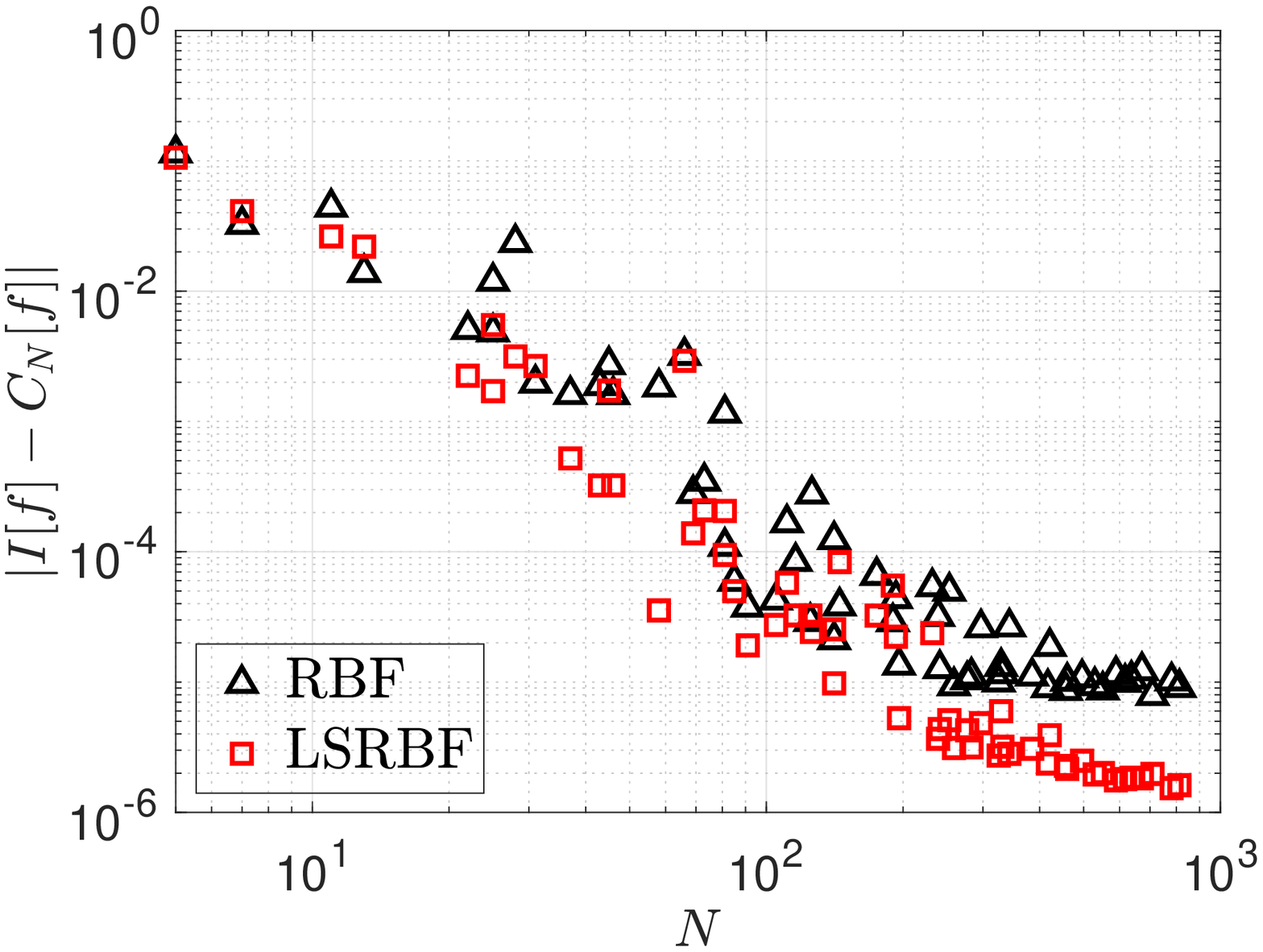} 
    		\caption{Errors, noise of magnitude $10^{-4}$}
    		\label{fig:LSRBF_Genz1_Noise_Eminus4}
  	\end{subfigure}%
	~
	\begin{subfigure}[b]{0.45\textwidth}
		\includegraphics[width=\textwidth]{%
      		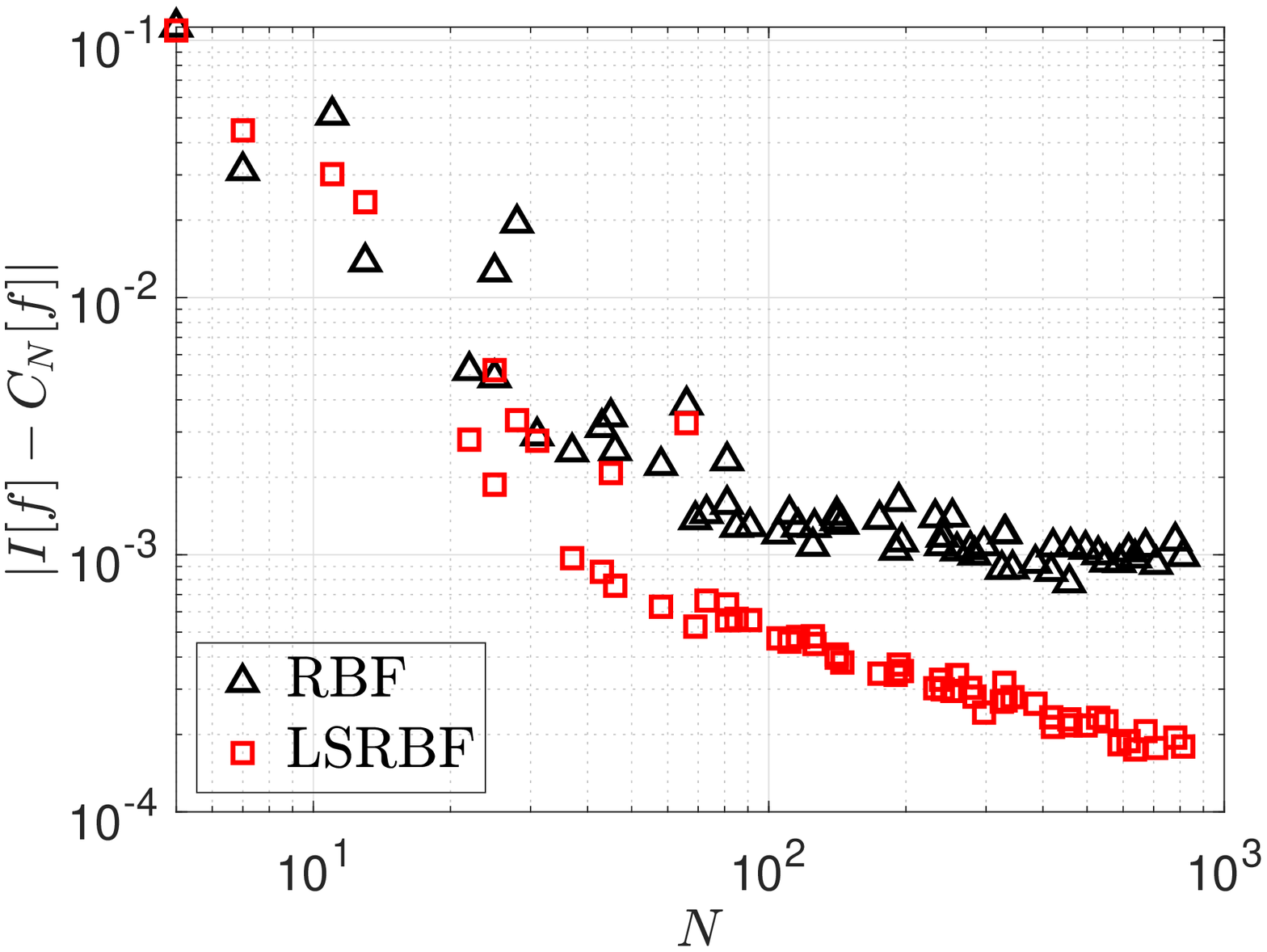} 
    		\caption{Errors, noise of magnitude $10^{-2}$}
    		\label{fig:LSRBF_Genz1_Noise_Eminus2}
  	\end{subfigure}%
  	\caption{
		Stability measure and errors for Genz' first test function $g_1$ on $\Omega = [0,1]^2$ with $\omega \equiv 1$ using an interpolatory RBF-QF and a stable LSRBF-QF. 
		Random points and a Gaussian kernel with a constant shape parameter were used. 
	}
  	\label{fig:LSRBF}
\end{figure} 

Figure \ref{fig:LSRBF_stabMeasure} provides the values of the stability measure for the interpolatory RBF-QF (``RBF") and the stable LSRBF-QF (``LSRBF"). 
For the same center points (same RBF function space for which the quadrature is exact), the LSRBF-QF uses more data points to evaluate the integrand than the interpolatory RBF-QF. 
The other way around, for the same data points, the interpolatory RBF-QF is exact for a larger RBF function space than the LSRBF-QF. 
At the same time, the interpolatory RBF-QF is found to have a suboptimal stability measure (due to negative weights), which results in stability issues. 
In contrast, in all cases, the LSRBF-QF has an optimal stability measure (due to the weights being positive). 
Further, Figure \ref{fig:LSRBF_Genz1_noNoise} reports on the errors of the interpolatory RBF-QF and the stable LSRBF-QF on $N$ random points applied to Genz' first test function $g_1$ on $\Omega = [0,1]^2$ with $\omega \equiv 1$. 
In this example, both formulas perform similarly. 
In Figures \ref{fig:LSRBF_Genz1_Noise_Eminus4} and \ref{fig:LSRBF_Genz1_Noise_Eminus2}, we repeated this experiment but added uniformly distributed noise of magnitude $10^{-4}$ and $10^{-2}$ to the function values at the data points. 
The accuracy of the interpolatory RBF-QF deteriorates notably stronger than that of the LSRBF-QF in the presence of noise due to the improved stability of the latter. 
We made the same observation also for other point distributions and Genz test functions. 

\begin{figure}[tb]
	\centering
  	\begin{subfigure}[b]{0.45\textwidth}
		\includegraphics[width=\textwidth]{%
      		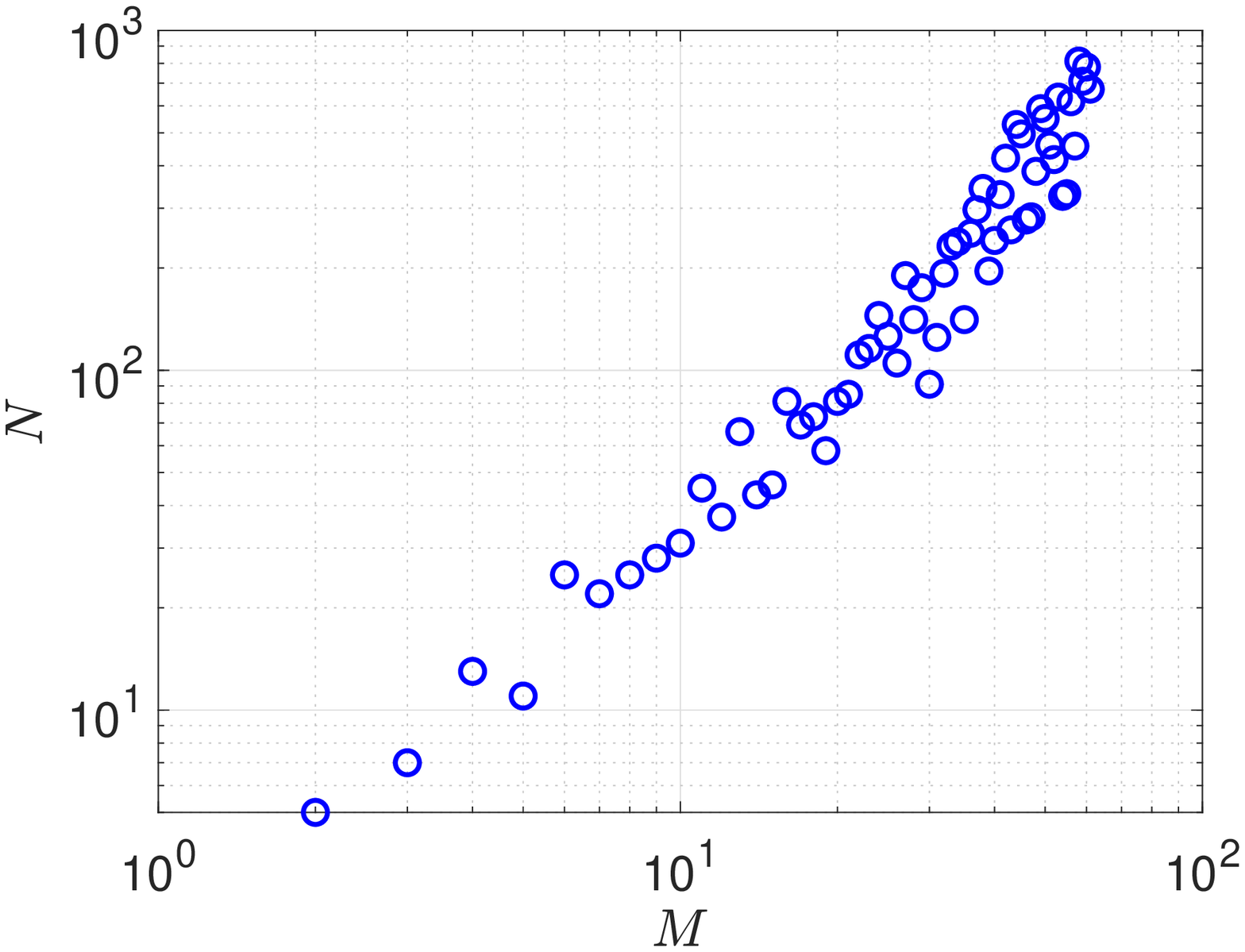} 
    		\caption{Random points}
    		\label{fig:ratio_random}
  	\end{subfigure}%
	~
	\begin{subfigure}[b]{0.45\textwidth}
		\includegraphics[width=\textwidth]{%
      		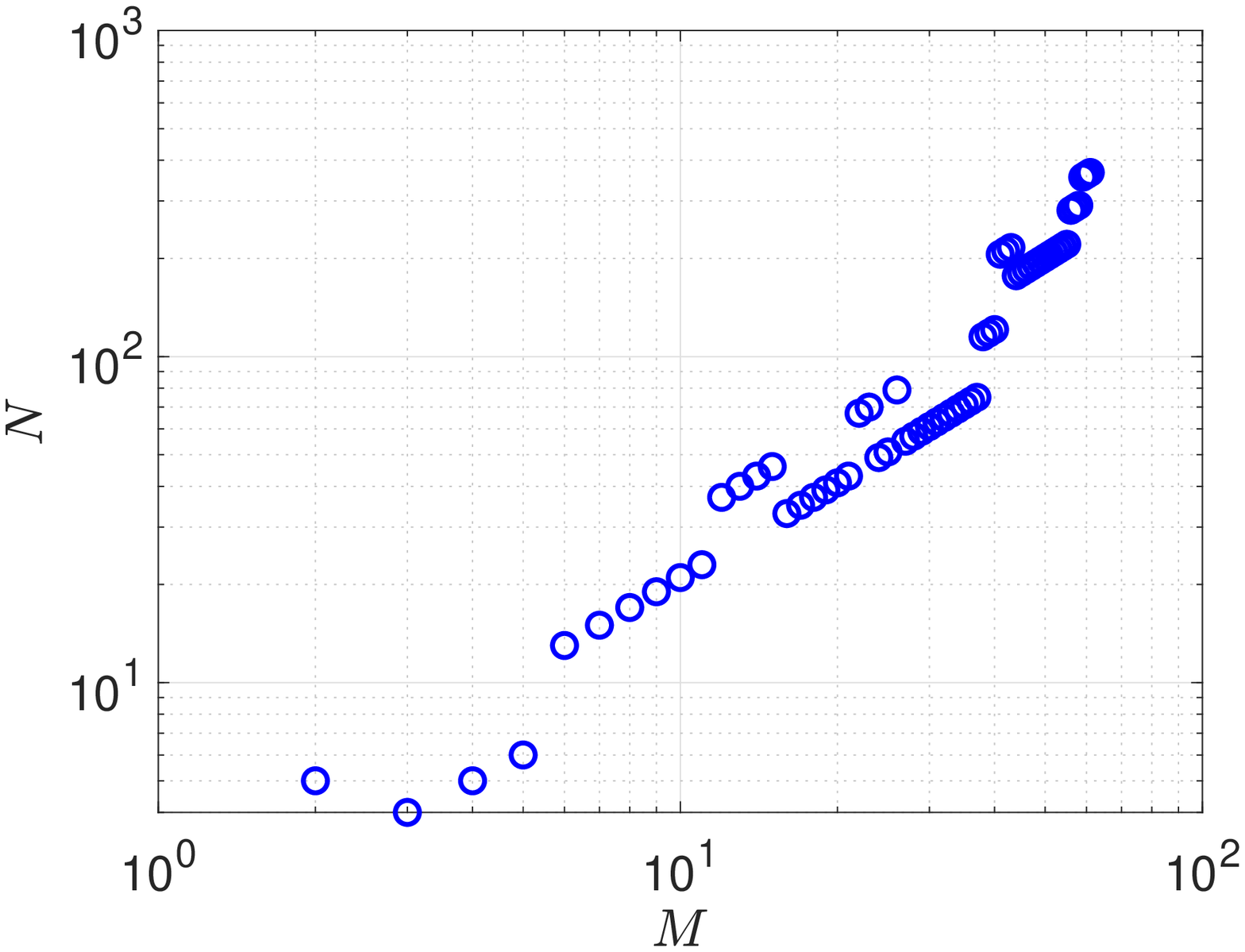} 
    		\caption{Halton points}
    		\label{fig:ratio_Halton}
  	\end{subfigure}%
  	\caption{
		The smallest number of random/Halton data points, $N$, needed to find a positive LSRBF-QF that is exact for the RBF approximation space induced by the first $M$ random/Halton center points. 
	}
  	\label{fig:LSRBF_ratio}
\end{figure} 

\begin{figure}[tb]
	\centering
  	\begin{subfigure}[b]{0.45\textwidth}
		\includegraphics[width=\textwidth]{%
      		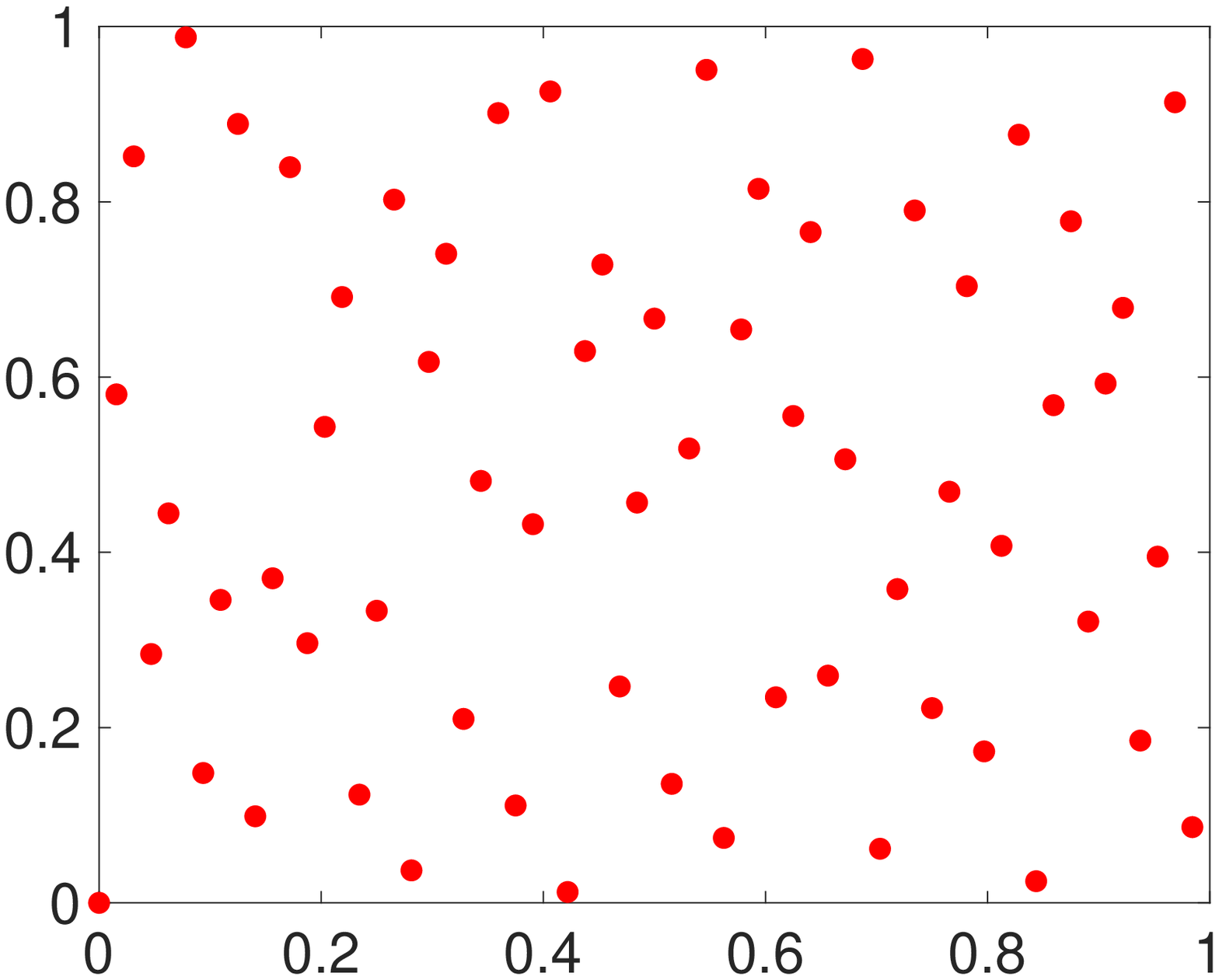} 
    		\caption{Halton points}
    		\label{fig:points_Halton}
  	\end{subfigure}%
	~
	\begin{subfigure}[b]{0.45\textwidth}
		\includegraphics[width=\textwidth]{%
      		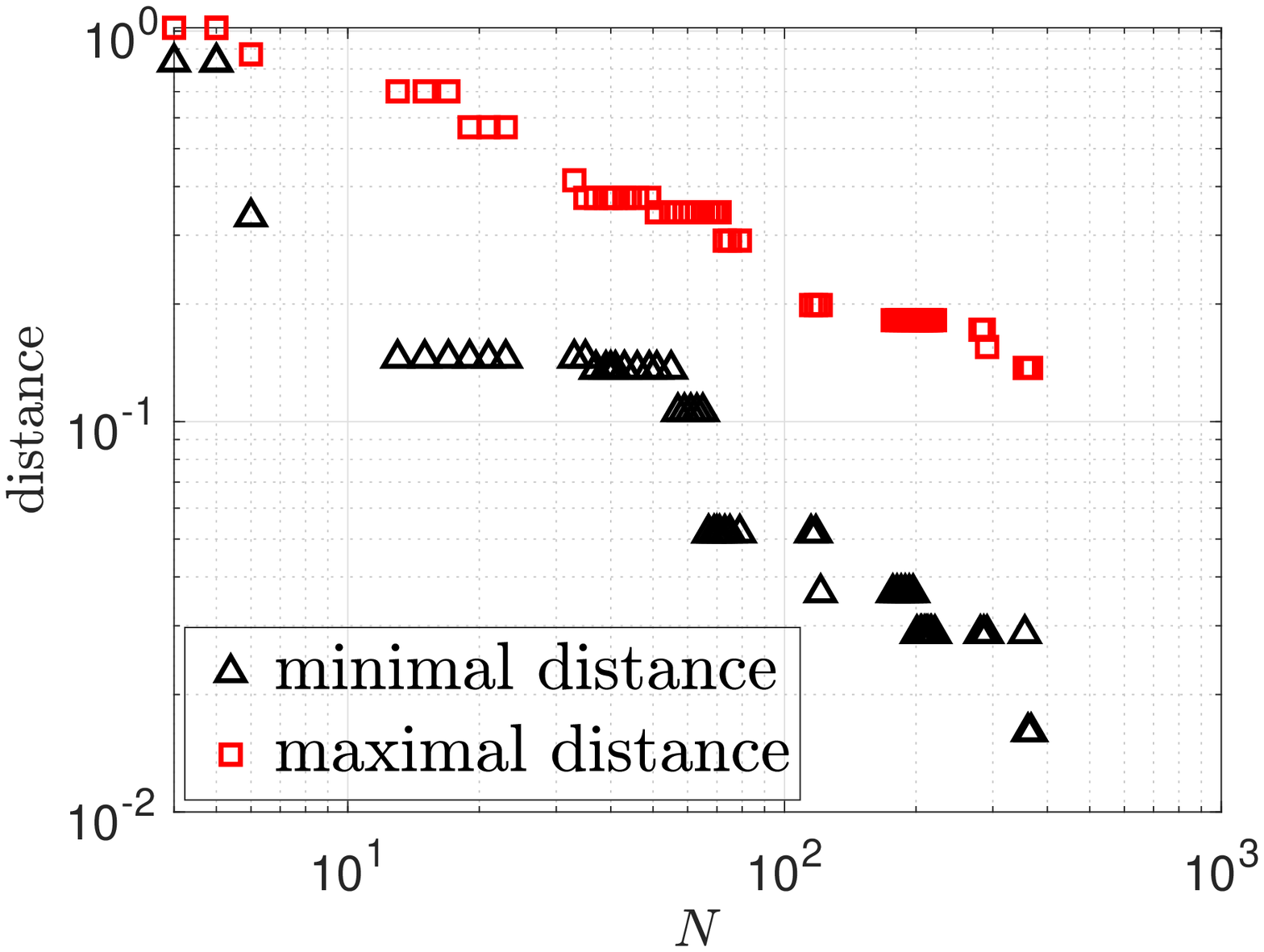} 
    		\caption{Minimal and maximal distance}
    		\label{fig:distance_Halton}
  	\end{subfigure}%
  	\caption{
		Halton points and the corresponding minimal and maximal distance, \eqref{eq:min_distance} and \eqref{eq:max_distance}. 
		The minimal distance is the smallest distance between any two distinct points. 
		The maximal distance is the largest distance between any point and the closest distinct point.  
	}
  	\label{fig:LSRBF_Halton}
\end{figure} 

The LSRBF-QF having an optimal stability measure (being positive) for sufficiently large $N$ can be explained by Corollary \ref{cor:LSRBFQF} since we are given a compact domain, a positive weight function, and a function space of continuous and bounded functions that contains constants. 
Figure \ref{fig:LSRBF_ratio} reports the smallest number of random/Halton data points $N$ we needed to find a positive LSRBF-QF that is exact for the RBF approximation space induced by using the first $M$ random/Halton points as centers. 
To also illustrate the semi-random Halton points, Figure \ref{fig:LSRBF_Halton} visualizes the first $64$ Halton points and their minimal and maximal distance for an increasing number $N$. 
Considering the model ``$N = C \cdot M^s$" and performing a least-squares fit for the parameters $C$ and $s$ given the data illustrated in Figure \ref{fig:LSRBF_ratio} revealed the following: 
For random data and center points, we found $C \approx 2.1 \cdot 10^{-1}$ and $s \approx 1.9$. 
For Halton data and center points, we found $C \approx 4.9 \cdot 10^{-2}$ and $s \approx 2.1$. 
In both cases, we found $N$ to be roughly linearly proportional to the squared dimension of the approximation space for which the positive least-squares quadrature is exact. 
Similar ratios were also observed in \cite{huybrechs2009stable,glaubitz2020stable,glaubitz2020stableQFs,glaubitz2020shock,glaubitz2020stableCFs,cohen2013stability,cohen2017optimal,migliorati2022stable,glaubitz2020constructing}. 
It might be argued that the observed ratio between $N$ and $M$ is necessary for the LSRBF-QF to avoid inherent stability issues predicted by the 'impossibility' theorem proved in \cite{platte2011impossibility}, which states that any procedure for approximating univariate functions from equally spaced samples that converges exponentially fast must also be exponentially ill-conditioned.

Finally, we address the convergence rate observed in Figure \ref{fig:LSRBF_Genz1_noNoise}. 
In theory, Gaussian RBF interpolants can converge almost exponentially fast\footnote{%
For a function from the appropriate native function space, the $L^\infty(\Omega)$-error between the function and its RBF interpolant is in $\mathcal{O}( \exp( -C \log h_{\rm max}(X_N) / h_{\rm max}(X_N) ) )$; see \cite{wendland2004scattered}.
} in the $L^\infty(\Omega)$ with the maximal (filling) distance. 
For simplicity, we considered the model ``$|I[f] - C_N[f]| = \exp( -C h_{\rm max}(X_N)^s )$" and performing a least-squares fit for the parameters $C$ and $s$ using the data presented in Figure \ref{fig:LSRBF_Genz1_noNoise}. 
For the interpolatory RBF-QF, we found $C \approx 2.0$ and $s \approx -1.1$. 
For the LSRBF-QF, we found $C \approx 1.8$ and $s \approx -1.3$. 
Both QFs converge roughly exponentially, with the LSRBF-QF converging slightly faster than the interpolatory RBF-QF, even in the noiseless case. 
A more general comment on the convergence of LSRBF-QFs is offered in Remark \ref{rem:conv_LSRBFQF}.

\begin{remark}[Convergence of LSRBF-QF]\label{rem:conv_LSRBFQF}
	Assume that the positive LSRBF-QF $C_N$ on $\Omega$ is exact for all functions from the RBF approximation space $\mathcal{S}_{M,d}(\Omega)$ with $d \geq 0$. 
	Due to $C_N$ being positive and exact for constants, we have $\|C_N\|_{\infty} = \|I\|_\infty$ and the Lebesgue inequality \eqref{eq:L-inequality} implies 
	\begin{equation*}\label{eq:L-inequality2} 
		| C_N[f] - I[f] | 
			\leq 2 \| I \|_{\infty} \left( \inf_{ s \in \mathcal{S}_{M,d}(\Omega) } \norm{ f - s }_{L^{\infty}(\Omega)} \right) 
	\end{equation*} 
	for any continuous $f: \Omega \to \R$. 
	Now consider a sequence of positive LSRBF-QFs $(C_N)_{N \in \N}$ with $C_N$ being exact for $\mathcal{S}_{M,d}(\Omega)$ with $M = M(N)$. 
	Assume that $\mathcal{S}_{M,d}(\Omega) \subset \mathcal{S}_{M+1,d}(\Omega)$ for all $M \in \N$ and that the given function $f$ lies in $\bigcup_{M \in \N} \mathcal{S}_{M,d}(\Omega)$. 
	Thus, if $M(N) \to \infty$ for $N \to \infty$, then $(C_N[f])_{N \in \N}$ converges to $I[f]$ as $N \to \infty$. 
	Let us now assume that the ratio between $M$ and $N$ is of the form $N = \mathcal{O}(M^2)$, which we numerically observed to be true. 
	The convergence rate of the LSRBF-QF for $f$ is then the square root of the convergence rate of the best approximation of $f$ from the sequence of approximation spaces $( \mathcal{S}_{M,d}(\Omega) )_{N \in \N}$ for the $L^{\infty}(\Omega)$-norm.
\end{remark}

\subsection{Polyharmonic splines} 
\label{sub:num_PHS}

\begin{figure}[tb]
	\centering 
	\begin{subfigure}[b]{0.32\textwidth}
		\includegraphics[width=\textwidth]{%
      	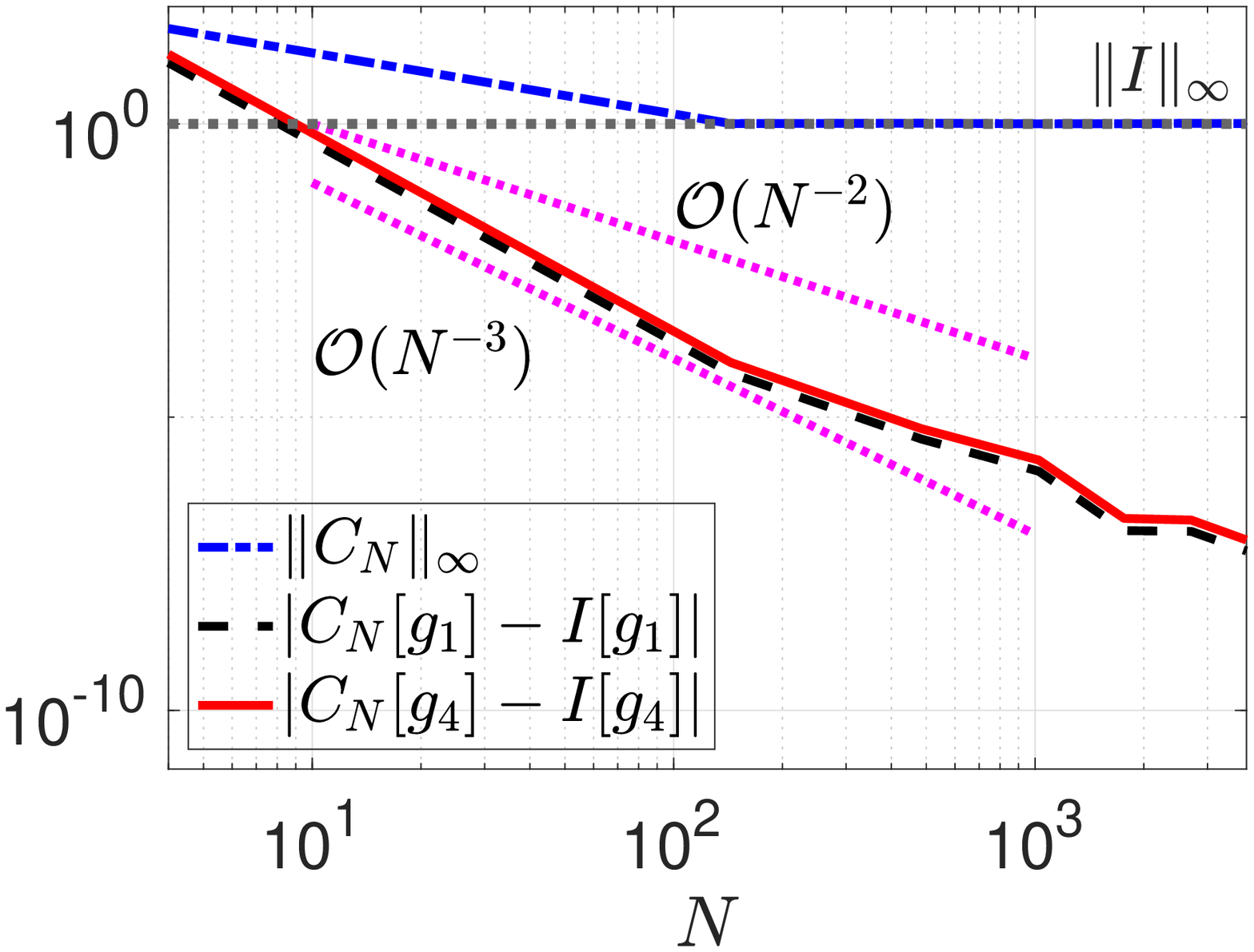} 
    \caption{Cubic, $d=-1$}
    \label{fig:S73_cubic_Genz14_Halton_noPol}
  	\end{subfigure}%
	\begin{subfigure}[b]{0.32\textwidth}
		\includegraphics[width=\textwidth]{%
      	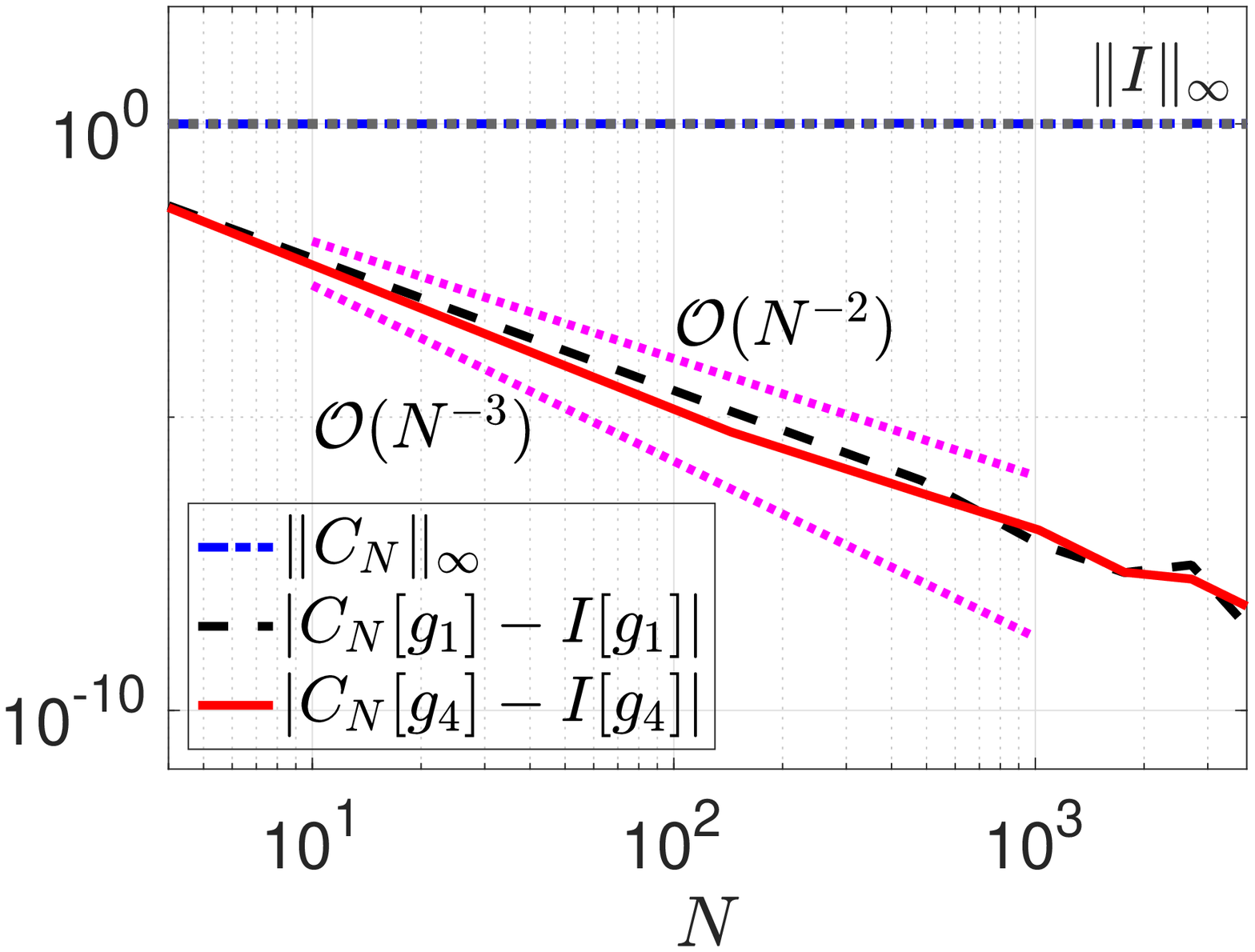} 
    \caption{Cubic, $d=0$}
    \label{fig:S73_cubic_Genz14_Halton_d0}
  	\end{subfigure}%
	\begin{subfigure}[b]{0.32\textwidth}
		\includegraphics[width=\textwidth]{%
      	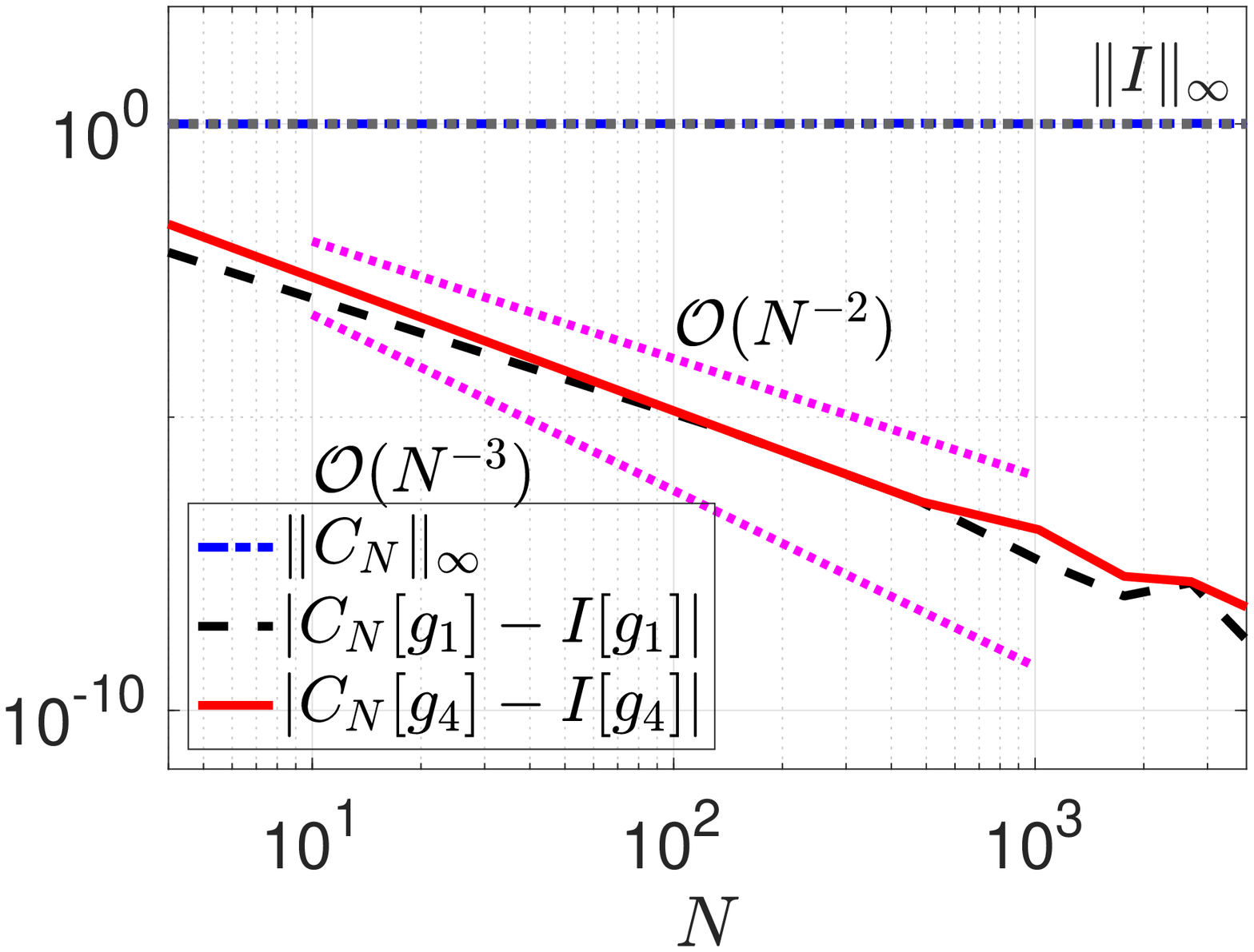} 
    \caption{Cubic, $d=1$}
    \label{fig:S73_cubic_Genz14_Halton_d1}
  	\end{subfigure}%
	\\
  	\begin{subfigure}[b]{0.32\textwidth}
		\includegraphics[width=\textwidth]{%
      	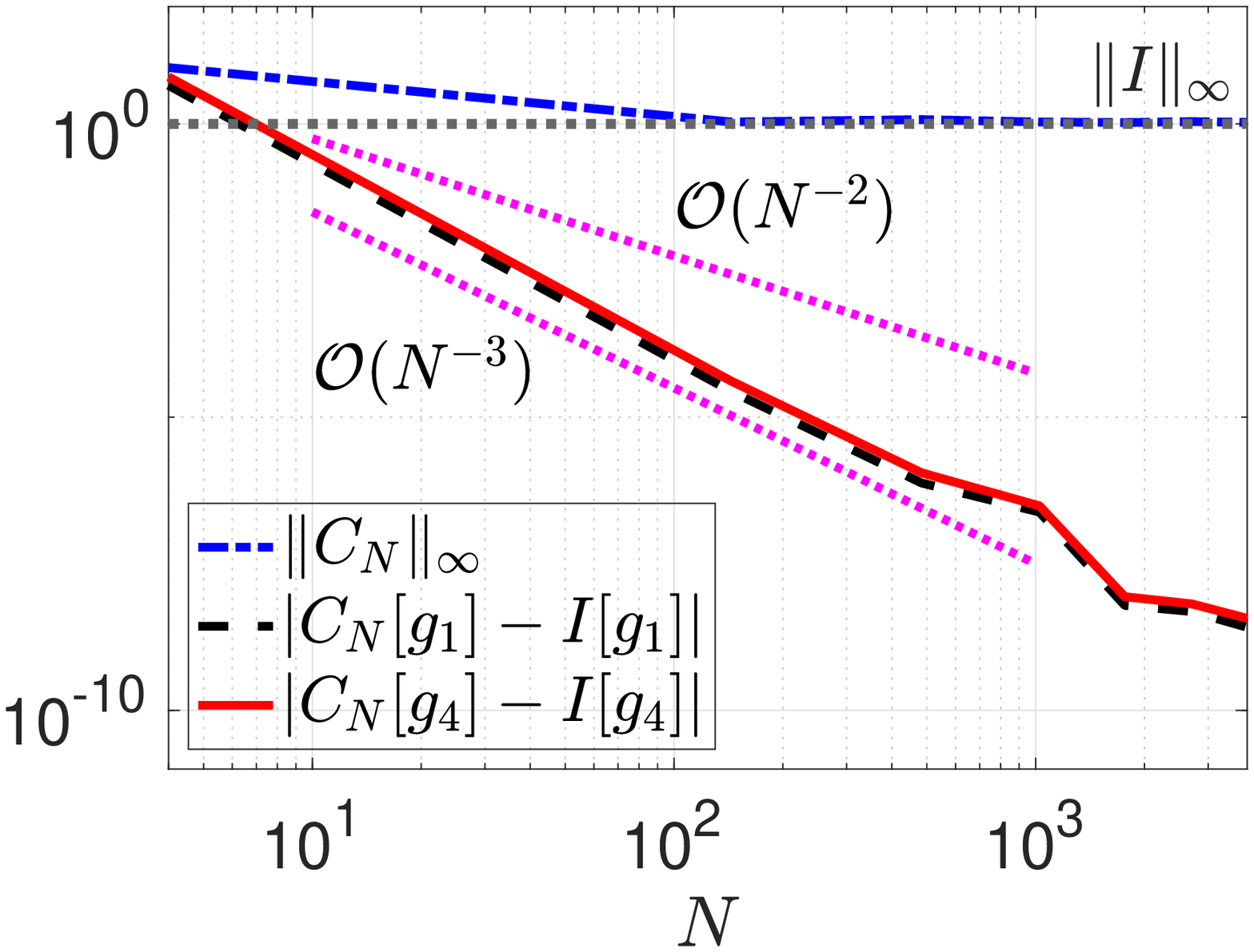} 
    \caption{Quintic, $d=-1$}
    \label{fig:S73_quintic_Genz14_Halton_noPol}
  	\end{subfigure}%
	\begin{subfigure}[b]{0.32\textwidth}
		\includegraphics[width=\textwidth]{%
      	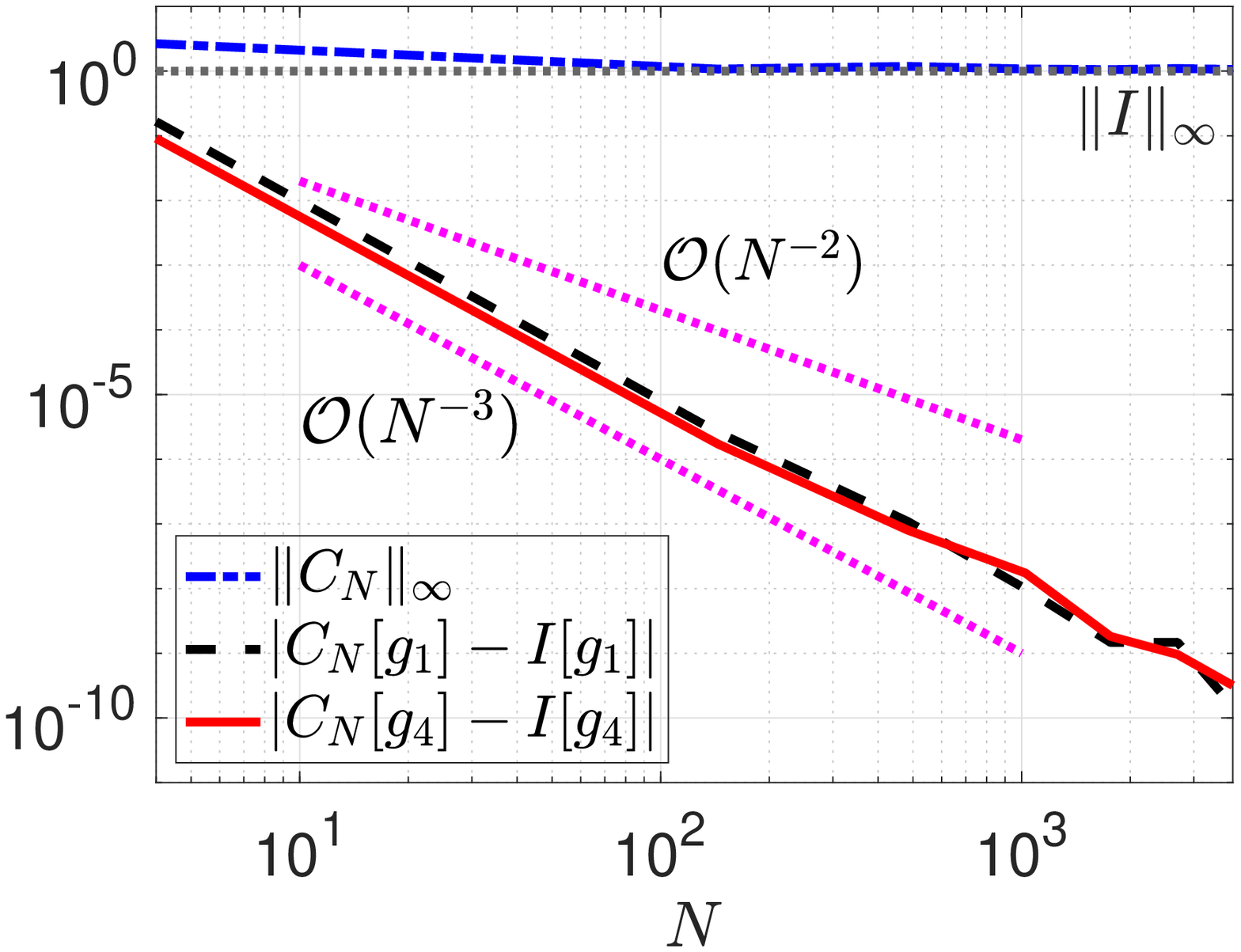} 
    \caption{Quintic, $d=0$}
    \label{fig:S73_quintic_Genz14_Halton_d0}
  	\end{subfigure}%
	\begin{subfigure}[b]{0.32\textwidth}
		\includegraphics[width=\textwidth]{%
      	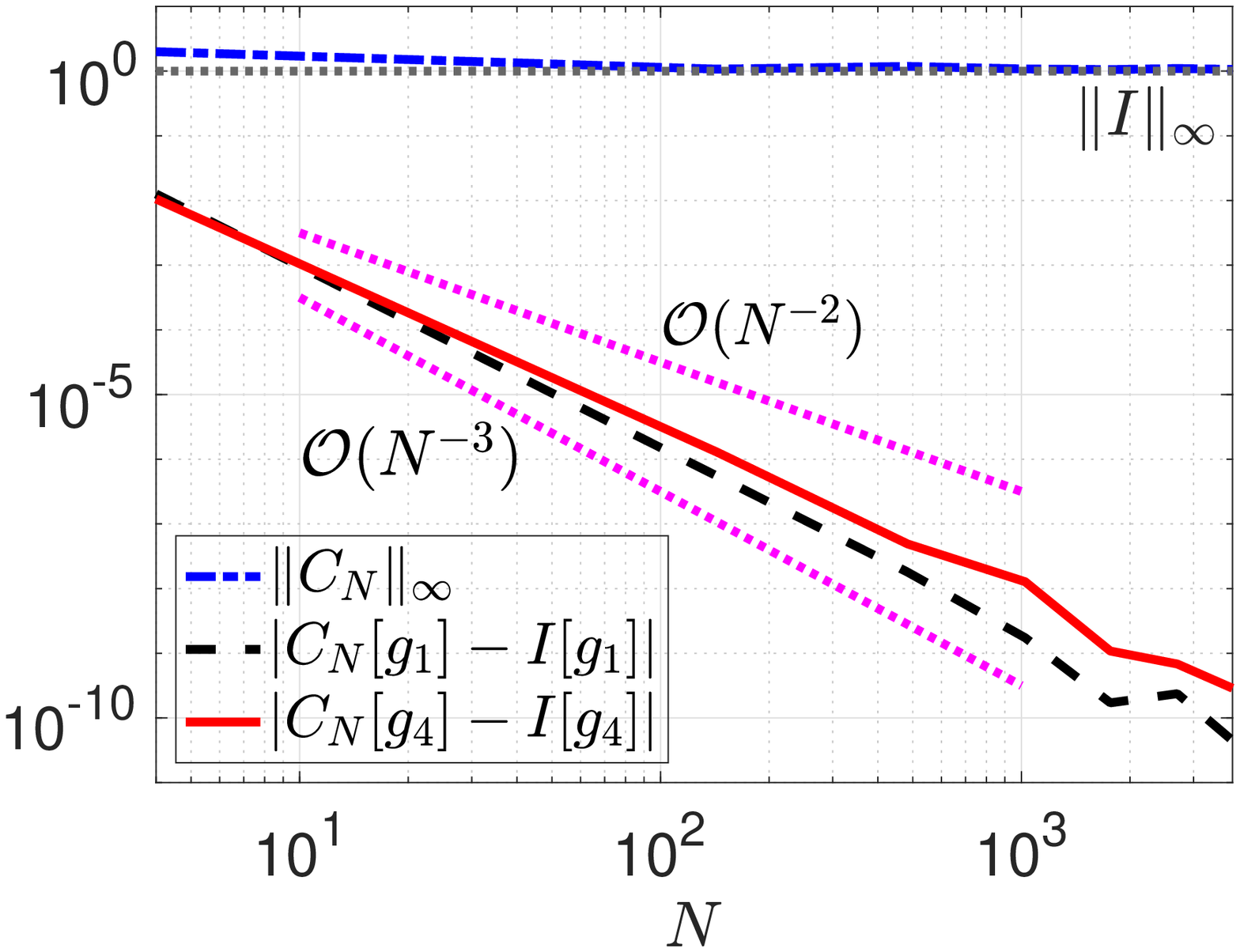} 
    \caption{Quintic, $d=1$}
    \label{fig:S73_quintic_Genz14_Halton_d1}
  	\end{subfigure}%
  	\caption{
	Error analysis for the cubic ($\varphi(r) = r^3$) and quintic ($\varphi(r) = r^5$) PHS RBF in two dimensions using Halton points. 
	The first and fourth Genz test functions $g_1, g_4$ were considered on $\Omega = [0,1]^2$; see \eqref{eq:Genz}. 
  	}
  	\label{fig:error_PHS_Genz14}
\end{figure} 

We end this section by providing a similar investigation for PHS. 
Again, the first and fourth Genz test functions on $\Omega = [0,1]^2$ are considered. 
However, no shape parameter is involved for PHS, and we consider their stability and accuracy for an increasing number of Halton points. 
Figure \ref{fig:error_PHS_Genz14} shows the results for the cubic ($\varphi(r) = r^3$) and quintic ($\varphi(r) = r^5$) PHS RBF. 
We either added no polynomials ($d=-1$) or polynomial terms of order $d=0$ and $d=1$. 
Figure \ref{fig:error_PHS_Genz14} shows that all RBF-QFs converge while being stable or at least asymptotically stable, independent of the added polynomial term. 
In particular, we see that adding polynomial terms does not affect the asymptotic stability of the PHS RBF-QFs, per our results from \S\ref{sec:connection}.  
Finally, we see that the convergence rate of the RBF-QF depends on the kernel rather than being governed by the polynomial degree $d$. 
The added polynomial term reduces the error of the PHS RBF-QF but not the convergence rate. 
We observe second-order convergence for the cubic PHS RBF and third-order convergence for the quintic PHS RBF.\footnote{In Figure \ref{fig:S73_cubic_Genz14_Halton_noPol} the cubic PHS RBF-QF first shows third-order convergence before it then settles for second-order convergence. We believe that the observed initial third-order decrease in the error is a combination of the second-order approximation rate of the cubic PHS-RBF interpolant and the decreasing Lebesgue constant $\|C_N\|_{\infty}$ in \eqref{eq:L-inequality}. Once the QF is stable ($\|C_N\|_{\infty} = \|I\|_{\infty}$), the second-order approximation rate dominates the error of the QF, and we thus start to observe second-order convergence.}
\section{Concluding thoughts} 
\label{sec:summary} 

In this work, we investigated stability of RBF-QFs. 
We started by showing that stability of RBF-QFs can be connected to the famous Lebesgue constant of the underlying RBF interpolant. 
This indicates that RBF-QFs might benefit from low Lebesgue constants. 
Furthermore, stability was proven for RBF-QFs based on compactly supported RBFs under the assumption of a sufficiently large number of (equidistributed) data points and the shape parameter(s) lying above a certain threshold. 
Finally, we showed that under certain conditions, asymptotic stability of RBF-QFs is independent of polynomial terms included in RBF approximations. 
A series of numerical tests accompanied the above findings.

\section*{Acknowledgements}
JG was supported by AFOSR \#F9550-18-1-0316 and ONR \#N00014-20-1-2595.
We thank Toni Karvonen for pointing out the connection between RBF-QFs and Bayesian quadrature.

\appendix 
\section{Moments} 
\label{sec:app_moments} 

Henceforth, we provide the moments for different RBFs. 
The one-dimensional case is discussed in \S\ref{sub:app_moments_1d}, while two-dimensional moments are derived in \S\ref{sub:app_moments_2d}.

\subsection{One-dimensional moments}
\label{sub:app_moments_1d}

Let us consider the one-dimensional case of $\Omega = [a,b]$ and distinct data points $x_1,\dots,x_N \in [a,b]$.

\subsubsection{Gaussian RBF} 

For $\varphi(r) = \exp( - \varepsilon^2 r^2 )$, the moment of the translated Gaussian RBF, 
\begin{equation}\label{eq:moment_1d_G}
	m_n = m(\varepsilon,x_n,a,b) = \int_a^b \exp( - \varepsilon^2 | x - x_n |^2 ) \intd x,
\end{equation} 
is given by 
\begin{equation*} 
	m_n = \frac{\sqrt{\pi}}{2 \varepsilon} \left[ \mathrm{erf}( \varepsilon(b-x_n) ) - \mathrm{erf}( \varepsilon(a-x_n) ) \right].
\end{equation*} 
Here, $\mathrm{erf}(x) = 2/\sqrt{\pi} \int_0^x \exp( -t^2 ) \intd t$ denotes the usual \emph{error function}, \cite[Section 7.2]{dlmf2021digital}.

\subsubsection{Polyharmonic splines} 

For $\varphi(r) = r^k$ with odd $k \in \N$, the moment of the translated PHS, 
\begin{equation*} 
	m_n = m(x_n,a,b) = \int_a^b \varphi( x - x_n ) \intd x,
\end{equation*} 
is given by 
\begin{equation*} 
	m_n = \frac{1}{k+1} \left[ (a-x_n)^{k+1} + (b-x_n)^{k+1} \right], 
	\quad n=1,2,\dots,N.
\end{equation*} 
For $\varphi(r) = r^k \log r$ with even $k \in \N$, on the other hand, we have 
\begin{equation*} 
	m_n 
		= (x_n - a)^{k+1} \left[ \frac{\log( x_n - a )}{k+1} - \frac{1}{(k+1)^2} \right] 
		+ (b - x_n)^{k+1} \left[ \frac{\log( b - x_n )}{k+1} - \frac{1}{(k+1)^2} \right].
\end{equation*} 
Note that for $x_n = a$ the first term is zero, while for $x_n = b$ the second term is zero.

\subsection{Two-dimensional moments}
\label{sub:app_moments_2d} 

Here, we consider the two-dimensional case, where the domain is given by a rectangular of the form $\Omega = [a,b] \times [c,d]$.

\subsubsection{Gaussian RBF}

For $\varphi(r) = \exp( - \varepsilon^2 r^2 )$, the two-dimensional moments can be written as products of one-dimensional moments. 
In fact, we have 
\begin{equation*} 
	\int_a^b \int_c^d \exp( - \varepsilon^2 \|(x-x_n,y-y_n\|_2^2 ) 
		=  m(\varepsilon,x_n,a,b) \cdot m(\varepsilon,y_n,c,d).
\end{equation*}
Here, the multiplicands on the right-hand side are the one-dimensional moments from \eqref{eq:moment_1d_G}.

\subsubsection{Polyharmonic splines and other RBFs} 

If it is not possible to trace the two-dimensional moments back to the one-dimensional ones, we are in need of another approach. 
This is, for instance, the case for PHS. 
We start by noting that for a data points $(x_n,y_n) \in [a,b] \times [c,d]$ the corresponding moment can be rewritten as follows: 
\begin{equation*} 
	m(x_n,y_n) 
		= \int_{a}^b \int_{c}^d \varphi( \| (x-x_n,y-y_n)^T \|_2 ) \intd y \intd x 
		= \int_{\tilde{a}}^{\tilde{b}} \int_{\tilde{c}}^{\tilde{d}} \varphi( \| (x,y)^T \|_2 ) \intd y \intd x
\end{equation*}
with translated boundaries $\tilde{a} = a - x_n$, $\tilde{b} = b - x_n$, $\tilde{c} = c - y_n$, and $\tilde{d} = d - y_n$. 
We are not aware of an explicit formula for such integrals for most popular RBFs readily available from the literature.  
That said, such formulas were derived in \cite{reeger2016numericalA,reeger2016numericalB,reeger2018numerical} (also see \cite[Chapter 2.3]{watts2016radial}) for the integral of $\varphi$ over a \emph{right triangle} with vertices $(0,0)^T$, $(\alpha,0)^T$, and $(\alpha,\beta)^T$. 
Assuming $\tilde{a} < 0 < \tilde{b}$ and $\tilde{c} < 0 < \tilde{d}$, we therefore partition the shifted domain ${\tilde{\Omega} = [\tilde{a},\tilde{b}] \times [\tilde{c},\tilde{d}]}$ into eight right triangles. 
Denoting the corresponding integrals by $I_1, \dots, I_8$, the moment $m(x_n,y_n)$ correspond to the sum of these integrals. 
The procedure is illustrated in Figure \ref{fig:moments}. 

\begin{figure}[tb]
	\centering 
  	\begin{tikzpicture}[domain = -6.5:6.5, scale=0.7, line width=1.25pt]

			\draw[->,>=stealth] (-4.5,0) -- (7,0) node[below] {$x$};
    			\draw[->,>=stealth] (0,-2.75) -- (0,4.5) node[right] {$y$};
			\draw (-4,0.1) -- (-4,-0.1) node [below] {$\tilde{a}$ \ \ \ \ };
			\draw (6,0.1) -- (6,-0.1) node [below] {\ \ \ \ $\tilde{b}$};
			\draw (-0.1,-2) -- (0.1,-2) node [below] {\ \ $\tilde{c}$};
			\draw (-0.1,3) -- (0.1,3) node [above] {\ \ $\tilde{d}$};

			\draw[blue] (-4,-2) rectangle (6,3);
			
			\draw[red,dashed] (0,0) -- (6,3) {};
			\draw[red,dashed] (0,0) -- (-4,3) {};
			\draw[red,dashed] (0,0) -- (-4,-2) {};
			\draw[red,dashed] (0,0) -- (6,-2) {};
			
			\draw[red] (4.5,1) node {\Large $I_1$};
			\draw[red] (1,2) node {\Large $I_2$};
			\draw[red] (-1,2) node {\Large $I_3$};
			\draw[red] (-3,1) node {\Large $I_4$};
			\draw[red] (-3,-0.5) node {\Large $I_5$};
			\draw[red] (-0.5,-1.15) node {\Large $I_6$};
			\draw[red] (1,-1.25) node {\Large $I_7$};
			\draw[red] (4.5,-0.75) node {\Large $I_8$};

	\end{tikzpicture}
  	\caption{Illustration of how the moments can be computed on a rectangle in two dimensions}
  	\label{fig:moments}
\end{figure}
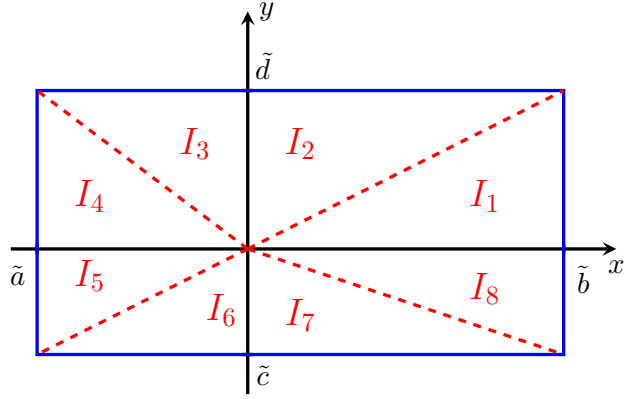

The special cases where one (or two) of the edges of the rectangle align with one of the axes can be treated similarly. 
However, in this case, a smaller subset of the triangles is considered. 
We leave the details to the reader, and note the following formula for the weights: 
\begin{equation*} 
\begin{aligned}
	m(x_n,y_n) 
		& = \left[ 1 - \delta_0\left(\tilde{b} \tilde{d}\right) \right] \left( I_1 + I_2 \right) 
		+ \left[ 1 - \delta_0\left(\tilde{a} \tilde{d}\right) \right] \left( I_3 + I_4 \right) \\
		& + \left[ 1 - \delta_0\left(\tilde{a} \tilde{c}\right) \right] \left( I_5 + I_6 \right) 
		+ \left[ 1 - \delta_0\left(\tilde{b} \tilde{c}\right) \right] \left( I_7 + I_8 \right) 
\end{aligned}
\end{equation*} 
Here, $\delta_0$ denotes the usual Kronecker delta defined as $\delta_0(x) = 1$ if $x = 0$ and $\delta_0(x) = 0$ if $x \neq 0$. 
The above formula holds for general $\tilde{a}$, $\tilde{b}$, $\tilde{c}$, and $\tilde{d}$. 
Note that all the right triangles can be rotated or mirrored in a way that yields a corresponding integral of the form 
\begin{equation}\label{eq:refInt} 
	I_{\text{ref}}(\alpha,\beta) 
		= \int_0^{\alpha} \int_0^{\frac{\beta}{\alpha}x} \varphi( \| (x,y)^T \|_2 ) \intd y \intd x.
\end{equation} 
More precisely, we have 
\begin{equation*}
\begin{alignedat}{4}
	& I_1 = I_{\text{ref}}(\tilde{b},\tilde{d}), \quad 
	&& I_2 = I_{\text{ref}}(\tilde{d},\tilde{b}), \quad  
	&& I_3 = I_{\text{ref}}(\tilde{d},-\tilde{a}), \quad 
	&& I_4 = I_{\text{ref}}(-\tilde{a},\tilde{d}), \\ 
	& I_5 = I_{\text{ref}}(-\tilde{a},-\tilde{c}), \quad 
	&& I_6 = I_{\text{ref}}(-\tilde{c},-\tilde{a}), \quad 
	&& I_7 = I_{\text{ref}}(-\tilde{c},\tilde{b}), \quad 
	&& I_8 = I_{\text{ref}}(\tilde{b},-\tilde{c}).
\end{alignedat} 
\end{equation*}
Finally, explicit formulas of the reference integral $I_{\text{ref}}(\alpha,\beta)$ over the right triangle with vertices $(0,0)^T$, $(\alpha,0)^T$, and $(\alpha,\beta)^T$ for some PHS can be found in Table \ref{tab:moments}. 
Similar formulas are also available, for instance, for Gaussian, multiquadric and inverse multiquadric RBFs. 

\begin{table}[t]
  \centering 
  \renewcommand{\arraystretch}{1.5}
  \begin{tabular}{c|c}
    $\varphi(r)$ & $I_{\text{ref}}(\alpha,\beta)$ \\ 
    \midrule 
    $r^2 \log r$ & $\frac{\alpha}{144} \left[ 24\alpha^3 \arctan\left( \beta/\alpha \right) + 6 \beta (3\alpha^2 + \beta^2) \log( \alpha^2 + \beta^2 ) - 33\alpha^2\beta - 7\beta^3 \right]$ \\ 
    $r^3$ & $\frac{\alpha}{40} \left[ 3 \alpha^4 \arcsinh\left( \beta/\alpha \right) + \beta (5\alpha^2 + 2 \beta^2) \sqrt{ \alpha^2 + \beta^2} \right]$ \\ 
    $r^5$ & $\frac{\alpha}{336} \left[ 15 \alpha^6 \arcsinh\left( \beta/\alpha \right) + \beta (33\alpha^4 + 26\alpha^2\beta^2 + 8 \beta^4) \sqrt{ \alpha^2 + \beta^2} \right]$ \\ 
    $r^7$ & $\frac{\alpha}{3346} \left[ 105 \alpha^8 \arcsinh\left( \beta/\alpha \right) + \beta (279\alpha^6 + 326\alpha^4\beta^2 + 200\alpha^2\beta^4 + 48 \beta^6) \sqrt{ \alpha^2 + \beta^2} \right]$ \\ 
  \end{tabular} 
  \caption{The reference integral $I_{\text{ref}}(\alpha,\beta)$---see \eqref{eq:refInt}---for some PHS}
  \label{tab:moments}
\end{table} 

We note that the approach presented above is similar to the one in \cite{sommariva2006numerical}, where the domain $\Omega = [-1,1]^2$ was considered. 
Later, the same authors extended their findings to simple polygons \cite{sommariva2006meshless} using the Gauss--Grenn theorem.  
Also see the recent work \cite{sommariva2021rbf}, addressing polygonal regions that may be nonconvex or even multiply connected, and references therein. 
It would be of interest to see if these approaches also carry over to computing products of RBFs corresponding to different centers or products of RBFs and their partial derivatives, again corresponding to different centers. 
Such integrals occur as elements of mass and stiffness matrices in numerical PDEs. 
In particular, they are desired to construct linearly energy stable (global) RBF methods for hyperbolic conservation laws \cite{glaubitz2020shock,glaubitz2021stabilizing,glaubitz2021towards}.  

\bibliographystyle{siamplain}
\bibliography{literature}

\end{document}